\documentclass{amsart}

\usepackage[english]{babel} 
\usepackage[utf8]{inputenc} 
\usepackage{amsmath} 
\usepackage{amsthm}
\usepackage{amssymb}
\usepackage{mathrsfs}
\usepackage{amsfonts} %
\usepackage{bm} 
\usepackage{fancyvrb}
\usepackage[linesnumbered,algoruled,boxed]{algorithm2e}
\usepackage{graphicx}
\usepackage{psfrag}
\usepackage{subfig}
\usepackage[square,sort,comma,numbers]{natbib} 
\usepackage[colorlinks=false, linkcolor=blue,  citecolor=OliveGreen,  pdfstartview={}{}]{hyperref} 
\usepackage{lineno}
\usepackage{stmaryrd}
\usepackage[font=scriptsize]{caption}
\usepackage{enumerate}
\usepackage{pgf,tikz,pgfplots}

\usepackage{pstricks-add}

\DeclareFontEncoding{FMS}{}{}
\DeclareFontSubstitution{FMS}{futm}{m}{n}
\DeclareFontEncoding{FMX}{}{}
\DeclareFontSubstitution{FMX}{futm}{m}{n}
\DeclareSymbolFont{fouriersymbols}{FMS}{futm}{m}{n}
\DeclareSymbolFont{fourierlargesymbols}{FMX}{futm}{m}{n}
\DeclareMathDelimiter{\VERT}{\mathord}{fouriersymbols}{152}{fourierlargesymbols}{147}

\newtheorem{theorem}{Theorem}[section] 
\newtheorem{lemma}[theorem]{Lemma}
\newtheorem{definition}[theorem]{Definition} 
\newtheorem{proposition}{Proposition}[section] 
\theoremstyle{remark}
\numberwithin{equation}{section}
\newcommand{\T}{\mathscr{T}}
\newcommand{\M}{\mathscr{M}}
\newcommand{\Sides}{\mathscr{S}}

\newcommand{\ff}{\mathbf{f}}
\newcommand{\yy}{\mathbf{y}}
\newcommand{\uu}{\mathbf{u}}
\newcommand{\vv}{\mathbf{v}}
\newcommand{\ww}{\mathbf{w}}
\newcommand{\zz}{\mathbf{z}}

\newcommand{\LL}{\mathbf{L}}
\newcommand{\HH}{\mathbf{H}}
\newcommand{\RR}{\mathbf{R}}
\newcommand{\WW}{\mathbf{W}}

\usepackage{pgf,tikz}
\usetikzlibrary{arrows}

\begin{document}

\title[\resizebox{4.5in}{!}{AFEMs for the pointwise tracking control problem of the Stokes equations}]{Adaptive finite element methods for the pointwise tracking optimal control problem of the Stokes equations}
\thanks{AA is partially supported by CONICYT through FONDECYT project 1170579. EO is partially supported by CONICYT through FONDECYT project 3160201. DQ is supported by USM through Programa de Incentivos a la Iniciaci\'on Cient\'ifica (PIIC)}

\author[A. Allendes]{Alejandro Allendes\textsuperscript{\textsection}}
\address{\textsection Departamento de Matem\'atica, Universidad T\'ecnica Federico Santa Mar\'ia, Valpara\'iso, Chile \texttt{alejandro.allendes@usm.cl}}
\author[F. Fuica]{Francisco Fuica\textsuperscript{\textdagger}}
\address{\textdagger Departamento de Matem\'atica, Universidad T\'ecnica Federico Santa Mar\'ia, Valpara\'iso, Chile.
\texttt{francisco.fuica@sansano.usm.cl}}
\author[E. Ot\'arola]{Enrique Ot\'arola\textsuperscript{\textparagraph}}
\address{\textparagraph Departamento de Matem\'atica, Universidad T\'ecnica Federico Santa Mar\'ia, Valpara\'iso, Chile. \texttt{enrique.otarola@usm.cl}}
\author[D. Quero]{Daniel Quero\textsuperscript{\textdaggerdbl}}\address{\textdaggerdbl Departamento de Matem\'atica, Universidad T\'ecnica Federico Santa Mar\'ia, Valpara\'iso, Chile.
\texttt{daniel.quero@alumnos.usm.cl}}

\begin{abstract}
We propose and analyze a reliable and efficient a posteriori error estimator for the pointwise tracking optimal control problem of the Stokes equations. This linear--quadratic optimal control problem entails the minimization of a cost functional that involves point evaluations of the velocity field that solves the state equations. This leads to an adjoint problem with a linear combination of Dirac measures as a forcing term and whose solution exhibits reduced regularity properties. We also consider constraints on the control variable. The proposed a posteriori error estimator can be decomposed as the sum of four contributions: three contributions related to the discretization of the state and adjoint equations, and another contribution that accounts for the discretization of the control variable. On the basis of the devised a posteriori error estimator, we design a simple adaptive strategy that illustrates our theory and exhibits a competitive performance.
\end{abstract}
\maketitle


\section{Introduction.}\label{Intro}

In this work we shall be interested in the design and analysis of an a posteriori error estimator for the pointwise tracking optimal control problem of the Stokes equations; control--constraints are also considered. To make matters precise, for $d\in\{2,3\}$, we let $\Omega\subset\mathbb{R}^{d}$ be an open and bounded polytopal domain with Lipschitz boundary $\partial\Omega$ and $\mathcal{D}$ be a finite ordered subset of $\Omega$ with cardinality $\#\mathcal{D} = m$. Given a set of desired states $\{{\bf{y}}_t\}_{t\in \mathcal{D}}$, a regularization parameter $\lambda>0$, and the cost functional 
\begin{equation}\label{def:cost_func}
J(\yy,\uu):=\frac{1}{2}\sum_{t\in \mathcal{D}}|\yy(t)-\yy_t|^{2}+\frac{\lambda}{2}\|\uu\|_{{\LL}^2(\Omega)}^2,
\end{equation} 
our problem reads as follows: Find $\min J(\yy,\uu)$ subject to the Stokes equations
\begin{equation}\label{def:state_eq}
\left\{
\begin{array}{rcll}
-\Delta \yy + \nabla p & = & \uu & \text{ in } \quad \Omega, \\
\text{div }\yy & = & 0 & \text{ in } \quad \Omega, \\
\yy & = & \mathbf{0} & \text{ on } \quad \partial\Omega,
\end{array}
\right.
\end{equation}
and the control constraints
\begin{equation}\label{def:box_constraints}
\uu\in \mathbb{U}_{ad},\quad \mathbb{U}_{ad}:=\{\mathbf{v} \in \LL^2(\Omega):  \mathbf{a} \leq \mathbf{v} \leq \mathbf{b} \text{ a.e. in } \Omega \},
\end{equation}
with $\mathbf{a},\mathbf{b} \in \mathbb{R}^{d}$ satisfying $\mathbf{a} < \mathbf{b}$. We immediately comment that, throughout this work, vector inequalities must be understood componentwise. In \eqref{def:cost_func}, $|\cdot|$ denotes the euclidean norm.


In the literature, several numerical techniques for approximating the solution to optimal control problems have been proposed and analyzed. A particular emphasis has been given to optimal control problems that entail the minimization of a quadratic functional subject to a linear partial differential equation (PDE) and control/state constraints \cite{MR2516528,MR2843956,MR2441683,MR2122182,MR3308473,Troltzsch}. Recent works have shown that a particular class of solution techniques obtaining efficient approximation results are the so-called adaptive finite element methods (AFEMs). The power of these methods is specially observed when used for solving problems involving geometric singularities \citep{KRS} or/and singular sources \cite{Allendes_et_al2017_2}, for which our optimal control problem is a particular instance.

AFEMs are iterative feedback procedures that improve the quality of the finite element approximation to a PDE while striving to keep an optimal distribution of computational resources measured in terms of degrees of freedom. An essential ingredient of AFEMs is an posteriori error estimator, which is of importance in computational practice because of its ability to provide computable information about errors  and drive adaptive mesh refinement algorithms. The a posteriori error analysis for standard finite element approximations of linear second-order elliptic boundary value problems has a solid foundation \cite{MR1885308,MR2648380,MR3076038,Verfurth}. In contrast to the well-established theory for linear elliptic PDEs, the a posteriori error analysis for finite
element approximations of a constrained optimal control problem has not yet been fully understood.  In view of their inherent nonlinear feature, which appears due to the
control constraints, the analysis involves more arguments and technicalities \cite{HHIK,LiuYan,KRS}.
  
The pointwise tracking optimal control problem for the Poisson equation has been considered in a number of works     \citep{MR3449612,MR3523574,MR3800041,2018arXiv180202918B}. 
In \citep{MR3800041}, the authors operate under the framework of Muckenhoupt weighted Sobolev spaces \cite{Nochetto2016} and circumvent the difficulties associated with the underlying adjoint equation: a Poisson equation with a linear combination of Dirac deltas as a forcing term. Weighted Sobolev spaces allow for working under a Hilbert space-based framework in comparison to the non-Hilbertian setting of
\cite{MR3449612,MR3523574,2018arXiv180202918B}. An priori error analysis for a standard finite element approximation of the aforementioned problem can be found in \citep{MR3800041,MR3523574} while its a posteriori error analysis has been recently provided in \citep{MR3449612,Allendes_et_al2017_2}. In contrast to these advances and to the best of our knowledge, the pointwise tracking optimal control problem for the Stokes equations has not been considered before. In this work we will be concerned with the design and analysis of an a posteriori error estimator for the aforementioned problem. We immediately notice that, since the cost functional involves point evaluations of the velocity field that solves the state equations, the momentum equation of the adjoint equations reads as follows:
\begin{equation}
-\Delta  \zz - \nabla r = \sum_{ t \in \mathcal{D}} ({\yy} - \yy_{t})\delta_{t}.
\label{eq:adjoint_intro}
\end{equation}
Consequently, $\zz \notin \HH^1(\Omega)$ and $r \notin L^2(\Omega)/\mathbb{R}$ \cite[Section IV.2]{Gal11}. As it is observed in \cite{Allendes_et_al2017}, standard a posteriori error estimation techniques \cite{Verfurth} completely fails when solving \eqref{eq:adjoint_intro}.

In this work we propose an a posteriori error estimator for the pointwise tracking optimal control problem of the Stokes equations
that can be decomposed as the sum of four contributions: two related to the discretization of the state equations, one associated to the discretization of the adjoint equations and one that accounts for the discretization of the control variable. Since problem \eqref{eq:adjoint_intro} involves the point evaluations of the velocity field that solves the state equations, we consider, for such a variable, an a posteriori error estimator in maximum norm \cite{Larsson_Svensson} while a standard one is considered for the associated pressure \cite{Verfurth}. For the adjoint variables we consider the a posteriori error estimator in Muckenhoupt weighted Sobolev spaces of \cite{Allendes_et_al2017}. We obtain global reliability and local efficiency properties. On the basis of the devised a posteriori error estimator, we also design a simple adaptive strategy that exhibits optimal experimental rates of convergence for the state and adjoint variables.

We conclude by mentioning that several solution techniques have been designed and studied for linear--quadratic optimal control problems governed by the Stokes equations. We refer the reader to \citep{MR2472877,MR2237883,MR2263034,MR2107381,MR2497339,MR2883156,MR2912614,MR2803865,
MR3009516,MR3426550,MR3347461} and references therein.

The plan of the paper is as follows. In Section \ref{sec:notation_and_prel}, we introduce the notation and functional framework we shall work with. In Section \ref{sec:stokes_stand} we review some standard regularity results for the Stokes equations on Lipschitz polytopes. In Section \ref{sec:pointwise_aposteriori_stokes}, we recall one of the maximum norm a posteriori error estimators for the Stokes equations developed in \citep{Larsson_Svensson} and provide an efficiency analysis for it. Section  \ref{sec:stokes_deltas_aposteriori} is dedicated to review the a posteriori error analysis in Muckenhoupt weighted Sobolev spaces developed in \citep{Allendes_et_al2017}.
Section \ref{sec:pt_ocp} contains the description of the pointwise tracking optimal control problem: we derive existence and uniqueness results together with first--order optimality conditions. The core of our work is Section \ref{sec:ocp}, where we introduce a discrete scheme that approximates our problem, devise an a posteriori error estimator and show in Section \ref{sec:reliability} and \ref{sec:efficiency}, its global reliability and local efficiency, respectively. We conclude, in Section \ref{sec:numerical_ex}, with a series of numerical examples that illustrate and go beyond our theory.


\section{Notation and Preliminaries.}\label{sec:notation_and_prel} 
Let us fix notation and the setting in which we will operate.

\subsection{Notation}\label{sec:notation}
Throughout this work $d\in\{2,3\}$ and $\Omega\subset\mathbb{R}^d$ is an open and bounded polytopal domain with Lipschitz boundary $\partial\Omega$. If $\mathcal{X}$ and $\mathcal{Y}$ are normed vector spaces, we write $\mathcal{X}\hookrightarrow\mathcal{Y}$ to denote that $\mathcal{X}$ is continuously embedded in $\mathcal{Y}$. We denote by $\mathcal{X}'$ and $\|\cdot\|_\mathcal{X}$ the dual and the norm of $\mathcal{X}$, respectively. 

We shall use lower--case bold letters to denote vector-valued functions whereas upper-case bold letters are used to denote function spaces. For a bounded domain $G \subset \mathbb{R}^d$, if $X(G)$ corresponds to a function space over $G$, we shall denote $\mathbf{X}(G)=[X(G)]^{d}$. In particular, we denote $\mathbf{L}^2(G)=[L^2(G)]^d$, which is equipped with the following inner product and norm:
\begin{equation*}
(\ww,\vv)_{\LL^2(G)}=\int_G \ww\cdot \vv, 
\qquad
\|\vv\|_{\LL^2(G)}=(\vv,\vv)_{\LL^2(G)}^{\frac{1}{2}}\qquad \forall\: \ww,\vv\in \mathbf{L}^2(G).
\end{equation*}

Finally, the relation $a \lesssim b$ indicates that $a \leq C b$, with a positive constant that depends neither on $a$, $b$ nor the discretization parameter. The value of $C$ might change at each occurrence.


\subsection{Weighted Sobolev spaces}\label{sec:wse}

We begin this section by introducing an ingredient that will be fundamental for the analysis that we will perform, that of a weight. A weight is an almost everywhere positive function $\omega\in L^1_{\text{loc}}(\mathbb{R}^d)$. 

For a Borel set $G\subset\mathbb{R}^d$ and a weight $\omega$, we define 
\begin{equation}\label{eq:notacion_peso}
\omega(G)=\int_G \omega.
\end{equation}

A particular class of weights that will be of importance is the so-called Muckenhoupt class $A_2(\mathbb{R}^d)$ \citep{Javier2001,Fabes_et_al1982,MR0293384,Turesson2000}.

\begin{definition}[Muckenhoupt class $A_2(\mathbb{R}^d)$]
Let $\omega$ be a weight. We say that $\omega$ belongs to the Muckenhoupt class $A_{2}(\mathbb{R}^d)$ if there exists a positive constant $C_{\omega}$ such that
\begin{equation}\label{eq:A2weight_def}
C_{\omega} = \sup_{B}\left(\frac{1}{|B|}\int_B\omega \right)\left(\frac{1}{|B|}\int_B\omega^{-1} \right) < \infty,
\end{equation}
where the supremum is taken over all balls $B$ in $\mathbb{R}^{d}$, and $|B|$ corresponds to the measure of $B$. If $\omega$ belongs to $A_{2}(\mathbb{R}^d)$, we say that $\omega$ is an $A_{2}(\mathbb{R}^d)$--weight. 
\end{definition}

In the analysis that we will perform, the following example of a Muckenhoupt weight will be essential:
Let $x_{0}$ be an interior point of $\Omega$ and denote by $\mathrm{d}_{x_0}(x) =  |x - x_{0}|$
the Euclidean distance to $x_0$. It is well--known that $\mathsf{d}_{x_{0}}^{\alpha}(x) = \mathsf{d}_{x_{0}}(x)^{\alpha}$ belongs to the Muckenhoupt class $A_2(\mathbb{R}^d)$ if and only if $\alpha\in (-d,d)$.

We now define weighted Lebesgue and Sobolev spaces as follows. Let $\omega \in A_2(\mathbb{R}^d)$ and $G\subseteq\Omega$ be an open and bounded domain. We define the weighted Lebesgue space $L^2(\omega,G)$ as
\begin{equation*}
L^2(\omega,G):=\bigg\{v\in L^1_{\text{loc}}(G): \|v\|_{L^2(\omega,G)}:= \left(\int_{G} |v|^{2} \omega \mathrm{d}x\right)^\frac{1}{2} < \infty\bigg\}.
\end{equation*}
We also define the weighted Sobolev space
\begin{equation*}
H^1(\omega,G):=\left\{v\in L^2(\omega,G) : |\nabla v| \in L^2(\omega,G) \right\},
\end{equation*}
which we equip with the norm
\begin{equation}
\|v\|_{H^1(\omega,G)}:= \left(\|v\|_{L^2(\omega,G)}^2 + \|\nabla v\|_{L^2(\omega,G)}^2\right)^\frac{1}{2}.
\label{eq:weighted_norm}
\end{equation}
Since $\omega \in A_2(\mathbb{R}^d)$, the results \citep[Proposition 2.1.2, Corollary 2.1.6]{Turesson2000} and \citep[Theorem 1]{Goldshtein_Ukhlov_2009} allow us to conclude that $L^{2}(\omega,\Omega)$ and $H^1(\omega,\Omega)$ are Hilbert spaces. Moreover, the space $C^\infty(\Omega)$ is dense in $H^1(\omega,\Omega)$. We define $H^{1}_{0}(\omega,\Omega)$ as the closure of $C_{0}^{\infty}(\Omega)$ in $H^1(\omega,\Omega)$. Additionally, a Poincaré inequality holds for all $v \in H^{1}_{0}(\omega,G)$ \citep[Theorem 1.3]{Fabes_et_al1982}. Consequently, in $H^{1}_{0}(\omega,G)$ the seminorm $\|\nabla v\|_{L^2(\omega,\Omega)}$ is equivalent to \eqref{eq:weighted_norm}. 

Finally, on the basis of the previously introduced spaces, we define the vector space $\HH^{1}_{0}(\omega,G):=[H^{1}_{0}(\omega,G)]^d$, which we equip with the norm
\begin{equation}\label{def:H_0^1-norm}
\|\nabla \vv \|_{\LL^{2}(\omega,G)} = \left(\sum_{i=1}^{d}\|\nabla \vv_{i}\|_{L^{2}(\omega,G)}^2 \right)^{\frac{1}{2}}.
\end{equation}


\section{Pointwise a posteriori error estimation for the Stokes equations}\label{sec:pointwise_stokes}
In Section \ref{sec:ocp} we will design and analyze an a posteriori error estimator for the pointwise tracking optimal control problem of the Stokes equations, i.e., problem \eqref{def:cost_func}--\eqref{def:box_constraints}. The error estimator involves several contributions related to the discretization of the state and adjoint equations and the control variable. Since the cost functional of the aforementioned optimal control problem involves pointwise evaluations of the velocity field that solves the state equations, it is thus imperative to consider a pointwise a posteriori error estimator for such equations. In an effort to make this contribution self contained, in this section we briefly review a posteriori error estimates in the maximum norm as well as standard results concerning regularity properties of the solution to the Stokes equations. We provide an efficiency analysis for one of the estimators proposed in \cite{Larsson_Svensson}.

\subsection{The Stokes problem in Lipschitz polytopes}\label{sec:stokes_stand}
Throughout this section, $\Omega$ denotes an open and bounded polytopal domain. Unless specified otherwise, we will assume that $\partial \Omega$ is Lipschitz. We introduce the bilinear forms
\begin{equation}\label{def:bilinear_a}
a: \HH_0^1(\Omega) \times \HH_0^1(\Omega) \rightarrow \mathbb{R}, 
\quad 
a(\ww,\vv):=\int_\Omega \nabla \ww : \nabla \vv =\sum_{i=1}^d \int_\Omega \nabla \ww_i \cdot \nabla \vv_i
\end{equation}	
for all $\vv, \ww \in  \HH_0^1(\Omega)$, and
\begin{equation}\label{def:bilinear_b}
b:\HH_0^1(\Omega) \times L^2(\Omega)/\mathbb{R} \rightarrow \mathbb{R}, 
\quad 
b(\vv,q):=-\int_\Omega q \:\text{div}\: \vv
\end{equation}
for all $\vv\in  \HH_0^1(\Omega)$ and $q\in L^2(\Omega)/\mathbb{R}$.

Given $\ff \in \HH^{-1}(\Omega)$, we consider the following weak version of the Stokes equations: Find $(\yy,p) \in \HH_{0}^{1}(\Omega) \times L^{2}(\Omega)/\mathbb{R}$ such that
\begin{equation}\label{eq:weak_stokes_problem}
\left\{
\begin{array}{rcll}
a(\yy,\vv) + b(\vv,p) & = & \langle \textbf{f},\vv \rangle_{\HH^{-1}(\Omega),\HH_{0}^{1}(\Omega)} & \forall\: \vv \in \HH_{0}^{1}(\Omega), \\
b(\yy,q) & = & 0 & \forall\: q \in L^{2}(\Omega)/\mathbb{R}.
\end{array}
\right.
\end{equation}

The following result state a global higher integrability result for the solution $(\yy,p)$ and, as consequence, a H\"older regularity estimate for the velocity field $\yy$; see \citep[Theorem 2.9]{Brown_Shen}, \citep[Theorem 1.1]{Larsson_Svensson} and \citep[Lemma 12]{demlar13}. 

\begin{theorem}[higher integrability]\label{thm:high_int}
There exists $\varepsilon > 0$ such that if $(3+\varepsilon)/(2+\varepsilon) < l < 3 + \varepsilon$, $\ff \in \WW^{-1,l}(\Omega)$, then there is a unique weak solution $(\yy,p) \in \WW^{1,l}(\Omega) \times L^{l}(\Omega)/\mathbb{R}$ to \eqref{eq:weak_stokes_problem}. In addition, we have that
\begin{equation}\label{eq:stig_lars_stab}
\|\yy\|_{\WW^{1,l}(\Omega)} + \|p\|_{L^{l}(\Omega)/\mathbb{R}} \lesssim \|\ff\|_{\WW^{-1,l}(\Omega)},
\end{equation}
where the hidden constant is independent of $\yy, p$ and $\ff$. This, in particular, implies that for $\kappa = 1 - d/l > 0$ we have $\yy \in \mathbf{C}^{0,\kappa}(\bar{\Omega})$ with a similar estimate.
\end{theorem}

The following result guarantees that, whenever $\ff \in \LL^{l}(\Omega)$ with $l \in [1,\infty)$, we have a local regularity result for the solution of the Stokes equations \eqref{eq:weak_stokes_problem}. We refer the reader to \cite[Theorem IV.4.1]{Gal11} for a proof of this result.

\begin{theorem}[local regularity]\label{prop:local_reg_FF2}
Let $(\uu,p)\in \HH_0^1(\Omega)\times L^2(\Omega)/\mathbb{R}$ be the unique solution to \eqref{eq:weak_stokes_problem} with $\mathbf{f}\in \LL^l(\Omega)$ and $l \in [1,\infty)$. If $D \Subset \Omega$, then $(\yy,p)\in \mathbf{W}^{2,l}(D) \times W^{1,l}(D)$. In addition, we have that
\begin{equation}\label{eq:estim_localreg}
\|\yy\|_{\WW^{2,l}(D)} + \|p\|_{W^{1,l}(D)} \lesssim \|\ff\|_{\LL^{l}(\Omega)},
\end{equation}
where the hidden constant depends on $\mathrm{dist}(D,\partial \Omega)$ and $\Omega$ but is independent of $\yy$, $p$, and $\ff$.
\end{theorem}

We now provide a local and weighted integrability result for $\yy$ and $p$.

\begin{proposition}[weighted integrability]\label{prop:weighted_int}
Let $(\yy,p) \in \HH_{0}^{1}(\Omega) \times L^{2}(\Omega)/\mathbb{R}$ denote the solution of \eqref{eq:weak_stokes_problem} with $\ff \in \LL^{l}(\Omega)$ and $l > d$. Let $x \in \Omega$, $\delta < \mathrm{dist}(x,\Omega)$ and $B$ denote the ball of radius $\delta$ and center $x$. If $\omega \in A_2(\mathbb{R}^d)$, then we have that $(\yy,p) \in \HH_0^{1}(\omega,B) \times L^{2}(\omega,B)/\mathbb{R}$. In addition, we have the estimate
\begin{equation}\label{eq:weighted_ineq}
\|\nabla \yy \|_{\LL^{2}(\omega,B)} + \|p \|_{L^{2}(\omega,B)} \lesssim \|\ff\|_{\LL^{l}(\Omega)},
\end{equation}
where the hidden constant depends on $\omega(B)$, $\mathrm{dist}(B,\partial \Omega)$, $\delta$, and $\Omega$, but is independent of $\yy$, $p$, and $\ff$.
\end{proposition}
\begin{proof}
Since $l > d$, the following embedding holds: $\mathbf{W}^{1,l}(B) \hookrightarrow \LL^{\infty}(B)$. This, combined with the fact that $\mathrm{dist}(B,\partial \Omega)>0$, and the estimate \eqref{eq:estim_localreg} of Theorem \ref{prop:local_reg_FF2} reveal that
\[
\|\nabla \yy\|_{\LL^{\infty}(B)} + \|p\|_{L^{\infty}(B)} \lesssim \|\yy\|_{\WW^{2,l}(B)} + \|p\|_{W^{1,l}(B)} \lesssim \|\ff\|_{\LL^{l}(\Omega)} .
\]
This estimate immediately implies that
\begin{equation*}
\|\nabla \yy \|_{\LL^2(\omega,B)} + \|p \|_{L^{2}(\omega,B)} \leq \omega(B)^{\frac{1}{2}} \left(\|\nabla \yy\|_{\LL^{\infty}(B)} + \|p\|_{L^{\infty}(B)} \right) 
\lesssim \|\ff\|_{\LL^{l}(\Omega)}.
\end{equation*}
\end{proof}

We conclude with the following regularity result; see \citep[Lemma 14]{demlar13}.

\begin{proposition}[higher differentiability]\label{prop:pressure_max}
If $\Omega$ is convex, then for $l>d$, sufficiently close to $d$ (depending on the maximum edge opening angle of $\Omega$), and $\kappa=1-d/l>0$, we have the following estimate
\begin{equation}\label{eq:preas_max}
\|\yy\|_{\mathbf{C}^{1,\kappa}(\bar{\Omega})} + \|p\|_{C^{0,\kappa}(\bar{\Omega})} \lesssim \|\ff\|_{\LL^{l}(\Omega)}.
\end{equation}
\end{proposition}

\subsection{Pointwise a posteriori error estimates}\label{sec:pointwise_aposteriori_stokes}

In this section we briefly present one of the pointwise a posteriori error estimators introduced and analyzed in \cite{Larsson_Svensson}. To accomplish this task, we assume that $\ff \in \LL^{\infty}(\Omega)$.

Let us start the discussion by introducing some standard finite element notation \cite{brenner,CiarletBook,Guermond-Ern}. Let $ \mathscr{T} = \{T\}$ be a conforming partition of $\bar \Omega$ into simplices $T$ with size $h_T := \textrm{diam}(T)$.
We denote by $\mathbb{T}$ the collection of conforming and shape regular meshes that are refinements of an initial mesh $\mathscr{T}_0$.

We define $\Sides$ as the set of internal $(d-1)$--dimensional interelement boundaries $S $ of $\T$. For $T \in \T$, let $\Sides^{}_T$ denote the subset of $\Sides$ that contains the sides in $\Sides$ which are sides of $T$. We also denote by $\mathcal{N}_{S}$ the subset of $\T$ that contains the two elements that have $S$
as a side. In addition, we define the \emph{stars} or \emph{patches} associated with an element $T \in \T$ as
\begin{equation}
\label{eq:patch}
\mathcal{N}_T= \bigcup_{T' \in \T: \Sides_T \cap \Sides_{T'} \neq \emptyset} T',
\end{equation}
and
\begin{equation}
\label{eq:patch_morin}
\mathcal{N}_T^*= \bigcup_{T' \in \T: T \cap {T'} \neq \emptyset} T'.
\end{equation}

For a discrete tensor valued function $\mathbf{V}_\T$, we denote by $[\![\mathbf{V}_\T \cdot \boldsymbol \nu]\!]$ the jump or interelement residual, which is defined, on the internal side $S\in \mathscr{S}$ shared by the distinct elements $T^+,T^-\in\mathcal{N}_S$, by
\begin{equation}\label{def:jump}
[\![\mathbf{V}_\T\cdot \boldsymbol\nu]\!]=\mathbf{V}_\T|_{T^+}\cdot\boldsymbol\nu^{+}+\mathbf{V}_\T|_{T^-}\cdot\boldsymbol\nu^{-}.
\end{equation}
Here, $\boldsymbol\nu^+,\boldsymbol\nu^-$ are unit normals on $S$ pointing towards $T^+,T^-$, respectively.

Given a mesh $\mathscr{T} \in \mathbb{T}$, we denote by $\mathbf{V}(\T)$ and $Q(\T)$ the finite element spaces
that approximate the velocity field and the pressure, respectively, based on the classical Taylor Hood elements \cite[Section 4.2.5]{Guermond-Ern}:
\begin{align}\label{def:discrete_spaces}
\begin{split}
Q(\T):&=\left\{ q_{\T} \in C(\bar{\Omega})\ : \ q_{\T}|^{}_T \in \mathbb{P}_{1}(T) \: \forall \: T \in\T \right\}  \cap L^2(\Omega)/\mathbb{R}, \\
\mathbf{V}(\T):&=\left\{ \vv_{\T} \in \mathbf{C}(\bar{\Omega})\ : \ \vv_{\T}|^{}_T \in \mathbb{P}_2(T)^d \: \forall \: T\in\T \right\}\cap \HH_0^1(\Omega).
\end{split}
\end{align}

With these spaces at hand, we define the Galerkin approximation to \eqref{eq:weak_stokes_problem} as the solution to the following problem: Find $(\yy_\T,p_\T) \in \mathbf{V}(\T)\times Q(\T)$ that solves
\begin{equation}\label{eq:discrete_stokes_problem}
\left\{
\begin{array}{rcll}
a(\yy_\T,\vv_\T) + b(\vv_\T,p_\T) & = & (\textbf{f},\vv_\T)_{\LL^{2}(\Omega)} & \forall \: \vv_\T \in \mathbf{V}(\T), \\
b(\yy_\T,q_\T) & = & 0 & \forall \: q_\T \in Q(\T). 
\end{array}
\right.
\end{equation}

On the basis of the previous definitions, we introduce the pointwise a posteriori error estimator $\mathcal{E}_\infty$ as follows:
\begin{equation}\label{def:estimator}
\mathcal{E}_{\infty}(\yy_{\T},p_{\T},\ff):=\max_{T\in\T}\mathcal{E}_{\infty,T}(\yy_{\T},p_{\T},\ff),
\end{equation}
where, for every $T\in \T$, the local a posteriori error indicators $\mathcal{E}_{\infty,T}$ are given by
\begin{multline}
\label{def:indicators_2}
\mathcal{E}_{\infty,T}(\yy_{\T},p_{\T},\ff)
:= 
h_{T}^2\|\ff+\Delta \yy_\T-\nabla p_\T\|_{\LL^\infty(T)}
\\+ \tfrac{h_T}{2}\|[\![\nabla \yy_\T \cdot \boldsymbol\nu]\!]\|_{\LL^\infty(\partial T\setminus \partial \Omega)} 
+ h_T\|\text{div }\yy_\T\|_{L^\infty(T)}.
\end{multline}

\subsubsection{Reliability}

In order to present the global reliability of the a posteriori error estimator $\mathcal{E}_{\infty}$, and for future reference, we introduce
\begin{equation}\label{def:l_T}
\ell_\T:=\left|\log\left(\max_{T\in\T}\frac{1}{h_T}\right)\right|.
\end{equation}

With all these ingredients at hand, we present the following result; see \citep[Theorem 4.1]{Larsson_Svensson} and \citep[Lemma 3]{demlar13}.

\begin{theorem}[global reliability of $\mathcal{E}_{\infty}$]\label{thm:reliability_max}
Let $(\yy,p)\in \HH_0^1(\Omega)\times L^2(\Omega)/\mathbb{R}$ be the solution to the Stokes equations \eqref{eq:weak_stokes_problem} and $(\yy_\T,p_\T)\in\mathbf{V}(\T)\times Q(\T)$ its numerical approximation obtained as the solution to \eqref{eq:discrete_stokes_problem}. 
Then,
\begin{equation}\label{eq:bound_vel}
\|\yy-\yy_\T\|_{\LL^\infty(\Omega)}
\lesssim
\ell_\T^{\beta_d}\mathcal{E}_{\infty},
\end{equation}
where $\beta_2=2$ and $\beta_3=4/3$.
\end{theorem}

\subsubsection{Efficiency}

We now proceed to investigate the local efficiency properties of the local error indicator $\mathcal{E}_{\infty,T}$ defined in \eqref{def:indicators_2}. To accomplish this task, we define, for $\mathscr{M}\subset \T$, $\mathfrak{m}\in\{2,\infty\}$, and $\mathbf{g}\in \LL^{\mathfrak{m}}(\Omega)$,
\begin{equation}\label{eq:def_osc}
\textrm{osc}^{}_{\mathfrak{m}}(\mathbf{g};\M):= 
\begin{cases}
\displaystyle \left( \sum_{T\in\mathscr{M}}h_T^{2} \| \mathbf{g} - \Pi_{\T} (\mathbf{g}) \|_{\LL^{2}(T)}^2 \right)^{\frac{1}{2}},\quad &\text{ if } \mathfrak{m}=2,\\
\displaystyle\max_{T \in \M}
 h_T^{2} \| \mathbf{g} - \Pi_{\T} (\mathbf{g}) \|_{\LL^{\infty}(T)},\quad &\text{ if } \mathfrak{m}=\infty,
\end{cases}
\end{equation}
where, $\Pi_{\T}$ denotes the $L^{2}$--orthogonal projection operator onto piecewise linear functions over $\T$.

For an edge, triangle or tetrahedron $G$, we denote by $\mathscr{V}(G)$ the set of vertices of $G$. We introduce, for $T \in \T$, the standard bubble function \cite[Section 2.3.1]{MR1885308}
\begin{equation}\label{def:bubbleT}
\varphi_{T} = (d+1)^{d+1}\prod_{\textsc{v} \in \mathscr{V}(T)}\phi_{\textsc{v}}|^{}_{T}, 
\end{equation}
where $\phi_{\textsc{v}}$ are the barycentric coordinates of $T$.
The function $\varphi_{T}$ satisfies the following properties:
\begin{equation}\label{bubbleT_properties}
|T|\lesssim \int_T \varphi_T,\quad
\text{supp }\varphi_T=T,
\quad
\|\nabla^k \varphi_T\|_{L^{2}(T)}
\lesssim
h_T^{\frac{d}{2}-k},
\quad
k=0,1,2.
\end{equation}


The local efficiency of the indicator \eqref{def:indicators_2} is as follows.

\begin{theorem}[local efficiency of $\mathcal{E}_{\infty,T}$]\label{thm:local_efficiency_maximum}
Let $(\yy,p)\in \HH_0^1(\Omega)\times L^2(\Omega)/\mathbb{R}$ be the solution to the Stokes equations \eqref{eq:weak_stokes_problem} and $(\yy_\T,p_\T)\in\mathbf{V}(\T)\times Q(\T)$ its numerical approximation obtained as the solution to \eqref{eq:discrete_stokes_problem}. If $\Omega$ is convex and $\ff\in \LL^\infty(\Omega)$, then, for $T\in \T$, the local error indicators $\mathcal{E}_{\infty,T}$ defined as in \eqref{def:indicators_2}, satisfy that
\begin{equation}\label{eq:efficiency_estimate}
\mathcal{E}_{\infty,T}(\yy_{\T},p_{\T})
\lesssim
\|\yy-\yy_\T\|_{\LL^\infty(\mathcal{N}_T)}  + h_{T} \|p-p_\T\|_{L^\infty(\mathcal{N}_T)} +\mathrm{osc}^{}_{\infty}(\ff;\mathcal{N}_T),
\end{equation}
where $\mathcal{N}_T$ is defined as in \eqref{eq:patch}. The hidden constant is independent of the continuous and discrete solutions, the size of the elements in the mesh $\T$ and $\#\T$.
\end{theorem}
\begin{proof} We proceed in four steps.

\emph{Step 1.} To simplify the presentation of the material, we define $\mathbf{e}_\yy=\yy-\yy_\T$ and $e_p=p-p_\T$. 

Let us consider $\vv \in \HH_0^1(\Omega)$ which is such that $\vv|_{T} \in \mathbf{C}^2(T)$ for all $T \in \T$. We first invoke the fact that $(\yy,p)$ solves \eqref{eq:weak_stokes_problem} to arrive at
\begin{equation}\label{eq:error_eq_1}
a(\mathbf{e}_\yy,\vv) + b(\vv,e_p) 
= (\ff,\vv)_{\LL^{2}(\Omega)} - (\nabla \yy_{\T},\nabla \vv)_{\LL^2(\Omega)} + (p_{\T}, \text{div }  \vv)_{L^{2}(\Omega)}.
\end{equation}
Second, an integration by parts formula allow us to conclude that
\begin{multline}\label{effic_lhs}
a(\mathbf{e}_\yy,\vv) + b(\vv,e_p)   \\ 
= 
-\sum_{T \in \T}\bigg((\Delta \vv,\mathbf{e}_\yy)_{\LL^{2}(T)} + (e_p,\text{div} \ \vv)_{L^{2}(T)}\bigg) - \sum_{S \in \Sides} (\mathbf{e}_\yy,[\![\nabla \vv\cdot \boldsymbol \nu]\!])_{\LL^{2}(S)}.
\end{multline}
Third, we use again, an integration by parts formula, to arrive at
\begin{multline}\label{effic_rhs}
(\ff,\vv)_{\LL^2(\Omega)} - (\nabla \yy_{\T},\nabla \vv)_{\LL^2(\Omega)} + (p_{\T}, \ \text{div }  \vv)_{L^{2}(\Omega)}  = \\
=
\sum_{T \in \T} \bigg((\Pi_{\T}(\ff) + \Delta \yy_{\T} - \nabla p_{\T},\vv)_{\LL^{2}(T)} + (\ff - \Pi_{\T}(\ff),\vv)_{\LL^{2}(T)}\bigg)\\
 + \sum_{S \in \Sides}\bigg(([\![\nabla \yy_\T \cdot \boldsymbol\nu]\!],\vv)_{\LL^{2}(S)} \bigg).
\end{multline}
Notice that we have used that $p_{\T} \in Q(\T)$, which implies that $[\![ p_\T ]\!] = 0$. Consequently, \eqref{eq:error_eq_1}, \eqref{effic_lhs}, and \eqref{effic_rhs},  allow us to conclude the following identity
\begin{multline}\label{effic_chs}
-\sum_{T \in \T}\left((\Delta \vv,\mathbf{e}_\yy)_{\LL^{2}(T)} + (e_p,\text{div} \ \vv)_{L^{2}(T)}\right) - \sum_{S \in \Sides} (\mathbf{e}_\yy,[\![\nabla \vv\cdot \boldsymbol \nu]\!])_{\LL^{2}(S)}  \\
=
\sum_{T \in \T} \left((\Pi_{\T}(\ff) + \Delta \yy_{\T} - \nabla p_{\T},\vv)_{\LL^{2}(T)} + (\ff - \Pi_{\T}(\ff),\vv)_{\LL^{2}(T)}\right)\\ + \sum_{S \in \Sides}([\![\nabla \yy_\T \cdot \boldsymbol\nu]\!],\vv)_{\LL^{2}(S)}.
\end{multline}

\emph{Step 2.} Let $T \in \T$. We estimate the term $h_T^2\|\ff + \Delta \yy_{\T} - \nabla p_{\T}\|_{\LL^\infty(T)}$ in \eqref{def:indicators_2}. To accomplish this task, we first invoke the triangle inequality and obtain that
\[
h_T^2\|\ff + \Delta \yy_{\T} - \nabla p_{\T}\|_{\LL^\infty(T)}
\leq
h_T^2\|\Pi_{\T}(\ff) + \Delta \yy_{\T} - \nabla p_{\T}\|_{\LL^\infty(T)} + h_T^2\|\ff-\Pi_{\T}(\ff)\|_{\LL^\infty(T)}.
\]
To simplify the presentation of the material, we define $\RR_{T} := (\Pi_{\T}(\ff) + \Delta \yy_{\T} - \nabla p_{\T})|_T$. It thus suffices to bound $h_T^2\|\RR_{T}\|_{\LL^\infty(T)}$. To derive such a bound, we set $\vv=\varphi^{2}_{T}\RR_{T}$ in \eqref{effic_chs} and invoke properties of the function $\varphi_T$. This yields
\begin{multline}
\label{estimate_22}
\|\RR_{T}\|_{\LL^{2}(T)}^{2} 
\lesssim
\|\ff - \Pi_{\T}(\ff)\|_{\LL^{\infty}(T)}\|\varphi^{2}_{T}\RR_{T}\|_{\LL^{1}(T)}
\\
+ \|\mathbf{e}_\yy\|_{\LL^\infty(T)}\|\Delta(\varphi^{2}_{T}\RR_{T})\|_{\LL^{1}(T)}
+\|e_p\|_{L^\infty(T)}\|\text{div}(\varphi^{2}_{T}\RR_{T})\|_{L^{1}(T)},
\end{multline}
where we have used that, for $S \in \Sides_{T}$,  $\int_S \mathbf{e}_\yy [\![\nabla (\varphi^{2}_{T}\RR_{T})\cdot \boldsymbol \nu]\!] = 0$.

On the other hand, standard computations reveal that 
\[\Delta(\varphi^{2}_{T}\RR_{T}) = 2\RR_{T}(\varphi_{T}\Delta \varphi_{T} + |\nabla \varphi_{T}|^{2}) + 4\varphi_{T}\nabla \varphi_{T} \nabla \RR_{T} + \varphi_{T}^{2}\Delta \RR_{T}.\]
This, in conjunction with the properties that $\varphi_{T}$ satisfies, stated in \eqref{bubbleT_properties}, and the inverse estimates of \cite[Lemma 4.5.3]{brenner}, imply that
\begin{equation}\label{eq:estim1_R}
\|\Delta(\varphi^{2}_{T}\RR_{T})\|_{\LL^{1}(T)} \lesssim h_{T}^{\frac{d}{2}-2}\|\RR_{T}\|_{\LL^{2}(T)}.
\end{equation}
Similar arguments to the ones that yield \eqref{eq:estim1_R} allow us to derive
\begin{equation}
\|\text{div}(\varphi^{2}_{T}\RR_{T})\|_{L^{1}(T)} \lesssim h^{\frac{d}{2} - 1}_{T}\|\RR_{T}\|_{\LL^{2}(T)},
\quad
\| \varphi^{2}_{T}\RR_{T}\|_{\LL^{1}(T)} \lesssim h^{\frac{d}{2}}_{T}\|\RR_{T}\|_{\LL^{2}(T)}.
\label{eq:estim3_R}
\end{equation}
We thus replace the estimates \eqref{eq:estim1_R}--\eqref{eq:estim3_R} into \eqref{estimate_22} to arrive at
\begin{equation*}
h_T^{2}\|\RR_{T}\|_{\LL^2(T)} 
\lesssim
h_{T}^{\frac{d}{2}+2}\|\ff - \Pi_{\T}(\ff)\|_{\LL^{\infty}(T)}
+
h_{T}^{\frac{d}{2}}\|\mathbf{e}_\yy\|_{\LL^\infty(T)}
+
h_{T}^{\frac{d}{2}+1}\|e_p\|_{L^\infty(T)}.
\end{equation*}
The inverse estimate $\|\RR_{T}\|_{\LL^{\infty}(T)} \lesssim h_T^{-\frac{d}{2}}\|\RR_{T}\|_{\LL^2(T)} $ allows us to conclude.

\emph{Step 3.} Let $T\in \T$ and $S\in \mathscr{S}_T$. We proceed to bound the jump term $\tfrac{h_T}{2}\|[\![\nabla \yy_\T\cdot \boldsymbol\nu]\!]\|_{\LL^\infty(\partial T \setminus \partial\Omega)}$ in \eqref{def:indicators_2}. We begin by invoking standard arguments to conclude the existence 
of an edge bubble function $\varphi_S\in \mathbb{P}_{(14d - 19)}(\mathcal{N}_S)$, 
such that satisfies the following properties
\begin{align}\label{eq:properties_bubbleS}
\varphi_{S} = 0 \text{ on } \partial \mathcal{N}_{S}, \quad \nabla \varphi_{S} = \mathbf{0}  \text{ on } \partial \mathcal{N}_{S}, \quad  [\![\nabla \varphi_{S} \cdot \boldsymbol{\nu} ]\!] = 0  \text{ on } S,
\end{align} 
and
\begin{equation}
\label{edge_estimate0}
|S| \|[\![\nabla \yy_\T\cdot \boldsymbol\nu]\!]\|_{\LL^{\infty}(S)} \lesssim \int_{S} [\![\nabla \yy_\T\cdot \boldsymbol\nu]\!]\boldsymbol\varphi_{S},
\end{equation}
where the vector--valued bubble function $\boldsymbol{\varphi}_{S}$ is given by
\[\boldsymbol \varphi_S:= \left\{\begin{array}{ll} 
(2\varphi_S,\varphi_S)^{T}, & d = 2, \\
(9\varphi_S,\varphi_S,\varphi_S)^{T}, & d = 3.
\end{array}
\right.
\]
We have assumed, without loss of generality, that $\|[\![\nabla \yy_\T\cdot \boldsymbol\nu]\!]\|_{\LL^\infty(S)}=([\![\nabla \yy_\T\cdot \boldsymbol\nu]\!])_1(\textsc{v}) > 0$, with $\textsc{v}\in \mathscr{V}(S)$. Now, we set $\vv=\boldsymbol{\varphi}_S$ in \eqref{effic_chs} and use \eqref{eq:properties_bubbleS} to conclude that
\begin{multline}\label{edge_estimate} 
\int_{S} [\![\nabla \yy_\T\cdot \boldsymbol\nu]\!]\boldsymbol\varphi_{S} 
\lesssim
\sum_{T' \in \mathcal{N}_S}\bigg( \|\mathbf{e}_\yy\|_{\LL^\infty(T')}\|\Delta \boldsymbol\varphi_S\|_{\LL^{1}(T')}\\
+\|e_p\|_{L^\infty(T')}\|\text{div }\boldsymbol\varphi_S\|_{L^{1}(T')}
+
\|\RR_{T}\|_{\LL^{\infty}(T')}\|\boldsymbol\varphi_S\|_{\LL^{1}(T')} + \|\ff-\Pi_{T}(\ff)\|_{\LL^{\infty}(T')}\|\boldsymbol\varphi_S\|_{\LL^{1}(T')} \bigg).
\end{multline}
With this estimate at hand, we invoke standard arguments and the derived estimate for $\|\RR_{T}\|_{\LL^{\infty}(T)}$ to arrive at
\begin{multline*}
\int_{S} [\![\nabla \yy_\T\cdot \boldsymbol\nu]\!]\boldsymbol\varphi_{S} 
\lesssim
\sum_{T' \in \mathcal{N}_S}\bigg(h_T^{d-2}\|\mathbf{e}_\yy\|_{\LL^\infty(T')}+h_{T}^{d-1} \|e_p\|_{L^\infty(T')}  
+ h_{T}^{d}\|\ff - \Pi_{T}(\ff)\|_{\LL^{\infty}(T')}\bigg).
\end{multline*}
We thus replace the previous estimate into \eqref{edge_estimate0} and use, in view of the mesh regularity assumptions, that $|T|/|S| \approx h_T$ to conclude that
\begin{multline*}
h_T \|[\![\nabla \yy_\T\cdot \boldsymbol\nu]\!]\|_{\LL^{\infty}(S)}
\lesssim
\sum_{T' \in \mathcal{N}_S}\bigg(\|\mathbf{e}_\yy\|_{\LL^\infty(T')}+h_T\|e_p\|_{L^\infty(T')}  
+ h_{T}^2\|\ff - \Pi_{T}(\ff)\|_{\LL^{\infty}(T')}\bigg).
\end{multline*}

\emph{Step 4.} Let $T\in \T$. The goal of this step is to estimate the term $h_T\|\text{div }\yy_\T\|_{L^\infty(T)}$ in \eqref{def:indicators_2}. To achieve this, we first use that $\mathrm{div }\ \yy=0$, and thus an integration by parts formula in conjunction with the properties \eqref{bubbleT_properties} of $\varphi_T$ to arrive at
\begin{align}\label{divergence_estimate}
\|\text{div }\yy_\T\|_{L^2(T)}^2
& \lesssim \int_{T} \text{div }(\yy_\T-\yy) \left(\varphi_T\text{ div }\yy_\T
\right)\lesssim 
\left|\int_T \mathbf{e}_\yy \cdot \nabla ( \varphi_T\text{div }\yy_\T)\right|
\\ \nonumber
& \lesssim
h_T^{\frac{d}{2}-1}\|\mathbf{e}_\yy\|_{\LL^\infty(T)}\|\text{div }\yy_{\T}\|_{L^{2}(T)},
\end{align}
where we also have used an inverse inequality. Consequently, using an inverse estimate, again, we conclude that 
\begin{equation}\label{divergence_estimate_3}
h_{T}\|\text{div }\yy_\T\|_{L^\infty(T)}
\lesssim h_T^{1-\frac{d}{2}}\|\text{div }\yy_\T\|_{L^2(T)} \lesssim
\|\mathbf{e}_\yy\|_{\LL^\infty(T)}.
\end{equation}

The collection of the estimates derived in \emph{Steps 2, 3} and \emph{4} concludes the proof.
\end{proof}

\subsection{A posteriori error estimates in energy norm.}\label{sec:standard_estimates}
In this section we briefly review a posteriori error estimates in energy norm for the Stokes equations \eqref{eq:weak_stokes_problem}. Let $(\yy_\T,p_{\T})$ be the solution to \eqref{eq:discrete_stokes_problem}. We introduce the residual--type a posteriori error estimator
\begin{equation}\label{def:standard_estimator}
E_\yy(\yy_\T,p_\T,\ff)
=
\left(\sum_{T\in\T}E_{\yy,T}^2(\yy_\T,p_\T,\ff)\right)^\frac{1}{2},
\end{equation}
where, for every $T\in\T$, the local error indicators $E_{\yy,T}(\yy_\T,p_\T,\ff)$ are defined by
\begin{multline}
{E}_{\yy,T}(\yy_{\T},p_{\T},\ff)
:= 
\bigg(h_T^2\|\ff+\Delta \yy_\T-\nabla p_\T\|_{\LL^2(T)}^2 
\\
+
\tfrac{h_T}{2}\|[\![\nabla \yy_\T \cdot \boldsymbol\nu]\!]\|_{\LL^2(\partial T\setminus \partial \Omega)}^2
+ 
\|\text{div }\yy_\T\|_{L^2(T)}^2 \bigg)^\frac{1}{2}.
\label{def:standard_indicator}
\end{multline}

The following result states the global reliability of the a posteriori error estimator $E_\yy$ and the local efficiency of the indicator $E_{\yy,T}$. For a proof see \citep[Theorem 4.70]{Verfurth}.

\begin{theorem}[global reliability of ${E}_{\yy}$ and local efficiency of ${E}_{\yy,T}$]\label{thm:reliabilty_efficiency_standard}
Let $(\yy,p)$ be the solution to the Stokes equations \eqref{eq:weak_stokes_problem} and $(\yy_\T,p_\T)\in\mathbf{V}(\T)\times Q(\T)$ its numerical approximation obtained as the solution to \eqref{eq:discrete_stokes_problem}. Then, for every $T\in\T$, the following a posteriori error estimates hold
\begin{equation}\label{eq:upper_estimate_standard}
\|\nabla(\yy-\yy_\T)\|_{\LL^2(\Omega)}^2+\|p-p_\T\|_{L^2(\Omega)}^2
\lesssim
E_\yy^2(\yy_{\T},p_{\T},\ff)
\end{equation}
and
\begin{equation}\label{eq:lower_estimate_standard}
E_{\yy,T}^2(\yy_{\T},p_{\T},\ff)
\lesssim
\|\nabla(\yy-\yy_\T)\|_{\LL^2(\mathcal{N}_T)}^2+\|p-p\|_{L^2(\mathcal{N}_T)}^2+\mathrm{osc}^{2}_{2}(\ff;\mathcal{N}_T),
\end{equation}
where $\mathcal{N}_T$ is defined as in \eqref{eq:patch}. The hidden constants are independent of the continuous and discrete solutions, the size of the elements in the mesh $\T$ and $\#\T$.
\end{theorem}


\section{The Stokes problem with Dirac sources}\label{sec:Stokes_dirac}
As previously mentioned, the a posteriori error estimator that we will design in Section \ref{sec:ocp} involves several contributions, which are related to the discretization of the state and adjoint equations and the control variable. We shall observe, in Section \ref{sec:pt_ocp}, that the adjoint equations involve, specifically, in the \emph{momentum equation}, a linear combination of Dirac sources as forcing term. It will be thus crucial to consider an a posteriori error estimator for the Stokes equations under such a singular setting. The purpose of this section is thus to review the arguments developed in \citep[Section 3]{Allendes_et_al2017}, where such an a posteriori error analysis is developed; the analysis relies on the theory of Muckenhoupt weights and Muckenhoupt weighted Sobolev spaces introduced in Section \ref{sec:wse}.

\subsection{Well--posedness.}\label{sec:reg_prop_deltas}
Let ${t_{0}}$ be an interior point of $\Omega$. Consider the following boundary value problem: Find $(\zz,r)$ such that 
\begin{equation}\label{eq:stokes_delta}
\left\{
\begin{array}{rcll}
-\Delta \mathbf{z} + \nabla r & = & \mathbf{F}\delta_{t_{0}} & \text{ in } \quad \Omega, \\
\text{div}\: \zz & = & 0 & \text{ in } \quad \Omega, \\
\zz & = & \mathbf{0} & \text{ on } \quad \partial\Omega,
\end{array}
\right.
\end{equation}
where $\delta_{t_{0}}$ denotes the Dirac delta supported at ${t_{0}}\in\Omega$ and  $\mathbf{F}\in \mathbb{R}^d$. The asymptotic behavior of the solution $(\zz,r)$ near the point $t_0$ \cite[Section IV.2]{Gal11}, which reads
\begin{equation}
\label{eq:asymptotic}
|\nabla \zz(x)| \approx |x-t_0|^{1-d}, \quad  |r(x)| \approx |x-t_0|^{1-d},
\end{equation}
motivates the following the weak formulation of problem \eqref{eq:stokes_delta} \citep[Section 3]{Allendes_et_al2017}: Find $(\zz,r)\in \HH_0^1(\mathsf{d}_{t_{0}}^{\alpha},\Omega)\times L^2(\mathsf{d}_{t_{0}}^{\alpha},\Omega)/\mathbb{R}$ such that 
\begin{equation}\label{eq:weak_stokes_delta}
\left\{\!\!
\begin{array}{rcll}
a(\zz,\ww)+b(\ww,r) & = & \langle \mathbf{F}\delta_{t_{0}},\ww\rangle  &\forall \: \ww \in  \HH_0^1(\mathsf{d}_{t_{0}}^{-\alpha},\Omega), \\
b(\zz,s) & = & 0 & \forall \: s \in L^2(\mathsf{d}_{t_{0}}^{-\alpha},\Omega)/\mathbb{R},
\end{array}
\right.
\hspace{-0.4cm}
\end{equation}
where $\langle \cdot,\cdot \rangle$ denotes the duality pairing between $\HH_0^1(\mathsf{d}_{t_{0}}^{-\alpha},\Omega)'$ and $\HH_0^1(\mathsf{d}_{t_{0}}^{-\alpha},\Omega)$. The following comments are in order:
\begin{enumerate}
\item if $\alpha \in (-d,d)$, the weights $\mathsf{d}_{t_{0}}^{\alpha}$ and $\mathsf{d}_{t_{0}}^{-\alpha}$ belong to the Muckenhoupt class $A_2(\mathbb{R}^d)$. Consequently $\HH_0^1(\mathsf{d}_{t_{0}}^{\alpha},\Omega)$ and $\HH_0^1(\mathsf{d}_{t_{0}}^{-\alpha},\Omega)$ are Hilbert, and
\item if $\alpha \in (d - 2, d)$, then $\delta_{t_{0}}\in H_0^1(\mathsf{d}_{t_{0}}^{-\alpha},\Omega)'$ \citep[Lemma 7.1.3]{Kozlov_et_al1997} and, consequently, the duality pairing term in \eqref{eq:weak_stokes_delta} is well--defined.
\end{enumerate}

We now present an alternative weak formulation for problem \eqref{eq:weak_stokes_delta}:
Find $(\zz,r)\in \HH_0^1(\mathsf{d}_{t_{0}}^{\alpha},\Omega)\times L^2(\mathsf{d}_{t_{0}}^{\alpha},\Omega)/\mathbb{R}$ such that 
\begin{equation}\label{eq:alt_weak_form}
c((\zz,r),(\ww,s))=\langle \mathbf{F}\delta_{t_{0}},\ww\rangle
\end{equation}
for all $(\ww,s) \in \HH_0^1(\mathsf{d}_{t_{0}}^{-\alpha},\Omega)\times L^2(\mathsf{d}_{t_{0}}^{-\alpha},\Omega)/\mathbb{R}$, where $c((\zz,r),(\ww,s)) := a(\zz,\ww)+b(\ww,r)-b(\zz,s)$. Since $t_0 \in \Omega$, there is a neighborhood of $\partial \Omega$ where $\mathsf{d}_{t_{0}}^{\alpha}$ has no degeneracies or singularities; $\mathsf{d}_{t_{0}}^{\alpha}$ thus belongs to the restricted Muckenhoupt class $A_2(\Omega)$ \cite[Definition 2.5]{MR1601373}. It can be proved that problem \eqref{eq:alt_weak_form} admits a unique solution; see \citep[Theorem 14]{Salgado_Otarola2017}. Moreover, the following a priori error estimate can be obtained \citep[Theorem 14]{Salgado_Otarola2017}:
\begin{equation}\label{eq:a_priori_estimates}
\|\nabla \zz\|_{\LL^2(\mathsf{d}_{t_{0}}^{\alpha},\Omega)}+\|r\|_{L^2(\mathsf{d}_{t_{0}}^{\alpha},\Omega)/\mathbb{R}}\lesssim|\mathbf{F}|\|\delta_{t_{0}}\|_{\HH_0^1(\mathsf{d}_{t_{0}}^{-\alpha},\Omega)'}.
\end{equation}
We finally notice that with such a well–posedness result at hand, an inf--sup condition for the bilinear form $c$ follows; see \cite[Th\'eor\`eme 6.3.1]{MR0227584} and \cite[Th\'eor\`emes 3.1 et 3.2]{MR0163054}.

\subsection{A posteriori error estimates.}\label{sec:stokes_deltas_aposteriori}
In this section we present the a posteriori error estimates developed in \citep[Section 5]{Allendes_et_al2017}. To accomplish this task, we begin by introducing the following finite element approximation to problem \eqref{eq:weak_stokes_delta}: Find $(\zz_\T,r_\T)\in \mathbf{V}(\T)\times Q(\T)$ such that
\begin{align}\label{eq:weak_pde_delta}
\begin{cases}
\begin{array}{rcll}
a(\zz_\T,\ww_\T)+b(\ww_\T,r_\T)&=& \mathbf{F}\cdot\ww_\T(t_0)\quad&\forall \: \ww_\T\in \mathbf{V}(\T),\\
b(\zz_\T,s_\T) &=&0\quad &\forall \: s_\T\in Q(\T).
\end{array}
\end{cases}
\end{align}
Notice that, since $\ww_\T\in \mathbf{C}(\bar{\Omega})$, we have that $\langle \mathbf{F}\delta_{t_0},\ww_\T\rangle = \mathbf{F}\cdot\ww_\T(t_0)$.  

To present the a posteriori error estimator, we define, for $T\in\T$, 
\begin{equation}\label{def:D_T}
D_T:=\max_{x\in T}|x-t_0|.
\end{equation}
With the previous discrete setting at hand, we introduce, for $\alpha\in(d-2,d)$ and $T\in \T$, the
\emph{element error indicators} 
\begin{multline}\label{def:pointwise_indicator}
\mathcal{E}_{\alpha,T}(\zz_\T,r_\T,\mathbf{F})
:=
\bigg(h_T^2 D_T^\alpha \|\Delta \zz_\T-\nabla r_\T\|_{\LL^2(T)}^2 + \|\text{div\:} \zz_\T\|^2_{L^2(\mathsf{d}^\alpha_{t_0},T) }\\
+\: h_T D_T^\alpha\|[\![\nabla \zz_\T\cdot \boldsymbol\nu ]\!]\|_{\LL^2(\partial T\setminus \partial \Omega)}^2+h_T^{\alpha+2-d}|\mathbf{F}|^2\chi(\{t_0\in T\})\bigg)^{\frac{1}{2}},
\end{multline}
where the function $\chi(\{t_0\in T\})$ equals one if $t_0\in T$ and zero otherwise. The \emph{error estimator} is thus defined as
\begin{equation}\label{def:pointwise_estimator}
\mathcal{E}_{\alpha}(\zz_\T,r_\T,\mathbf{F})
:= \left(\sum_{T\in \T}\mathcal{E}_{\alpha,T}^2(\zz_\T,r_\T,\mathbf{F})\right)^{\frac{1}{2}}.
\end{equation}

The following result states the global reliability of the a posteriori error estimator $\mathcal{E}_{\alpha}$ and the local efficiency of the indicator $\mathcal{E}_{\alpha,T}$.
\begin{theorem}[global reliability of $\mathcal{E}_{\alpha}$ and local efficiency of $\mathcal{E}_{\alpha,T}$]\label{thm:global_reli_delta_est}
Let $(\zz,r)$ be the unique solution to problem \eqref{eq:weak_stokes_delta} and $(\zz_\T,r_\T)\in \mathbf{V}(\T)\times Q(\T)$ its finite element approximation given as the solution to \eqref{eq:weak_pde_delta}. If $\alpha\in (d-2,d)$, we thus have that 
\begin{equation}\label{eq:global_rel_delta}
\|\nabla(\zz-\zz_\T)\|_{\LL^2(\mathsf{d}_{t_0}^\alpha,\Omega)}+\|r-r_\T\|_{L^2(\mathsf{d}_{t_0}^\alpha,\Omega)}
\lesssim 
\mathcal{E}_{\alpha}(\zz_\T,r_\T,\mathbf{F}),
\end{equation}
and
\begin{equation}
\label{eq:local_eff_delta}
\mathcal{E}_{\alpha,T}^2(\zz_\T,r_\T,\mathbf{F})
\lesssim
\|\nabla(\zz-\zz_\T)\|_{\LL^2(\mathsf{d}_{t_0}^\alpha,\mathcal{N}^{*}_T)}^2+\|r-r_\T\|_{L^2(\mathsf{d}_{t_0}^\alpha,\mathcal{N}^{*}_T)}^2,
\end{equation}
where $\mathcal{N}^{*}_T$ is defined as in \eqref{eq:patch_morin}. The hidden constants are independent of the continuous and discrete solutions, the size of the elements in the mesh $\T$ and $\#\T$.
\end{theorem}
\begin{proof}
We refer the reader to \citep[Theorem 7]{Allendes_et_al2017} and \citep[Theorem 10]{Allendes_et_al2017} for a proof of \eqref{eq:global_rel_delta} and \eqref{eq:local_eff_delta}, respectively.
\end{proof}


\section{The pointwise tracking optimal control problem.}\label{sec:pt_ocp}
In this section we precisely describe and analyze a weak version of the optimal control problem \eqref{def:cost_func}--\eqref{def:box_constraints}, which reads:
\begin{equation}\label{def:weak_ocp}
\min_{\HH_0^1(\Omega)\times\mathbb{U}_{ad}} J(\yy,\uu)
\end{equation}
subject to 
\begin{align}\label{eq:weak_pde}
\begin{cases}
\begin{array}{rcll}
a(\yy,\vv)+b(\vv,p)&=&(\uu,\vv)_{{\LL}^2(\Omega)}\quad&\forall \: \vv\in \HH_0^1(\Omega),\\
b(\yy,q) &=&0\quad &\forall \: q\in L^2(\Omega)/\mathbb{R}.
\end{array}
\end{cases}
\end{align}

Since $a$ is coercive on $\HH_0^1(\Omega)$ and $b$ satisfies an inf-sup condition, there is a unique solution $(\yy,p) \in \HH_0^1(\Omega) \times L^2(\Omega)/\mathbb{R}$ to problem \eqref{eq:weak_pde} \cite[Theorem 4.3]{Guermond-Ern}. In addition, we have that \cite[Theorem 4.3]{Guermond-Ern}
\begin{equation}\label{eq:standard_estimate_stokes}
\| \nabla \yy\|_{\LL^2(\Omega)} + \|p\|_{L^{2}(\Omega)} \lesssim \|\uu\|_{\LL^{2}(\Omega)}.
\end{equation}
Due to Rham’s Theorem \cite[Section 4.1.3]{Guermond-Ern} we can consider the following equivalent formulation of problem \eqref{eq:weak_pde}  \citep[Proposition 4.6]{Guermond-Ern}: Find $\yy \in  \mathbf{X}$ such that
\begin{equation}\label{eq:weak_constrained_pde}
a(\yy,\vv)=(\uu,\vv)_{{\LL}^2(\Omega)}\quad\forall \: \vv \in \mathbf{X},
\end{equation}
where $\mathbf{X}:=\{\vv\in \HH_0^1(\Omega):\text{ div }\vv=0\}$. 

To provide an analysis for \eqref{def:weak_ocp}--\eqref{eq:weak_pde}, we introduce the control-to-state operator  $\mathcal{S}:\LL^2(\Omega)\rightarrow\mathbf{X}$ which, given a control $\uu$, associates to it the unique state $\yy \in \HH_0^1(\Omega)$ that solves \eqref{eq:weak_constrained_pde}. With this operator at hand, we introduce the reduced cost functional
\begin{equation}\label{def:red_cost_functional}
j(\uu):=J(\mathcal{S}\uu,\uu)=\frac{1}{2}\sum_{t\in \mathcal{D}}|\mathcal{S}\uu(t)-\yy_t|^2+\frac{\lambda}{2}\|\uu\|_{\LL^2(\Omega)}^2.
\end{equation}
We comment that, since the control variable $\uu\in \mathbb{U}_{ad}\subset \LL^\infty(\Omega)$ and $\partial\Omega$ is Lipschitz, the results of Theorem \ref{thm:high_int} guarantee the Hölder regularity of $\yy=\mathcal{S}\uu$; point evaluations of $\yy=\mathcal{S}\uu$ in \eqref{def:red_cost_functional} are thus well defined. 

We present the following existence and uniqueness result.
\begin{theorem}[existence and uniqueness]
The optimal control problem \eqref{def:weak_ocp}--\eqref{eq:weak_pde} admits a unique solution $(\bar{\yy},\bar{\uu})\in\HH_0^1(\Omega)\times \mathbb{U}_{ad}$.
\end{theorem}
\begin{proof}
We begin by noticing that the reduced cost functional $j$ is strictly convex and continuous.  In addition, $\mathbb{U}_{ad}$ is a nonempty, bounded, convex, and closed subset of $\LL^2(\Omega)$. We thus apply \citep[Theorem 2.14]{Troltzsch} to conclude the desired result.
\end{proof}

The following result is standard \cite[Lemma 2.21]{Troltzsch}: If $\bar \uu$ denotes the optimal control of \eqref{def:weak_ocp}--\eqref{eq:weak_pde}, then
\begin{equation}
\label{eq:variational_inequality}
j'(\bar \uu) (\uu - \bar \uu) \geq 0 \quad \forall \: \uu \in \mathbb{U}_{ad}. 
\end{equation}
Here $j'(\bar{\uu})$ denotes the Gate\^aux--derivative of the functional $j$ in $\bar{\uu}$. To explore this variational  inequality and obtain optimality conditions we first shall state and derive some results on weighted Sobolev spaces. 

Let us consider an ordered set of points $\mathcal{D}\subset\Omega$ with finite cardinality $m := \#\mathcal{D} <\infty$. We define
\[
d_{\mathcal{D}} = \left\{\begin{array}{ll}
\mathrm{dist}(\mathcal{D},\partial \Omega), & \mbox{if }m=1,
\\
\min \left \{ \mathrm{dist}(\mathcal{D},\partial \Omega), \min\{|t-t'|: t,t' \in \mathcal{D}, \ t\neq t' \} \right \}, & \mbox{otherwise}.
\end{array}\right.
\]
Since $\mathcal{D} \subset \Omega$ and $\mathcal{D}$ is finite, we immediately conclude that $d_{\mathcal{D}}>0$. We now define the weight $\rho$ that will be of importance for the analysis that we will perform: if $m=1$, then
\begin{equation}\label{def:weight_rho}
\rho(x) = \mathsf{d}_t^\alpha(x),
\end{equation}
otherwise
\begin{equation}\label{def:weight_rho_complete}
\rho(x)=
\begin{cases}
\mathsf{d}_t^\alpha(x),\:&\exists t\in \mathcal{D}:\mathsf{d}_t(x) < \frac{d_\mathcal{D}}{2},\\
1, &\mathsf{d}_t(x)\geq \frac{d_\mathcal{D}}{2} \: \forall \: t\in \mathcal{D},
\end{cases}
\end{equation}
where $\mathsf{d}_t(x) := |x - t|$ and $\alpha \in (d-2,2)$. Since $(d-2,d) \subset (-d,d)$, it can be proved that the weight $\rho$ belongs to the Muckenhoupt class $A_{2}(\mathbb{R}^d)$ \cite[Theorem 6]{ACDT2014}. 

We present the following embedding result.
\begin{theorem}[$\HH_0^1(\rho,\Omega) \hookrightarrow \LL^2(\Omega)$]
\label{thm:weighted_poincare}
 If $\alpha \in (d - 2,2)$ then $\HH_0^1(\rho,\Omega) \hookrightarrow \LL^2(\Omega)$. Moreover, the following weighted Poincaré inequality holds
\begin{equation}
\|\vv\|_{\LL^{2}(\Omega)} \lesssim \|\nabla \vv\|_{\LL^{2}(\rho,\Omega)} \quad \forall \: \vv \in \HH_{0}^{1}(\rho,\Omega),
\end{equation}
where the hidden constant depends only on $\Omega$ and $d_{\mathcal{D}}$.
\end{theorem}
\begin{proof}
The proof follows from \cite[Lemmas 1 and 2]{Allendes_et_al2017_2}.
\end{proof}

We now derive, on the basis of the ideas of \cite[Lemma 3]{Allendes_et_al2017_2}, a regularity result in weighted Sobolev spaces.

\begin{lemma}[weighted regularity]
\label{lemma_weighted_reg}
Let $(\yy,p)\in \HH_0^1(\Omega)\times L^2(\Omega)/\mathbb{R}$ be the solution to \eqref{eq:weak_pde} with $\uu\in \mathbb{U}_{ad}$. Thus, we have that $(\yy,p)\in \HH_0^1(\rho^{-1},\Omega)\times L^2(\rho^{-1},\Omega)/\mathbb{R}$.
\end{lemma}
\begin{proof}
We prove that $\yy \in \HH_0^1(\rho^{-1},\Omega)$; similar arguments reveal that $p\in L^2(\rho^{-1},\Omega)/\mathbb{R}$. We begin by noticing that
\[
\|\nabla \yy\|_{\LL^2(\rho^{-1},\Omega)}^2=\sum_{i=1}^d \|\nabla	\yy_i\|_{L^2(\rho^{-1},\Omega)}^2=\sum_{i=1}^d\int_\Omega \rho^{-1}|\nabla \yy_i|^2.
\]
For each $t\in \mathcal{D}$, we denote by $B(t)$ the ball of center $t$ and radius $\frac{d_\mathcal{D}}{2}$ and set $A=\Omega\setminus\cup_{t\in \mathcal{D}}B(t)$. We thus have, for each $i\in\{1,\ldots,d\}$, that
\[
\int_\Omega \rho^{-1}|\nabla \yy_i|^2 = \int_A \rho^{-1}|\nabla \yy_i|^2 +  \sum_{t\in\mathcal{D}}\int_{B(t)}\rho^{-1}|\nabla \yy_i|^2 = \mathrm{I} + \mathrm{II}.
\]

We first estimate $\mathrm{I}$. In view of definitions \eqref{def:weight_rho} and \eqref{def:weight_rho_complete}, we conclude that there exists $a>0$ such that $\rho(x) \geq a$ for every $x\in A$. Consequently, since $\yy \in \HH_0^1(\Omega)$, we conclude in view of \eqref{eq:standard_estimate_stokes} that
\[
\mathrm{I} = \int_A \rho^{-1}|\nabla \yy_i|^2
\lesssim
\int_A |\nabla \yy_i|^2
\leq
\| \nabla \yy\|_{\LL^{2}(\Omega)}^2
\lesssim 
\|  \uu \|_{\LL^{2}(\Omega)}^2.
\]

We now bound $\mathrm{II}$. Since $B(t)\Subset\Omega$, $\rho \in A_2(\mathbb{R}^d)$, and $\uu \in \mathbb{U}_{ad} \subset \LL^{\infty}(\Omega)$, we can apply the results of Proposition \ref{prop:weighted_int} to arrive at the estimate
\[
\int_{B(t)}\rho^{-1}|\nabla \yy_i|^2
\lesssim
\rho^{-1}(B(t))\|\uu\|_{\LL^l(\Omega)}^2, \quad l>d, \quad i\in\{1,\ldots,d\},
\]
which implies that $\mathrm{II} \lesssim \|\uu\|_{\LL^l(\Omega)}^2$ for $l>d$. This concludes the proof.
\end{proof}

To explore \eqref{eq:variational_inequality} we introduce the adjoint variable $(\zz,r)$ as the unique solution to: Find $(\zz,r) \in \HH_{0}^{1}(\rho,\Omega) \times L^{2}(\rho,\Omega)/\mathbb{R}$ such that
\begin{equation}\label{eq:adj_eq}
\left\{\begin{array}{rcll}
a(\mathbf{w},\zz) - b(\mathbf{w},r)&=&\displaystyle{\sum_{t\in\mathcal{D}}}\langle ({\yy} - \yy_{t})\delta_{t},\mathbf{w} \rangle \quad&\forall \: \mathbf{w} \in \HH_0^1(\rho^{-1},\Omega),\\
b(\zz,s) &=&0 \quad &\forall \:  s \in L^2(\rho^{-1},\Omega)/\mathbb{R},
\end{array}
\right.
\end{equation}
where $\yy = \mathcal{S}\uu$ solves \eqref{eq:weak_pde}. The well--posedness of \eqref{eq:adj_eq} follows from \cite[Section 4]{Salgado_Otarola2017} combined with the fact that $\delta_{t} \in H_{0}^{1}(\rho^{-1},\Omega)'$ \cite[Lemma 7.1.3]{Kozlov_et_al1997}. 

\begin{theorem}[optimality conditions]
\label{thm:optimality_cond}
Let $\alpha \in (d-2,d)$. The pair $(\bar{\yy},\bar{\uu}) \in \HH_{0}^{1}(\Omega) \times \mathbb{U}_{ad}$ is optimal for the pointwise tracking optimal control problem \eqref{def:weak_ocp}--\eqref{eq:weak_pde} if and only if $\bar \yy = \mathcal{S} \bar \uu$ and $\bar{\uu} \in \mathbb{U}_{ad}$ satisfies the variational inequality
\begin{equation}\label{eq:variational_ineq}
(\bar{\zz} + \lambda \bar{\uu},\uu - \bar{\uu})_{\LL^{2}(\Omega)} \geq 0  \quad \forall \: \uu \in \mathbb{U}_{ad},
\end{equation}
where $(\bar{\zz},\bar r) \in \HH^{1}_{0}(\rho,\Omega) \times \LL^{2}(\rho,\Omega)/ \mathbb{R}$ corresponds to the optimal adjoint state, which solves \eqref{eq:adj_eq} with $\yy$ replaced by $\bar \yy = \mathcal{S}\bar{\uu}$.
\end{theorem}
\begin{proof}
A simple computation shows that, for all $\uu\in\mathbb{U}_{ad}$, \eqref{eq:variational_inequality} can be written as follows:
\begin{equation}\label{eq:var_ineq2}
\sum_{t \in \mathcal{D}}\left(\mathcal{S}\bar{\uu}(t) - \yy_{t} \right)(\yy - \bar{\yy})(t) + \lambda(\bar{\uu},\uu - \bar{\uu})_{\LL^{2}(\Omega)} \geq 0,
\end{equation}
where $\mathbf{\yy} = \mathcal{S}\uu$. Let us concentrate on the first term on the left hand side of the previous expression. To study such a term, we invoke the results of Lemma \ref{lemma_weighted_reg} to conclude that $\yy - \bar{\yy} \in \HH^{1}_{0}(\Omega)\cap\HH^{1}_{0}(\rho^{-1},\Omega)$ and $p - \bar{p} \in L^{2}(\Omega)/\mathbb{R}\cap L^{2}(\rho^{-1},\Omega)/\mathbb{R}$. We can thus consider $\ww = \yy - \bar{\yy}$ and $s = p - \bar{p}$ as test functions in problem \eqref{eq:adj_eq}. This yields, on the basis of $\text{div}(\yy - \bar{\yy}) = 0$ a.e. in $\Omega$, that
\begin{equation}\label{eq:opt_cond_1}
a(\yy - \bar{\yy},\bar{\mathbf{z}}) = \sum_{t \in \mathcal{D}}(\bar{\yy}(t) - \yy_{t})(\yy - \bar{\yy})(t).
\end{equation}
Now, notice that $(\yy-\bar{\yy},p-\bar{p}) \in \HH^{1}_{0}(\Omega) \times L^{2}(\Omega)/\mathbb{R}$ solves the problem
\begin{align}\label{eq:y-ybar_eq}
\begin{cases}
\begin{array}{rcll}
a(\yy-\bar{\yy},\vv)+b(\vv,p-\bar{p})&=&(\uu- \bar{\uu},\vv)_{{\LL}^2(\Omega)} \quad&\forall \: \vv\in \HH_0^1(\Omega),\\
b(\yy-\bar{\yy},q) &=&0 \quad &\forall \: q\in L^2(\Omega)/\mathbb{R}.
\end{array}
\end{cases}
\end{align}
With this problem at hand, we invoke a density argument and obtain that
\begin{equation}\label{eq:opt_cond_2}
a(\yy - \bar{\yy},\bar{\mathbf{z}})_{\LL^{2}(\Omega)}=(\uu-\bar{\uu},\bar{\zz})_{{\LL}^2(\Omega)}.
\end{equation}
In fact, let $\{\zz_{n}\}_{n \in \mathbb{N}} \subset \mathbf{C}^{\infty}_{0}(\Omega)$ be such that $\zz_n\rightarrow \bar{\zz}$ in $\HH_0^1(\rho,\Omega)$. We can thus set, for $n \in \mathbb{N}$, $\vv=\zz_n$ and $q=0$ in \eqref{eq:y-ybar_eq}. This yields 
\begin{equation*}
a(\yy-\bar{\yy},\zz_n)_{\LL^2(\Omega)}+b(\zz_n,p-\bar{p})=(\uu-\bar{\uu},\zz_n)_{{\LL}^2(\Omega)}.
\end{equation*}
We now observe that
\begin{equation*}
| \left (\uu-\bar{\uu},\bar{\zz})_{\LL^2(\Omega)} - (\uu-\bar{\uu},\zz_n)_{{\LL}^2(\Omega)} \right| \leq \rho^{-1}(\Omega)^{\frac{1}{2}} \|\uu-\bar{\uu} \|_{\LL^{\infty}}\|\bar{\zz} - \zz_{n} \|_{\LL^{2}(\rho,\Omega)} \rightarrow 0
\end{equation*}
as $n \rightarrow \infty$ upon using a Poincar\'e inequality.
The continuity of the bilinear form $b$ on $\HH_0^1(\rho,\Omega) \times L^2(\rho^{-1},\Omega)$ immediately implies that $b(\zz_n,p-\bar{p})$ converges to $0$. Finally, the continuity of the bilinear form $a$ on  $\HH_0^1(\rho^{-1},\Omega) \times \HH_0^1(\rho,\Omega)$ and the fact that $\yy-\bar{\yy} \in \HH_0^1(\rho^{-1},\Omega)$ allow us to obtain the required expression \eqref{eq:opt_cond_2}. This, \eqref{eq:var_ineq2}, and \eqref{eq:opt_cond_1} allow us to conclude.
\end{proof}
In order to obtain an explicit characterization for the optimal control variable $\bar{\uu}$, we introduce the projection operator $\Pi:\LL^1(\Omega)\rightarrow \mathbb{U}_{ad}$ as
\begin{equation}\label{def:proj_operator}
\Pi(\vv):=\min\{\mathbf{b},\max\{\vv,\mathbf{a}\}\}.
\end{equation}
With this projector at hand, we recall the so--called projection formula; see \citep[Lemma 2.26]{Troltzsch}: The optimal control $\bar{\uu}$ satisfies \eqref{eq:variational_ineq} if and only if 
\begin{equation}\label{eq:projection_formula}
\bar{\uu}=\Pi\bigg(-\frac{1}{\lambda}\bar{\zz}\bigg).
\end{equation}

To summarize, the pair $(\bar{\yy},\bar{\uu})$ is optimal for the pointwise tracking optimal control problem \eqref{def:weak_ocp}--\eqref{eq:weak_pde} if and only if $(\bar{\yy},\bar{p},\bar{\zz},\bar{r},\bar{\uu})\in \HH_0^1(\Omega)\times L^2(\Omega)/\mathbb{R}\times \HH_0^1(\rho,\Omega)\times L^2(\rho,\Omega)/\mathbb{R}\times \mathbb{U}_{ad}$ solves \eqref{eq:weak_pde}, \eqref{eq:adj_eq} and \eqref{eq:variational_ineq}.


\section{A posteriori error analysis for the optimal control problem.}\label{sec:ocp}
The optimal adjoint pair $(\bar{\zz},\bar{r})$, that solves \eqref{eq:adj_eq}, exhibits reduced regularity properties. In fact, the asymptotic behavior \eqref{eq:asymptotic} implies that $(\bar{\zz},\bar{r}) \notin \HH^2(\Omega) \times H^1(\Omega)$. As a consequence, optimal error estimates for an standard a priori error analysis of \eqref{def:weak_ocp}--\eqref{eq:weak_pde} cannot be expected. This motivates the development and analysis of adaptive finite element methods (AFEMs) for problem \eqref{def:weak_ocp}--\eqref{eq:weak_pde}. In addition, as it is customary in a posteriori error analysis, the study of AFEMs are also motivated by restrictions on the domain $\Omega$ that are needed to perform an a priori error analysis. In the following section we will propose and analyze a reliable and locally efficient a posteriori error estimator for the optimal control problem \eqref{def:weak_ocp}--\eqref{eq:weak_pde}. To accomplish this task, we begin by introducing a discrete scheme for such an optimal control problem.

\subsection{Finite element discretization.}
\label{sec:fem}
In order to propose a solution technique for problem \eqref{def:weak_ocp}--\eqref{eq:weak_pde}, we define 
\begin{equation*}
\mathbb{U}_{ad}(\T) := \mathbf{U}(\T)\cap \mathbb{U}_{ad},
\quad 
\mathbf{U}(\T):= \{ \uu\in \mathbf{C}(\bar \Omega) \ : \ \uu|^{}_T\in \mathbb{P}_2(T)^{d}\ \forall\: T \in \T\}.
\end{equation*}
The discrete counterpart of \eqref{def:weak_ocp}--\eqref{eq:weak_pde} thus reads as follows: Find min $J(\yy_\T,\uu_\T)$ subject to the discrete state equations
\begin{equation}\label{eq:discrete_state_eq}
\left\{
\begin{array}{rcll}
a(\yy_\T,\vv_\T) + b(\vv_\T,p_\T) & = & (\uu_\T,\vv_\T)_{\LL^2(\Omega)} & \forall\: \vv_\T \in \mathbf{V}(\T), \\
b(\yy_\T,q_\T) & = & 0 & \forall\: q_\T \in Q(\T),
\end{array}
\right.
\end{equation}
and the discrete control constraints $\uu_\T\in\mathbb{U}_{ad}(\T)$. Standard arguments reveal the existence of a unique optimal pair $(\bar{\yy}_\T,\bar{\uu}_\T)$. In addition,  the pair $(\bar{\yy}_\T,\bar{\uu}_\T)$ is optimal for the previous discrete optimal control problem if and only if $\bar{\yy}_\T$ solves \eqref{eq:discrete_state_eq}, and $\bar{\uu}_\T$ satisfies the variational inequality
\begin{equation}\label{eq:discrete_variational_ineq}
(\bar{\zz}_\T + \lambda \bar{\uu}_\T,\uu_\T - \bar{\uu}_\T)_{\LL^{2}(\Omega)} \geq 0  \quad \forall \: \uu_\T \in \mathbb{U}_{ad}(\T),
\end{equation}
where $(\bar \zz_\T, \bar r_{\T})$ solves
\begin{equation}\label{eq:discrete_adj_eq}
\left\{\begin{array}{rcll}
a(\mathbf{w}_\T,\zz_\T) - b(\mathbf{w}_\T,r_\T)&=&\displaystyle{\sum_{t\in\mathcal{D}}}\langle ({\yy}_\T - \yy_{t})\delta_{t},\mathbf{w}_\T \rangle \quad& \forall \: \mathbf{w}_\T \in \mathbf{V}(\T), \\
b(\zz_\T,s_\T) &=&0 \quad & \forall \: s_\T \in Q(\T). 
\end{array}
\right.\hspace{-0.4cm}
\end{equation}
\subsection{A posteriori error estimates.}\label{sec:aposteriori_estimates}
We now construct the error estimators associated with the state and adjoint equations, \eqref{eq:weak_pde} and \eqref{eq:adj_eq}, respectively. To accomplish this task, we introduce the following auxiliary variables: Let $(\hat{\yy},\hat{p})\in \HH_0^1(\Omega)\times L^2(\Omega)/\mathbb{R}$ and $(\hat{\zz},\hat{r})\in \HH_0^1(\rho,\Omega)\times L^2(\rho,\Omega)/\mathbb{R}$ be the solutions to
\begin{align}\label{eq:state_hat}
\begin{cases}
\begin{array}{rcll}
a(\hat{\yy},\vv)+b(\vv,\hat{p})&=&(\bar{\uu}_\T,\vv)_{{\LL}^2(\Omega)}\quad&\forall \: \vv\in \HH_0^1(\Omega),\\
b(\hat{\yy},q) &=&0\quad &\forall \: q\in L^2(\Omega)/\mathbb{R},
\end{array}
\end{cases}
\end{align}
and
\begin{equation}\label{eq:adjoint_hat}
\left\{\begin{array}{rcll}
a(\mathbf{w},\hat{\zz}) - b(\mathbf{w},\hat{r})&=&\displaystyle{\sum_{t\in\mathcal{D}}}\langle (\bar{\yy}_\T - \yy_{t})\delta_{t},\mathbf{w} \rangle \quad&\forall \: \mathbf{w} \in \HH_0^1(\rho^{-1},\Omega),\\
b(\hat{\zz},s) &=&0 \quad &\forall \: s \in L^2(\rho^{-1},\Omega)/\mathbb{R},
\end{array}
\right.
\end{equation}
respectively. We immediately notice that $(\bar{\yy}_\T,\bar{p}_\T)$ and $(\bar{\zz}_\T,\bar{r}_\T)$ can be seen as finite element approximations of $(\hat{\yy},\hat{p})$ and $(\hat{\zz},\hat{r})$, respectively. These properties motivate the introduction of the following local error indicators:
\begin{multline}\label{def:indicator_st}
{E}_{st,T}(\bar{\yy}_\T,\bar{p}_\T,\bar{\uu}_\T)
:=
\left(h_T^2\|\bar{\uu}_\T+\Delta \bar{\yy}_\T-\nabla \bar{p}_\T\|_{\LL^2(T)}^2   \right.\\  
 \left. +\tfrac{h_T}{2}\|[\![\nabla \bar{\yy}_\T \cdot \boldsymbol\nu]\!]\|_{\LL^2(\partial T\setminus \partial \Omega)}^2 + \|\text{div }\bar{\yy}_\T\|_{L^2(T)}^2 \right)^\frac{1}{2},
\end{multline}
\begin{multline}\label{def:indicator_st_1}
\mathcal{E}_{st,T}(\bar{\yy}_\T,\bar{p}_\T,\bar{\uu}_\T)
:= h_T^2\|\bar{\uu}_\T+\Delta \bar{\yy}_\T-\nabla \bar{p}_\T\|_{\LL^\infty(T)} \\  
+ \tfrac{h_T}{2}\|[\![\nabla \bar{\yy}_\T \cdot \boldsymbol\nu]\!]\|_{\LL^\infty(\partial T\setminus \partial \Omega)} + h_T\|\text{div }\bar{\yy}_\T\|_{L^\infty(T)},
\end{multline}
\begin{multline}\label{def:indicator_ad}
\mathcal{E}_{ad,T}(\bar{\zz}_\T,\bar{r}_\T,\bar{\yy}_\T)
:=
\bigg(h_T^2 D_T^\alpha \|\Delta \bar{\zz}_\T+\nabla \bar{r}_\T\|_{\LL^2(T)}^2 + \| \text{div } \bar{\zz}_\T\|^2_{L^2(\rho,T) }
\\
+\ h_T D_T^\alpha\|[\![\nabla \bar{\zz}_\T\cdot \boldsymbol\nu ]\!]\|_{\LL^2(\partial T\setminus \partial \Omega)}^2  +\sum_{t\in\mathcal{D}}h_T^{\alpha+2-d}|\bar{\yy}_\T(t)-\yy_t|^2\chi(\{t\in T\})\bigg)^{\frac{1}{2}},
\end{multline}
where 
\[
D_T = \min_{t \in \mathcal{D}} \left\{ \max_{x\in T} |x-t |\right\}.
\]
With these local error indicators at hand, we introduce the following a posteriori error estimators:
\begin{align}
\label{def:estimator_pressure}
E_{st}(\bar{\yy}_\T,\bar{p}_\T,\bar{\uu}_\T):&=\left(\sum_{T\in\T}{E}_{st,T}^2(\bar{\yy}_\T,\bar{p}_\T,\bar{\uu}_\T)\right)^{\frac{1}{2}},
\\\label{def:estimator_velocity}
\mathcal{E}_{st}(\bar{\yy}_\T,\bar{p}_\T,\bar{\uu}_\T):&=\max_{T\in\T}\mathcal{E}_{st,T}(\bar{\yy}_\T,\bar{p}_\T,\bar{\uu}_\T),
\\\label{def:estimator_ad_velocity}
\mathcal{E}_{ad}(\bar{\zz}_\T,\bar{r}_\T,\bar{\yy}_\T):&=\left(\sum_{T\in\T}\mathcal{E}_{ad,T}^2(\bar{\zz}_\T,\bar{r}_\T,\bar{\yy}_\T)\right)^{\frac{1}{2}}.
\end{align}

We assume that
\begin{equation}\label{eq:patch_property}
\forall \: T \in \T, \ \# (\mathcal{N}_T^*\cap \mathcal{D})\leq 1, 
\end{equation}
that is, for every element $T\in\T$ its patch $\mathcal{N}^{*}_T$ contains at most one observable point. This is not a restrictive assumption, as it can always be satisfied by starting with a suitably refined mesh $\T_0$.

In view of the results presented in Sections \ref{sec:pointwise_aposteriori_stokes}, \ref{sec:standard_estimates} and \ref{sec:stokes_deltas_aposteriori}, we can immediately conclude the following estimates:
\begin{align}\label{eq:state_estimate_press}
\|\hat{p}-\bar{p}_\T\|_{L^2(\Omega)}
&\lesssim
{E}_{st}(\bar{\yy}_\T,\bar{p}_\T,\bar{\uu}_\T),
\\ \label{eq:state_estimate_velo}
\|\hat{\yy}-\bar{\yy}_\T\|_{\LL^\infty(\Omega)}
&\lesssim
\ell_\T^{\beta_d}\mathcal{E}_{st}(\bar{\yy}_\T,\bar{p}_\T,\bar{\uu}_\T),
\end{align}
and
\begin{equation}\label{eq:adjoint_estimate}
\|\nabla(\hat{\zz}-\bar{\zz}_\T)\|_{\LL^2(\rho,\Omega)}+\|\hat{r}-\bar{r}_\T\|_{L^2(\rho,\Omega)}
\lesssim
\mathcal{E}_{ad}(\bar{\zz}_\T,\bar{r}_\T,\bar{\yy}_\T),
\end{equation}
where $\ell_\T$ is defined in \eqref{def:l_T}, and $\beta_d$ is provided in the statement of Theorem \ref{thm:reliability_max}.

We now define the a posteriori error estimator associated to the discretization of the optimal control variable
\begin{equation}
\mathcal{E}_{ct}(\bar{\zz}_\T,\bar{\uu}_\T)= \left(\sum_{T\in\T}\mathcal{E}_{ct,T}^2(\bar{\zz}_\T,\bar{\uu}_\T)\right)^{\frac{1}{2}},
\label{def:control_estimator}
\end{equation}
based on the local error indicators
\begin{equation}
\label{def:control_indicator}
\mathcal{E}_{ct,T}(\bar{\zz}_\T,\bar{\uu}_\T)=\|\bar{\uu}_\T-\Pi(\lambda^{-1}\bar{\zz}_\T)\|_{\LL^2(T)}^{}.
\end{equation}

On the basis of the previously introduced a posteriori error estimators, we define the global a posteriori error estimator to the optimal control problem  \eqref{def:weak_ocp}--\eqref{eq:weak_pde} as the sum of four contributions:
\begin{equation}\label{def:estimator_ocp}
\mathcal{E}_{ocp}(\bar{\zz}_\T,\bar{r}_\T,\bar{\yy}_\T,\bar{p}_\T,\bar{\uu}_\T)
:=
\left(\mathcal{E}_{st}^2+\mathcal{E}_{ad}^2+\mathcal{E}_{ct}^2+ E_{st}^2\right)^{\frac{1}{2}}.
\end{equation}

In order to prove a reliability result for the error estimator \eqref{def:estimator_ocp}, we introduce the following auxiliary variables: Let $(\tilde{\yy},\tilde{p})\in \HH_0^1(\Omega)\times L^2(\Omega)/\mathbb{R}$ and $(\tilde{\zz},\tilde{r})\in \HH_0^1(\rho,\Omega)\times L^2(\rho,\Omega)/\mathbb{R}$ be the solutions to
\begin{align}\label{eq:state_tilde}
\begin{cases}
\begin{array}{rcll}
a(\tilde{\yy},\vv)+b(\vv,\tilde{p})&=&(\tilde{\uu},\vv)_{{\LL}^2(\Omega)}\quad&\forall \: \vv\in \HH_0^1(\Omega),\\
b(\tilde{\yy},q) &=&0\quad &\forall \: q\in L^2(\Omega)/\mathbb{R},
\end{array}
\end{cases}
\end{align}
and
\begin{equation}\label{eq:adjoint_tilde}
\left\{\begin{array}{rcll}
a(\mathbf{w},\tilde{\zz}) - b(\mathbf{w},\tilde{r})&=&\displaystyle{\sum_{t\in\mathcal{D}}}\langle (\tilde{\yy} - \yy_{t})\delta_{t},\mathbf{w} \rangle \quad&\forall \: \mathbf{w} \in \HH_0^1(\rho^{-1},\Omega),\\
b(\tilde{\zz},s) &=&0 \quad &\forall \: s \in L^2(\rho^{-1},\Omega)/\mathbb{R},
\end{array}
\right.
\end{equation}
respectively, where $\tilde{\uu}:=\Pi(-\frac{1}{\lambda}\bar{\zz}_\T)$.

Finally, we define $\mathbf{e}_{\zz}:= \bar{\zz}-\bar{\zz}_\T$, $e_r:=\bar{r}-\bar{r}_\T$, $\mathbf{e}_{\uu} := \bar{\uu}-\bar{\uu}_\T$, and
\begin{equation}\label{def:error_norm}
\|\mathbf{e}  \|^2_{\Omega}: =
 \|\mathbf{e}_{\yy}\|_{\LL^\infty(\Omega)}^2 +
\|e_p \|_{L^{2}(\Omega)}^2 +
\|\nabla \mathbf{e}_{\zz} \|_{\LL^2(\rho,\Omega)}^2
+ \|e_r\|_{L^2(\rho,\Omega)/\mathbb{R}}^2
+ \|\mathbf{e}_{\uu}\|_{\LL^2(\Omega)}^2,
\end{equation}
where $\mathbf{e}_{\yy}$ and $e_p$ are given as in Theorem \ref{thm:local_efficiency_maximum}.

\subsection{A posteriori error estimator: reliability}\label{sec:reliability}
With all the previous ingredients at hand, we can establish the following result.

\begin{theorem}[global reliability property of $\mathcal{E}_{ocp}$]
\label{thm:global_reliability}
Let $(\bar{\yy},\bar{p},\bar{\zz},\bar{r},\bar{\uu})\in \HH_0^1(\Omega)\times \\ L^2(\Omega)/\mathbb{R}\times \HH_0^1(\rho,\Omega)\times L^2(\rho,\Omega)/\mathbb{R}\times \mathbb{U}_{ad}$ be the solution to the optimality system \eqref{eq:weak_pde}, \eqref{eq:adj_eq} and \eqref{eq:variational_ineq} and $(\bar{\yy}_\T,\bar{p}_\T,\bar{\zz}_\T,\bar{r}_\T,\bar{\uu}_\T)\in \mathbf{V}(\T)\times Q(\T)\times\mathbf{V}(\T)\times Q(\T)\times \mathbb{U}_{ad}(\T)$ its numerical approximation given by \eqref{eq:discrete_state_eq}--\eqref{eq:discrete_adj_eq}. If $\alpha\in (d-2,2)$, then 
\begin{equation}\label{global_reliability_2}
\| \mathbf{e}  \|^2_{\Omega} \lesssim
\ell_\T^{2\beta_d}\mathcal{E}_{st}^2+\mathcal{E}_{ad}^2+\mathcal{E}_{ct}^2+{E}_{st}^2
\lesssim
(1+\ell_\T^{2\beta_d})\mathcal{E}_{ocp}^2.
\end{equation}
The term $\ell_\T$ is defined in \eqref{def:l_T}, $\beta_d$ is given as in Theorem \ref{thm:reliability_max} and the hidden constants are independent of the continuous and discrete solutions, the size of the elements in the mesh $\T$ and $\#\T$. The constants, however, blow up as $\lambda\downarrow 0$.
\end{theorem}

\begin{proof}
We proceed in six steps.

\emph{Step 1.} We bound the error $\|\bar{\uu}-\bar{\uu}_\T\|_{\LL^2(\Omega)}$. To accomplish this task, we recall that $\tilde{\uu}=\Pi(-\frac{1}{\lambda}\bar{\zz}_\T)$ and notice that it can be equivalently characterized by 
\begin{equation}\label{eq:ineq_u_tilde}
(\bar{\zz}_\T+\lambda \tilde{\uu},\uu-\tilde{\uu})_{\LL^2(\Omega)}\geq 0  \quad \forall\: \uu\in \mathbb{U}_{ad};
\end{equation}
\cite[Lemma 2.26]{Troltzsch}. With the auxiliary control variable $\tilde{\uu}$ at hand, a simple application of the triangle inequality yields
\begin{equation}\label{eq:control_ineq_1}
\|\bar{\uu}-\bar{\uu}_\T\|_{\LL^2(\Omega)}
\leq
\|\bar{\uu}-\tilde{\uu}\|_{\LL^2(\Omega)} + \|\tilde{\uu}-\bar{\uu}_\T\|_{\LL^2(\Omega)}.
\end{equation}
In view of the definition of $\tilde{\uu}$, the second term on the right hand side of \eqref{eq:control_ineq_1} corresponds to the global error estimator $\mathcal{E}_{ct}$ which is defined in \eqref{def:control_estimator}.
It thus suffices to control the term  $\|\bar{\uu}-\tilde{\uu}\|_{\LL^2(\Omega)}$. We thus begin by setting $\uu=\tilde{\uu}$ in \eqref{eq:variational_ineq} and $\uu=\bar{\uu}$ in \eqref{eq:ineq_u_tilde}. Adding the obtained inequalities we arrive at
\begin{equation}\label{eq:control_ineq_2}
\lambda\|\bar{\uu}-\tilde{\uu}\|_{\LL^2(\Omega)}^2
\leq 
(\bar{\zz}-\bar{\zz}_\T,\tilde{\uu}-\bar{\uu})_{\LL^2(\Omega)}.
\end{equation}

We now invoke the auxiliary adjoint states $\hat{\zz}$ and $\tilde{\zz}$, defined as the solutions to problems \eqref{eq:adjoint_hat} and \eqref{eq:adjoint_tilde}, respectively, to write the previous inequality as follows:
\begin{align*}
\lambda\|\bar{\uu}-\tilde{\uu}\|_{\LL^2(\Omega)}^2
&\leq
(\bar{\zz}-\tilde{\zz},\tilde{\uu}-\bar{\uu})_{\LL^2(\Omega)} +(\tilde{\zz}-\hat{\zz},\tilde{\uu}-\bar{\uu})_{\LL^2(\Omega)} +(\hat{\zz}-\bar{\zz}_\T,\tilde{\uu}-\bar{\uu})_{\LL^2(\Omega)}. 
\end{align*}
We bound the term $(\bar{\zz}-\tilde{\zz},\tilde{\uu}-\bar{\uu})_{\LL^2(\Omega)}$. Notice that $(\tilde{\yy}-\bar{\yy},\tilde{p}-\bar{p})\in \HH_0^1(\Omega)\times L^2(\Omega)/\mathbb{R}$ and $(\bar{\zz}-\tilde{\zz},\bar{r}-\tilde{r})\in \HH_0^1(\rho,\Omega)\times L^2(\rho,\Omega)/\mathbb{R}$ solve 
\begin{align}
\begin{cases}\label{eq:state_tilde_bar}
\begin{array}{rcll}
a(\tilde{\yy}-\bar{\yy},\vv)+b(\vv,\tilde{p}-\bar{p})&=&(\tilde{\uu}- \bar{\uu},\vv)_{{\LL}^2(\Omega)}, 
\\
b(\tilde{\yy}-\bar{\yy},q) &=&0,
\end{array}
\end{cases}
\end{align}
for all $\vv\in \HH_0^1(\Omega)$ and $q\in L^2(\Omega)/\mathbb{R}$, and
\begin{align}
\label{eq:adjoint_bar_tilde}
\begin{cases}
\begin{array}{rcll}
a(\mathbf{w},\bar{\zz}-\tilde{\zz}) - b(\mathbf{w},\bar{r}-\tilde{r})&=&\displaystyle{\sum_{t\in\mathcal{D}}}\langle (\bar{\yy} - \tilde{\yy})\delta_{t},\mathbf{w} \rangle,
\\
b(\bar{\zz}-\tilde{\zz},s) &=&0,
\end{array}
\end{cases}
\end{align}
for all $\mathbf{w} \in \HH_0^1(\rho^{-1},\Omega)$ and $s \in L^2(\rho^{-1},\Omega)/\mathbb{R}$, respectively. We thus set $\ww=\tilde{\yy}-\bar{\yy} \in \HH_0^1(\Omega) \cap \HH_0^1(\rho^{-1},\Omega)$ in \eqref{eq:adjoint_bar_tilde}. Similar density arguments to the ones developed in the proof of Theorem \ref{thm:optimality_cond} reveal that \eqref{eq:state_tilde_bar} holds with $\vv = \bar{\zz}-\tilde{\zz}$. Consequently,
\[
(\bar{\zz}-\tilde{\zz},\tilde{\uu}-\bar{\uu})_{\LL^2(\Omega)}=-\sum_{t\in\mathcal{D}}|\bar \yy(t) -\tilde \yy(t)|^2\leq 0.
\]
This estimate allows us to conclude that
\[
\lambda\|\bar{\uu}-\tilde{\uu}\|_{\LL^2(\Omega)}^2
\leq 
(\tilde{\zz}-\hat{\zz},\tilde{\uu}-\bar{\uu})_{\LL^2(\Omega)} +(\hat{\zz}-\bar{\zz}_\T,\tilde{\uu}-\bar{\uu})_{\LL^2(\Omega)}.
\]
Standard estimates combined with the weighted Poincar\'e inequality of Theorem \ref{thm:weighted_poincare} allow us to arrive at
\begin{align}\label{eq:control_ineq_3}
\begin{split}
\|\bar{\uu}-\tilde{\uu}\|_{\LL^2(\Omega)}^2
&\lesssim
\|\tilde{\zz}-\hat{\zz}\|_{\LL^2(\Omega)}^2
+\|\hat{\zz}-\bar{\zz}_\T\|_{\LL^2(\Omega)}^2\\
&\lesssim
\|\tilde{\zz}-\hat{\zz}\|_{\LL^2(\Omega)}^2 +\|\nabla(\hat{\zz}-\bar{\zz}_\T)\|_{\LL^2(\rho,\Omega)}^2
\lesssim\|\tilde{\zz}-\hat{\zz}\|_{\LL^2(\Omega)}^2 +\mathcal{E}_{ad}^2,
\end{split}
\end{align}
where, in the last inequality, we have used the a posteriori error estimate \eqref{eq:adjoint_estimate}. 

We now bound the term $\|\tilde{\zz}-\hat{\zz}\|_{\LL^2(\Omega)}$. Notice that the pair $(\tilde{\zz}-\hat{\zz},\tilde{r}-\hat{r})\in\HH_0^1(\rho,\Omega)\times L^2(\rho,\Omega)$ solves 
\begin{equation*}
\begin{cases}
\begin{array}{rcll}
a(\mathbf{w},\tilde{\zz}-\hat{\zz}) - b(\mathbf{w},\tilde{r}-\hat{r})&=&\displaystyle{\sum_{t\in\mathcal{D}}}\langle (\tilde{\yy} - \bar{\yy}_\T)\delta_{t},\mathbf{w} \rangle \quad&\forall \:\mathbf{w} \in \HH_0^1(\rho^{-1},\Omega),\\
b(\tilde{\zz}-\hat{\zz},s) &=&0 \quad &\forall \: s \in L^2(\rho^{-1},\Omega)/\mathbb{R}.
\end{array}
\end{cases}
\end{equation*}
We thus first apply the estimate of Theorem \ref{thm:weighted_poincare} and then the stability estimate \eqref{eq:a_priori_estimates} to conclude that
\begin{equation}\label{eq:control_ineq_4}
\|\tilde{\zz}-\hat{\zz}\|_{\LL^2(\Omega)}\lesssim \|\nabla(\tilde{\zz}-\hat{\zz})\|_{\LL^2(\rho,\Omega)}\lesssim \|\tilde{\yy}-\bar{\yy}_\T\|_{\LL^\infty(\Omega)}.
\end{equation}
To control the right hand side of the previous expression, we use the triangle inequality to obtain that
$$\|\tilde{\yy}-\bar{\yy}_\T\|_{\LL^\infty(\Omega)}
\leq
\|\tilde{\yy}-\hat{\yy}\|_{\LL^\infty(\Omega)}+\|\hat{\yy}-\bar{\yy}_\T\|_{\LL^\infty(\Omega)};$$ 
the pair $(\hat{\yy},\hat{p})\in\HH_0^1(\Omega)\times L^2(\Omega)/\mathbb{R}$ solves \eqref{eq:state_hat}. The results of Theorem \ref{thm:high_int} guarantee the existence of $l>d$ such that $\|\tilde{\yy}-\hat{\yy}\|_{\LL^\infty(\Omega)}\lesssim \|\tilde{\yy}-\hat{\yy}\|_{\WW^{1,l}(\Omega)}$. Thus,
\begin{equation}\label{eq:control_ineq_5}
\|\tilde{\yy}-\hat{\yy}\|_{\LL^\infty(\Omega)}
\lesssim 
\|\tilde{\uu}-\bar{\uu}_\T\|_{\WW^{-1,l}(\Omega)} 
\lesssim 
\|\tilde{\uu}-\bar{\uu}_\T\|_{\LL^2(\Omega)} = \mathcal{E}_{ct}.
\end{equation}
Now, since $\bar{\yy}_\T$ is the Galerkin approximation of $\hat{\yy}$, the term $\|\hat{\yy}-\bar{\yy}_\T\|_{\LL^\infty(\Omega)}$ is estimated by invoking the global reliability of the a posteriori error estimator $\mathcal{E}_{st}$ defined in \eqref{def:estimator_velocity}: 
$\|\hat{\yy}-\bar{\yy}_\T\|_{\LL^\infty(\Omega)}
\lesssim
\ell^{\beta_d}_\T\mathcal{E}_{st}.
$
Replacing the obtained estimates into \eqref{eq:control_ineq_4}, we obtain that
\begin{equation}\label{eq:control_ineq_6}
\|\tilde{\zz}-\hat{\zz}\|_{\LL^2(\Omega)}
\lesssim
\mathcal{E}_{ct} + \ell_\T^{\beta_d}\mathcal{E}_{st}.
\end{equation}
This, in light of \eqref{eq:control_ineq_3}, implies that
\begin{equation}\label{eq:control_ineq_7}
\|\bar{\uu}-\tilde{\uu}\|_{\LL^2(\Omega)}^2 \lesssim \mathcal{E}_{ct}^2 +\ell_\T^{2\beta_d}\mathcal{E}_{st}^2  +\mathcal{E}_{ad}^2,
\end{equation}
which, in view of \eqref{eq:control_ineq_1}, allows us to conclude the a posteriori error estimate
\begin{equation}\label{eq:control_ineq_8}
\|\bar{\uu}-\bar{\uu}_\T\|_{\LL^2(\Omega)}^2 \lesssim (1+\ell_\T^{2\beta_d})\mathcal{E}_{ocp}^2.
\end{equation}

\emph{Step 2.} The goal of this step is to bound the error $\|\bar{\yy}-\bar{\yy}_\T\|_{\LL^\infty(\Omega)}$. To accomplish this task, we write $\bar{\yy}-\bar{\yy}_\T=(\bar{\yy}-\hat{\yy})+(\hat{\yy}-\bar{\yy}_\T)$, and estimate each term separately. To control the first term we invoke a similar argument to the one that yields \eqref{eq:control_ineq_5}:
\begin{equation}\label{eq:state_ineq_1}
\|\bar{\yy}-\hat{\yy}_\T\|_{\LL^\infty(\Omega)}
\lesssim 
\|\bar{\yy}-\hat{\yy}_\T\|_{\WW^{1,l}(\Omega)}
\lesssim
\|\bar{\uu}-\bar{\uu}_\T\|_{\LL^2(\Omega)},
\end{equation}
which can be directly bound with the use of \eqref{eq:control_ineq_8}. On the other hand, by using the global reliability of the error estimator $\mathcal{E}_{st}$ we arrive at $\|\hat{\yy}-\bar{\yy}_\T\|_{\LL^\infty(\Omega)}\lesssim \ell_\T^{\beta_d}\mathcal{E}_{st}$. The collection on the previous results yield
\begin{equation}\label{eq:state_ineq_2}
\|\bar{\yy}-\bar{\yy}_\T\|_{\LL^\infty(\Omega)}^2
\lesssim
(1+\ell_\T^{2\beta_d})\mathcal{E}_{ocp}^2.
\end{equation}

\emph{Step 3.} We bound the error $\nabla(\bar{\zz}-\bar{\zz}_\T)$ in the $\LL^2(\rho,\Omega)$--norm. A simple application of the triangle inequality yields
\begin{equation}\label{eq:adj_ineq_1}
\|\nabla(\bar{\zz}-\bar{\zz}_\T)\|_{\LL^2(\rho,\Omega)}
\leq
\|\nabla(\bar{\zz}-\hat{\zz})\|_{\LL^2(\rho,\Omega)} + \|\nabla(\hat{\zz}-\bar{\zz}_\T)\|_{\LL^2(\rho,\Omega)}.
\end{equation}
The first term on the right--hand side of the previous expression can be bounded in view of the stability estimate \eqref{eq:a_priori_estimates} and \eqref{eq:state_ineq_2}. In fact,
\begin{equation}\label{eq:adj_ineq_2}
\|\nabla(\bar{\zz}-\hat{\zz})\|_{\LL^2(\rho,\Omega)}
\lesssim
\|\bar{\yy}-\bar{\yy}_\T\|_{\LL^\infty(\Omega)} 
\lesssim
(1+\ell_\T^{\beta_d})\mathcal{E}_{ocp}.
\end{equation}
To control $\|\nabla(\hat{\zz}-\bar{\zz}_\T)\|_{\LL^2(\rho,\Omega)}$, we resort to the global reliability of the error estimator $\mathcal{E}_{ad}$: $\|\nabla(\hat{\zz}-\bar{\zz}_\T)\|_{\LL^2(\rho,\Omega)}
\lesssim
\mathcal{E}_{ad}$. With this estimate at hand, we thus replace \eqref{eq:adj_ineq_2} into \eqref{eq:adj_ineq_1} to obtain that
\begin{equation}\label{eq:adj_ineq_3}
\|\nabla(\bar{\zz}-\bar{\zz}_\T)\|_{\LL^2(\rho,\Omega)}^2
\lesssim 
(1+\ell_\T^{2\beta_d})\mathcal{E}_{ocp}^2.
\end{equation}

\emph{Step 4.} The goal of this step is to bound the term $\|\bar{r}-\bar{r}_\T\|_{L^2(\rho,\Omega)/\mathbb{R}}$. We write $\bar{r}-\bar{r}_\T=(\bar{r}-\hat{r})+(\hat{r}-\bar{r}_\T)$, and immediately notice that
\eqref{eq:a_priori_estimates} and \eqref{eq:state_ineq_2} yield
\begin{equation}\label{eq:adj_pressure_ineq_1}
\|\bar{r}-\hat{r}\|_{L^2(\rho,\Omega)/\mathbb{R}}
\lesssim
\|\bar{\yy}-\bar{\yy}_\T\|_{\LL^\infty(\Omega)}
\lesssim
(1+\ell_\T^{\beta_d})\mathcal{E}_{ocp}.
\end{equation}
We now invoke the global reliability of the error estimator $\mathcal{E}_{ad}$: $\|\hat{r}-\bar{r}_\T\|_{L^2(\rho,\Omega)/\mathbb{R}}
\lesssim
\mathcal{E}_{ad}$. The collection of our derived results allow us to arrive at
\begin{equation}\label{eq:adj_pressure_ineq_2}
\|\bar{r}-\bar{r}_\T\|_{L^2(\rho,\Omega)/\mathbb{R}}
\lesssim 
(1+\ell_\T^{\beta_d})\mathcal{E}_{ocp}.
\end{equation}

\emph{Step 5.} To obtain the estimate \eqref{global_reliability_2}, we must estimate the term $\|\bar{p}-\bar{p}_\T\|_{L^2(\Omega)}$.  To accomplish this task, we write $\bar{p}-\bar{p}_\T=(\bar{p}-\hat{p})+(\hat{p}-\bar{p}_\T)$ and estimate each term separately. To estimate the first term, we use \eqref{eq:standard_estimate_stokes} and \eqref{eq:control_ineq_8} to obtain that
\begin{equation}\label{eq:pressure_ineq_1}
\|\bar{p}-\hat{p}\|_{L^2(\Omega)}\lesssim \|\bar{\uu}-\bar{\uu}_\T\|_{\LL^2(\Omega)} \lesssim (1+\ell_\T^{\beta_d})\mathcal{E}_{ocp}.
\end{equation}
To estimate $(\hat{p}-\bar{p}_\T)$, we invoke the global reliability property \eqref{eq:state_estimate_press} of the error estimator $E_{st}$ to obtain that $\|\hat{p}-\bar{p}_\T\|_{L^2(\Omega)}\lesssim  E_{st}$. We thus collect the derived estimates to obtain that
\begin{equation}\label{eq:pressure_ineq_2}
\|\bar{p}-\bar{p}_\T\|_{L^2(\Omega)}
\lesssim 
(1+\ell_\T^{\beta_d})\mathcal{E}_{ocp}.
\end{equation}

\emph{Step 6.} The collection of the estimates \eqref{eq:control_ineq_8}, \eqref{eq:state_ineq_2}, \eqref{eq:adj_ineq_3}, \eqref{eq:adj_pressure_ineq_2} and \eqref{eq:pressure_ineq_2} yield the desired estimate \eqref{global_reliability_2}. 
\end{proof}

\subsection{A posteriori error estimator: efficiency}\label{sec:efficiency}

In what follows we examine the efficiency properties of the a posteriori error estimator $\mathcal{E}_{ocp}$, which is defined as in \eqref{def:estimator_ocp}. To accomplish this task, we analyze each of its contributions separately.

\subsubsection{Efficiency properties of $\mathcal{E}_{st}(\bar{\yy}_\T,\bar{p}_\T,\bar{\uu}_\T)$}\label{sec:efficiency_st_y}

%

We begin by introducing the error equation associated to the state equations \eqref{eq:weak_pde}. Let us consider $\vv \in \HH_0^1(\Omega)$ which is such that $\vv|_{T} \in \mathbf{C}^2(T)$ for all $T \in \T$. Following similar arguments the ones that yield \eqref{effic_rhs} we obtain, from the \emph{momentum equation} in \eqref{eq:weak_pde}, that
\begin{multline}\label{eq:erroreq_state}
-\sum_{T \in \T}\bigg((\mathbf{e}_{\yy},\Delta \vv)_{\LL^{2}(T)} + (e_{p},\text{div} \ \vv)_{L^{2}(T)}\bigg) - \sum_{S \in \Sides} (\mathbf{e}_{\yy},[\![\nabla \vv\cdot \boldsymbol \nu]\!])_{\LL^{2}(S)} \\
=
\sum_{T \in \T} \bigg(( \bar \uu_{\T} + \Delta \bar \yy_\T - \nabla \bar p_{\T},\vv)_{\LL^{2}(T)} + (\mathbf{e}_{\uu},\vv)_{\LL^{2}(T)}\bigg) + \sum_{S \in \Sides}([\![\nabla \bar \yy_\T \cdot \boldsymbol\nu]\!],\vv)_{\LL^{2}(S)}.
\end{multline}
On the basis of this error equation, we proceed to obtain the following efficiency result.

\begin{theorem}[local efficiency of $\mathcal{E}_{st}$]\label{thm:local_efficiency_maximum_2}
Let $(\bar{\yy},\bar{p},\bar{\zz},\bar{r},\bar{\uu})\in \HH_0^1(\Omega)\times L^2(\Omega)/\mathbb{R} \times \\ \HH_0^1(\rho,\Omega)\times L^2(\rho,\Omega)/\mathbb{R}\times \mathbb{U}_{ad}$ be the solution to the optimality system \eqref{eq:weak_pde}, \eqref{eq:adj_eq}, and \eqref{eq:variational_ineq}, and $(\bar{\yy}_\T,\bar{p}_\T,\bar{\zz}_\T,\bar{r}_\T,\bar{\uu}_\T)\in \mathbf{V}(\T)\times Q(\T)\times\mathbf{V}(\T)\times Q(\T)\times \mathbb{U}_{ad}(\T)$ its numerical approximation given by \eqref{eq:discrete_state_eq}--\eqref{eq:discrete_adj_eq}. If $\Omega\subset \mathbb{R}^2$, then, for $T\in \T$, the local error indicator $\mathcal{E}_{st,T}(\bar{\yy}_\T,\bar{p}_\T,\bar{\uu}_\T)$, defined in \eqref{def:indicator_st_1}, satisfies that
\begin{equation}\label{eq:efficiency_estimate_2}
\mathcal{E}_{st,T}(\bar{\yy}_\T,\bar{p}_\T,\bar{\uu}_\T)
\lesssim
\|\mathbf{e}_{\yy}\|_{\LL^\infty(\mathcal{N}_T)}  +  h_{T}^{1-\frac{d}{2}}\|e_{p}\|_{L^2(\mathcal{N}_T)} + h_{T}^{2-\frac{d}{2}}\|\mathbf{e}_{\uu}\|_{\LL^{2}(\mathcal{N}_T)},
\end{equation}
where $\mathcal{N}_T$ is defined as in \eqref{eq:patch}. The hidden constant is independent of the continuous and discrete solutions, the size of the elements in the mesh $\T$ and
$\#\T$.
\end{theorem}
\begin{proof} 
The proof closely follows the arguments developed in the proof of Theorem \ref{thm:local_efficiency_maximum}. On the basis of \eqref{def:indicator_st_1} and \eqref{eq:erroreq_state}, we proceed in three steps.

\emph{Step 1.} Let $T \in \T$. We estimate the term $h_T^2\|\bar{\uu}_\T + \Delta \bar{\yy}_{\T} - \nabla \bar{p}_{\T}\|_{\LL^\infty(T)}$ in \eqref{def:indicator_st_1}. To accomplish this task, and in order to simplify the presentation of the material, we define $\mathfrak{R}_{T} := (\bar{\uu}_\T + \Delta \bar{\yy}_{\T} - \nabla \bar{p}_{\T})|_T$. To derive the desired bound, we set $\vv=\varphi^{2}_{T}\mathfrak{R}_{T}$ in \eqref{eq:erroreq_state} and invoke Hölder's inequality. In fact, we have that
\begin{multline}
\label{eq:estimate_st_1}
\|\varphi_{T}\mathfrak{R}_{T}\|_{\LL^{2}(T)}^{2} 
\lesssim
\|\mathbf{e}_\uu\|_{\LL^{2}(T)}\|\varphi^{2}_{T}\mathfrak{R}_{T}\|_{\LL^{2}(T)}
\\
+ \|\mathbf{e}_\yy\|_{\LL^\infty(T)}\|\Delta(\varphi^{2}_{T}\mathfrak{R}_{T})\|_{\LL^{1}(T)}
+\|e_p\|_{L^2(T)}\|\text{div}(\varphi^{2}_{T}\mathfrak{R}_{T})\|_{L^{2}(T)}.
\end{multline}
We now use the properties of the function $\varphi_T$ that allowed us to conclude \eqref{eq:estim1_R} and \eqref{eq:estim3_R} to arrive at
\begin{equation}
\label{eq:estimate_st_2}
h_T^2\|\mathfrak{R}_{T}\|_{\LL^{2}(T)}
\lesssim
h_T^2\|\mathbf{e}_\uu\|_{\LL^{2}(T)}
+ h_T^{\frac{d}{2}}\|\mathbf{e}_\yy\|_{\LL^\infty(T)}
+h_T\|e_p\|_{L^2(T)}.
\end{equation}
Finally, we use the inverse estimate $\|\mathfrak{R}_{T}\|_{\LL^{\infty}(T)} \lesssim h_T^{-\frac{d}{2}}\|\mathfrak{R}_{T}\|_{\LL^2(T)}$ to conclude that
\begin{equation}
\label{eq:estimate_st_3}
h_T^2\|\mathfrak{R}_{T}\|_{\LL^{\infty}(T)}
\lesssim
h_T^{2-\frac{d}{2}}\|\mathbf{e}_\uu\|_{\LL^{2}(T)} + \|\mathbf{e}_\yy\|_{\LL^\infty(T)} + h_T^{1-\frac{d}{2}}\|e_p\|_{L^2(T)}.
\end{equation}

\emph{Step 2.} Let $T\in \T$ and $S\in \mathscr{S}_T$. Our objective is to bound the jump term $h_T\|[\![\nabla \bar{\yy}_\T\cdot \boldsymbol\nu]\!]\|_{\LL^\infty(\partial T \setminus \partial\Omega)}$ in \eqref{def:indicator_st_1}. To accomplish this task, we invoke the vector--valued bubble function $\boldsymbol{\varphi}_S$ described in the proof of Theorem \ref{thm:local_efficiency_maximum}, set $\vv=\boldsymbol{\varphi}_S$ in \eqref{eq:erroreq_state}, and use the properties \eqref{eq:properties_bubbleS} to conclude that
\begin{multline}\label{edge_estimate_1}
\int_{S} [\![\nabla \bar \yy_\T \cdot \boldsymbol\nu]\!]\boldsymbol\varphi_{S} 
\lesssim
\sum_{T' \in \mathcal{N}_S}\bigg( \|\mathbf{e}_\yy\|_{\LL^\infty(T')}\|\Delta \boldsymbol\varphi_S\|_{\LL^{1}(T')}\\
+\|e_p\|_{L^2(T')}\|\text{div }\boldsymbol\varphi_S\|_{L^{2}(T')}
+
\|\mathfrak{R}_{T}\|_{\LL^{\infty}(T')}\|\boldsymbol\varphi_S\|_{\LL^{1}(T')} + \|\mathbf{e}_\uu\|_{\LL^{2}(T')}\|\boldsymbol\varphi_S\|_{\LL^{2}(T')} \bigg).
\end{multline}
With this estimate at hand, we invoke standard arguments and the derived estimate for $\|\mathfrak{R}_{T}\|_{\LL^{\infty}(T)}$ to arrive at
\begin{multline*}
\int_{S} [\![\nabla \yy_\T\cdot \boldsymbol\nu]\!]\boldsymbol\varphi_{S} 
\lesssim
\sum_{T' \in \mathcal{N}_S}\bigg(h_T^{d-2}\|\mathbf{e}_\yy\|_{\LL^\infty(T')}+h_T^{\frac{d}{2}-1} \|e_p\|_{L^2(T')}  
+ h_{T}^{\frac{d}{2}}\|\mathbf{e}_\uu\|_{\LL^{2}(T')}\bigg).
\end{multline*}
We thus replace the previous estimate into \eqref{edge_estimate0} and use that $|T|/|S| \approx h_T$ to conclude that
\begin{multline*}
h_T \|[\![\nabla \yy_\T\cdot \boldsymbol\nu]\!]\|_{\LL^{\infty}(S)}
\lesssim
\sum_{T' \in \mathcal{N}_S}\bigg(\|\mathbf{e}_\yy\|_{\LL^\infty(T')}+h_T^{1-\frac{d}{2}}\|e_p\|_{L^2(T')}  
+ h_{T}^{2-\frac{d}{2}}\|\mathbf{e}_\uu\|_{\LL^{2}(T')}\bigg).
\end{multline*}

\emph{Step 3.} Let $T\in \T$. The goal of this step is to estimate the term $h_T\|\text{div }\bar{\yy}_\T\|_{L^\infty(T)}$ in \eqref{def:indicators_2}. To accomplish this task, we utilize the arguments that allowed us to arrive at \eqref{divergence_estimate}:
\begin{equation}\label{eq:divergence_estimate_3}
\|\text{div }\bar{\yy}_\T\|_{L^2(T)}^2
\lesssim
h_T^{\frac{d}{2}-1}\|\mathbf{e}_\yy\|_{\LL^\infty(T)}\|\text{div }\bar{\yy}_{\T}\|_{L^{2}(T)},
\end{equation}
Consequently, the use of inverse estimate yields
\begin{equation}\label{eq:divergence_estimate_4}
h_{T}\|\text{div }\bar{\yy}_\T\|_{L^\infty(T)}
\lesssim 
\|\mathbf{e}_\yy\|_{\LL^\infty(T)}.
\end{equation}

The collection of the estimates derived in \emph{Steps 1, 2} and \emph{3} yield \eqref{eq:efficiency_estimate_2}. This concludes the proof.
\end{proof}


We now investigate the efficiency properties of
the local a posteriori error indicator $E_{st,T}$ introduced in \eqref{def:indicator_st}.

\begin{theorem}[local efficiency of ${E}_{st}$]\label{thm:local_efficiency_maximum_pressure}
Let $(\bar{\yy},\bar{p},\bar{\zz},\bar{r},\bar{\uu})\in \HH_0^1(\Omega)\times L^2(\Omega)/\mathbb{R}\times \HH_0^1(\rho,\Omega)\times L^2(\rho,\Omega)/\mathbb{R}\times \mathbb{U}_{ad}$ be the solution to the optimality system \eqref{eq:weak_pde}, \eqref{eq:adj_eq} and \eqref{eq:variational_ineq} and $(\bar{\yy}_\T,\bar{p}_\T,\bar{\zz}_\T,\bar{r}_\T,\bar{\uu}_\T)\in \mathbf{V}(\T)\times Q(\T)\times\mathbf{V}(\T)\times Q(\T)\times \mathbb{U}_{ad}(\T)$ its numerical approximation given by \eqref{eq:discrete_state_eq}--\eqref{eq:discrete_adj_eq}. Then, for $T\in \T$, the local error indicator $E_{st,T}(\bar{\yy}_\T,\bar{p}_\T,\bar{\uu}_\T)$, defined in \eqref{def:indicator_st}, satisfies that
\begin{equation}\label{eq:efficiency_estimate_pressure}
{E}_{st,T}^2(\bar{\yy}_\T,\bar{p}_\T,\bar{\uu}_\T)
\lesssim
h_T^{d-2}\|\mathbf{e}_{\yy}\|_{\LL^\infty(\mathcal{N}_T)}^2  +  \|e_{p}\|_{L^2(\mathcal{N}_T)}^2 + h_{T}^{2}\|\mathbf{e}_{\uu}\|_{\LL^{2}(\mathcal{N}_T)}^2,
\end{equation}
where $\mathcal{N}_T$ is defined as in \eqref{eq:patch}. The hidden constant is independent of the solution, its  approximation, the size of the elements in the mesh $\T$ and
$\#\T$.
\end{theorem}
\begin{proof}
The control of the terms $h_T^2\|\bar{\uu}_\T+\Delta \bar{\yy}_\T-\nabla \bar{p}_\T\|_{\LL^2(T)}^2$ and $\|\text{div }\bar{\yy}_\T\|_{L^2(T)}^2 $ in \eqref{def:indicator_st} follow directly from the estimates \eqref{eq:estimate_st_2} and \eqref{eq:divergence_estimate_3}, respectively. 

We proceed to estimate the remaining term $h_T\|[\![\nabla \bar{\yy}_\T \cdot \boldsymbol\nu]\!]\|_{\LL^2(\partial T\setminus \partial \Omega)}^2$ in \eqref{def:indicator_st}. To accomplish this task, we set $\vv=[\![\nabla \bar{\yy}_\T\cdot \boldsymbol\nu]\!]\varphi_S$ in \eqref{eq:erroreq_state} and invoke \eqref{eq:properties_bubbleS} and standard bubble functions arguments to conclude that
\begin{multline}\label{edge_estimate_2}
\|[\![\nabla \bar{\yy}_\T\cdot \boldsymbol\nu]\!]\varphi_S^{\frac{1}{2}}\|_{\LL^2(S)}^2
\lesssim
\sum_{T' \in \mathcal{N}_S}\bigg( h_T^{\frac{d}{2}-2}\|\mathbf{e}_\yy\|_{\LL^\infty(T')} + h_T^{-1}\|e_p\|_{L^2(T')}\\
+ \|\mathfrak{R}_{T}\|_{\LL^{2}(T')} + \|\mathbf{e}_\uu\|_{\LL^{2}(T')} \bigg)h_T^{\frac{1}{2}}\|[\![\nabla \bar{\yy}_\T\cdot \boldsymbol\nu]\!]\|_{\LL^2(S)},
\end{multline}
where $\mathfrak{R}_{T}=(\bar{\uu}_\T + \Delta \bar{\yy}_{\T} - \nabla \bar{p}_{\T} )|_T$. With this estimate at hand, we invoke standard arguments and the derived estimate for $\|\mathfrak{R}_T\|_{\LL^2(T)}$ to arrive at 
\begin{multline}\label{edge_estimate_3}
h_T\|[\![\nabla \bar{\yy}_\T\cdot \boldsymbol\nu]\!]\|_{\LL^2(S)}^2
\lesssim
\sum_{T' \in \mathcal{N}_S}\bigg( h_T^{d-2}\|\mathbf{e}_\yy\|_{\LL^\infty(T')}^2
\\
+\|e_p\|_{L^2(T')}^2 + h_T^2\|\mathbf{e}_\uu\|_{\LL^{2}(T')}^2 \bigg).
\end{multline}
This concludes the proof.
\end{proof}


\subsubsection{Efficiency properties of $\mathcal{E}_{ad}(\bar{\zz}_\T,\bar{r}_\T,\bar{\yy}_\T)$}\label{sec:efficiency_ad}
To derive efficiency properties for the local error indicator $\mathcal{E}_{ad,T}(\bar{\zz}_\T,\bar{r}_\T,\bar{\yy}_\T)$, defined in \eqref{def:indicator_ad}, we utilize the standard residual estimation
techniques developed in \cite[Section 5.3.2]{Allendes_et_al2017} but on the basis of suitable bubble functions whose construction we owe to \citep[Section 5.2]{MR3264365}; see also \cite[Section  5.1.2]{Allendes_et_al2017_2}.

Given $T\in\T$, we consider a bubble function $\psi_T$ which is such that
\begin{equation}\label{eq:bubble_morin_T}
\psi_T(t)=0 \ \  \forall \: t\in \mathcal{D},
\quad
0 \leq \psi_T\leq 1,
\quad
|T|\lesssim \int_T \psi_T,
\quad
\|\nabla\psi_T\|_{L^\infty(R_T)}\lesssim h_T^{-1},
\end{equation}
and there exits a simplex $T_* \subset T$ such that $R_T:= \text{supp}(\psi_T) \subset T_{*}$. Notice that, in light of \eqref{eq:patch_property}, there is at most one $t\in \mathcal{D}$ for each element $T$.  As a consequence of \eqref{eq:bubble_morin_T}, we have, for every $g\in \mathbb{P}_2(R_T)$, that
\begin{equation}\label{eq:standard_bubble_property_morin}
\|g\|_{L^2(R_T)} 
\lesssim 
\|\psi_T^{\frac{1}{2}} g\|_{L^2(R_T)}.
\end{equation}

On the other hand, given $S\in\Sides$, we introduce a bubble function $\psi_S$ that satisfies the following properties: $0 \leq \psi_S\leq 1$,
\begin{equation}\label{eq:bubble_morin_S}
\psi_S(t)=0 \ \  \forall \: t\in \mathcal{D},
\quad
|S|\lesssim \int_S \psi_S,
\quad
\|\nabla\psi_S\|_{L^\infty(R_S)}\lesssim h_T^{-\frac{1}{2}}|S|^\frac{1}{2},
\end{equation}
where $R_S:= \text{supp}(\psi_S)$ is such that, if $\mathcal{N}_S=\{T,T'\}$, there exist two simplices $T_{*}\subset T$ and $T_*'\subset T'$ such that $R_S\subset T_*\cup T_*' \subset \mathcal{N}_S$; see \cite[Figure 1]{Allendes_et_al2017_2}. 


The following estimates are instrumental \cite[Lemma 5.2]{MR3264365}.

\begin{lemma}[estimates for bubble functions] Let $T\in\T$ and $\psi_T$ be the bubble function that satisfies \eqref{eq:bubble_morin_T}. If $\alpha\in (0,d)$, then 
\begin{equation}\label{eq:bubble_morin_1}
h_T\|\nabla(g\psi_T)\|_{L^2(\rho^{-1},T)}
\lesssim
D_T^{-\frac{\alpha}{2}}\|g\|_{L^2(T)}\quad \forall \: g\in \mathbb{P}_2(T).
\end{equation}
Let $S\in\Sides$ and $\psi_S$ be the bubble function that satisfies \eqref{eq:bubble_morin_S}. If $\alpha\in(0,d)$, then
\begin{equation}\label{eq:bubble_morin_2}
h_T^{\frac{1}{2}}\|\nabla(g\psi_S)\|_{L^2(\rho^{-1},\mathcal{N}_S)}
\lesssim
D_T^{-\frac{\alpha}{2}}\|g\|_{L^2(S)}\quad \forall \: g\in \mathbb{P}_3(S),
\end{equation}
where $g$ is extended to $\mathcal{N}_S$ as a constant along the direction of one side of each element of $\T$ contained in $\mathcal{N}_S$.
\end{lemma}

An important ingredient in the analysis that we will provide below is the so--called \emph{residual}.
To define it, we first introduce $\mathcal{Z}:=\HH_0^1(\rho^{-1},\Omega)\times L^2(\rho^{-1},\Omega)/\mathbb{R}$ and $\mathcal{W}:=\HH_0^1(\rho,\Omega)\times L^2(\rho,\Omega)/\mathbb{R}$. We define the bilinear form $d: \mathcal{Z} \times \mathcal{W} \rightarrow \mathbb{R}$ by
\[
d((\zz,r),(\ww,s)) := a(\zz,\ww)-b(\ww,r)-b(\zz,s).
\]
With these ingredients at hand, we define the residual $\mathcal{R}=\mathcal{R}(\zz_\T,r_\T) \in \mathcal{Z}'$ by
\begin{equation}\label{def:residual}
\langle  \mathcal{R}, (\ww,s)\rangle_{\mathcal{Z}',\mathcal{Z}}
=
\sum_{t\in\mathcal{D}}\langle ({\yy} - \yy_{t})\delta_{t},\mathbf{w} \rangle - d((\zz_\T,r_\T),(\ww,s)),
\end{equation}
where $\langle \cdot, \cdot \rangle_{\mathcal{Z}',\mathcal{Z}}$ denotes the duality pairing between $\mathcal{Z}'$ and $\mathcal{Z}$. We thus apply a standard integration by parts argument to conclude
\begin{multline}\label{eq:adjoint_eq_residual}
\langle  \mathcal{R}, (\ww,s)\rangle_{\mathcal{Z}',\mathcal{Z}} 
=
\sum_{t\in\mathcal{D}}\langle ({\yy} - \yy_{t})\delta_{t},\mathbf{w} \rangle + \sum_{T\in\T}(\Delta z_\T+\nabla r_\T,\ww)_{\LL^2(T)}\\
+ \sum_{S\in \Sides} ([\![\nabla z_\T \cdot \boldsymbol\nu]\!],\ww)_{\LL^2(S)} - \sum_{T\in\T}(\text{div }\zz_\T,s)_{L^2(T)},
\end{multline}
for all $(\ww,s) \in \HH_0^1(\rho^{-1},\Omega)\times L^2(\rho^{-1},\Omega)/\mathbb{R}$.

With all these ingredients at hand, we derive local efficiency properties for the local error indicator $\mathcal{E}_{ad,T}(\bar{\zz}_\T,\bar{r}_\T,\bar{\yy}_\T)$.

\begin{theorem}[local efficiency of $\mathcal{E}_{ad}$]
Let $(\bar{\yy},\bar{p},\bar{\zz},\bar{r},\bar{\uu})\in \HH_0^1(\Omega)\times L^2(\Omega)/\mathbb{R}\times \HH_0^1(\rho,\Omega)\times L^2(\rho,\Omega)/\mathbb{R}\times \mathbb{U}_{ad}$ be the solution to the optimality system \eqref{eq:weak_pde}, \eqref{eq:adj_eq} and \eqref{eq:variational_ineq} and $(\bar{\yy}_\T,\bar{p}_\T,\bar{\zz}_\T,\bar{r}_\T,\bar{\uu}_\T)\in \mathbf{V}(\T)\times Q(\T)\times\mathbf{V}(\T)\times Q(\T)\times \mathbb{U}_{ad}(\T)$ its numerical approximation given by \eqref{eq:discrete_state_eq}--\eqref{eq:discrete_adj_eq}. If $\alpha\in (d-2,d)$, then, for $T\in \T$, the local error indicator $\mathcal{E}_{ad,T}(\bar{\zz}_\T,\bar{r}_\T,\bar{\yy}_\T)$ defined in \eqref{def:indicator_ad} satisfies that
\begin{multline}\label{eq:efficiency_estimate_3}
\mathcal{E}_{ad,T}^2(\bar{\zz}_\T,\bar{r}_\T,\bar{\yy}_\T)
\\
\lesssim
\|\nabla \mathbf{e}_{\zz}\|_{\LL^2(\rho,\mathcal{N}_T^*)}^2  +  \|e_{r}\|_{L^2(\rho,\mathcal{N}_T^*)}^2 + \#(T\cap\mathcal{D})h_T^{\alpha+2-d}\|\mathbf{e}_{\yy}\|_{\LL^{\infty}(T)}^2,
\end{multline}
where $\mathcal{N}_T^*$ is defined as in \eqref{eq:patch_morin}. The hidden constant is independent of the continuous and discrete solutions, the size of the elements in the mesh $\T$ and $\#\T$.
\end{theorem}
\begin{proof}
We estimate each contribution in \eqref{def:indicator_ad} separately.

\emph{Step 1.} Let $T\in\T$. We bound $h_T^2 D_T^\alpha \|\Delta \bar{\zz}_\T+\nabla \bar{r}_\T\|_{\LL^2(T)}^2$ in \eqref{def:indicator_ad}. To accomplish this task, we define $\boldsymbol \psi_T := \psi_T(\Delta \bar{\zz}_\T+\nabla \bar{r}_\T)$ and use  \eqref{eq:standard_bubble_property_morin} 
to obtain that
\begin{equation}\label{eq:adjoint_eq_estimate_1}
\|\Delta \bar{\zz}_\T+\nabla \bar{r}_\T\|_{\LL^2(T)}^2
\lesssim
\int_{R
_T}|\Delta \bar{\zz}_\T+\nabla \bar{r}_\T|^2 \psi_T
\lesssim
(\Delta \bar{\zz}_\T+\nabla \bar{r}_\T, \boldsymbol \psi_T)_{\LL^2(T)};
\end{equation}
$\psi_T$ denotes the bubble function that satisfies \eqref{eq:bubble_morin_T}. Now, notice that, for $t\in\mathcal{D}$, we have that $\boldsymbol \psi_T(t)= \psi_T(t)(\Delta \bar{\zz}_\T+\nabla \bar{r}_\T)(t)=0$. Thus, by setting $(\ww,s)=(\boldsymbol \psi_T,0)$ in \eqref{eq:adjoint_eq_residual} we obtain that
\begin{multline}
\label{eq:adjoint_eq_estimate_2}
(\Delta \bar{\zz}_\T+\nabla \bar{r}_\T, \boldsymbol \psi_T)_{\LL^2(T)} = \langle  \mathcal{R}, (\boldsymbol \psi_T,0)\rangle_{\mathcal{Z}',\mathcal{Z}}
=a(\mathbf{e}_\zz,\boldsymbol \psi_T)-b(\boldsymbol\psi_T,e_r)\\
\lesssim \left( \| \nabla \mathbf{e}_\zz \|_{\LL^2(\rho,T)}^2+\|e_r\|_{L^2(\rho,T)}^2\right)^\frac{1}{2} \|\nabla \boldsymbol \psi_T\|_{\LL^2(\rho^{-1},T)}.
\end{multline}
In view of \eqref{eq:bubble_morin_1} we thus conclude that
\begin{equation}\label{eq:adjoint_eq_estimate_3}
\|\nabla \boldsymbol \psi_T\|_{\LL^2(\rho^{-1},T)}
\lesssim 
h_T^{-1}D_T^{-\frac{\alpha}{2}}\|\Delta \bar{\zz}_\T+\nabla \bar{r}_\T\|_{\LL^2(T)};
\end{equation}
recall that $\boldsymbol \psi_T := \psi_T(\Delta \bar{\zz}_\T+\nabla \bar{r}_\T)$. Replacing \eqref{eq:adjoint_eq_estimate_3} into \eqref{eq:adjoint_eq_estimate_2}, and the obtained one in \eqref{eq:adjoint_eq_estimate_1}, we conclude that
\begin{equation}\label{eq:adjoint_eq_estimate_4}
h_T^2 D_T^\alpha \|\Delta \bar{\zz}_\T+\nabla \bar{r}_\T\|_{\LL^2(T)}^2
\lesssim
\| \nabla \mathbf{e}_\zz \|_{\LL^2(\rho,T)}^2+\|e_r\|_{L^2(\rho,T)}^2.
\end{equation} 

\emph{Step 2.}  Let $T\in\T$ and $S\in\Sides_T$. We bound $h_T D_T^\alpha\|[\![\nabla \bar{\zz}_\T\cdot \boldsymbol\nu ]\!]\|_{\LL^2(\partial T\setminus \partial \Omega)}^2$ in \eqref{def:indicator_ad}. To accomplish this task, we first define $\boldsymbol \psi_S := \psi_S [\![\nabla \bar{\zz}_\T\cdot \boldsymbol\nu ]\!]$. The use of
\eqref{eq:bubble_morin_S} yields
\begin{equation}\label{eq:adjoint_eq_estimate_5}
\|[\![\nabla \bar{\zz}_\T\cdot \boldsymbol\nu ]\!]\|_{\LL^2(S)}^2
\lesssim
\int_{R_S} |[\![\nabla \bar{\zz}_\T\cdot \boldsymbol\nu ]\!]|^2\psi_S
=
\left( [\![\nabla \bar{\zz}_\T\cdot \boldsymbol\nu ]\!],\boldsymbol \psi_S\right)_{\LL^2(S)}.
\end{equation}
We now set $(\ww,r)=(\boldsymbol\psi_S,0)$ in \eqref{eq:adjoint_eq_residual} and recall that $\psi_S(z)=0$, for every $z\in\mathcal{D}$, and that $R_S\subset T_*\cup T_*' \subset \mathcal{N}_S$, where $R_S=\text{supp}(\psi_S)$. This yields
\begin{multline}
\left( [\![\nabla \bar{\zz}_\T\cdot \boldsymbol\nu ]\!],\boldsymbol \psi_S\right)_{\LL^2(S)}
=
\sum_{T'\in\mathcal{N}_S}(\Delta \bar{\zz}_\T+\nabla r_\T,\boldsymbol\psi_S)_{\LL^2(T')} - \langle  \mathcal{R}, (\boldsymbol\psi_S,0)\rangle_{\mathcal{Z}',\mathcal{Z}} \\
=
\sum_{T'\in\mathcal{N}_S}(\Delta \bar{\zz}_\T+\nabla r_\T,\boldsymbol\psi_S)_{\LL^2(T')} - a(\mathbf{e}_\zz,\boldsymbol \psi_S)+b(\boldsymbol\psi_S,e_r)\\
\lesssim
\sum_{T'\in\mathcal{N}_S}\|\Delta \bar{\zz}_\T+\nabla r_\T\|_{\LL^2(T')}\|\boldsymbol\psi_S\|_{\LL^2(T')}\\
+ \sum_{T'\in\mathcal{N}_S} \left( \| \nabla \mathbf{e}_\zz \|_{\LL^2(\rho,T')}^2+\|e_r\|_{L^2(\rho,T')}^2\right)^\frac{1}{2} \|\nabla \boldsymbol \psi_S\|_{\LL^2(\rho^{-1},T')}.
\label{eq:adjoint_eq_estimate_6}
\end{multline}
We use that $\|\boldsymbol\psi_S\|_{\LL^2(T')}\approx|T'|^\frac{1}{2}|S|^{-\frac{1}{2}}\|\boldsymbol\psi_S\|_{\LL^2(S)}$ and apply \eqref{eq:bubble_morin_2} to conclude that
\begin{multline}
\begin{split}
\left( [\![\nabla \bar{\zz}_\T\cdot \boldsymbol\nu ]\!],\boldsymbol \psi_S\right)_{\LL^2(S)}
\lesssim
\sum_{T'\in\mathcal{N}_S}\|\Delta \bar{\zz}_\T+\nabla r_\T\|_{\LL^2(T')}|T'|^\frac{1}{2}|S|^{-\frac{1}{2}}\|\boldsymbol\psi_S\|_{\LL^2(S)}\\
+ \sum_{T'\in\mathcal{N}_S} \left( \| \nabla \mathbf{e}_\zz \|_{\LL^2(\rho,T')}^2+\|e_r\|_{L^2(\rho,T')}^2\right)^\frac{1}{2} D_{T'}^{-\frac{\alpha}{2}}h_{T'}^{-\frac{1}{2}}\|\boldsymbol \psi_S\|_{\LL^2(S)}.
\end{split}
\label{eq:adjoint_eq_estimate_7}
\end{multline}
We thus replace \eqref{eq:adjoint_eq_estimate_7} into \eqref{eq:adjoint_eq_estimate_5} to conclude that
\begin{equation}\label{eq:adjoint_eq_estimate_8}
h_T D_T^\alpha
\|[\![\nabla \bar{\zz}_\T\cdot \boldsymbol\nu ]\!]\|_{\LL^2(S)}^2
\lesssim 
\sum_{T'\subset\mathcal{N}_S} \left( \| \nabla \mathbf{e}_\zz \|_{\LL^2(\rho,T')}^2+\|e_r\|_{L^2(\rho,T')}^2\right),
\end{equation} 
where we have also used that $|T|/|S|\approx h_T$.

\emph{Step 3.} Let $T\in\T$. We bound the term $\|\text{div }\bar{\zz}_\T\|_{L^2(\rho,T)}^2$ in \eqref{def:indicator_ad}. Since, in view of \eqref{eq:adj_eq}, $\text{div }\bar{\zz}=0$, we immediately conclude that
\begin{equation}\label{eq:adjoint_eq_estimate_9}
\|\text{div }\bar{\zz}_\T\|_{L^2(\rho,T)}^2
=
\|\text{div }\mathbf{e}_\zz\|_{L^2(\rho,T)}^2
\lesssim
\|\nabla \mathbf{e}_\zz\|_{\LL^2(\rho,T)}^2.
\end{equation}

\emph{Step 4.} Let $T\in\T$ and $t\in\mathcal{D}$. In this step we estimate the term $h_T^{\alpha+2-d}|\bar{\yy}_\T(t)-\yy_t|^2\chi(\{t\in T\})$ in \eqref{def:indicator_ad}. We begin by noticing that, if $T\cap \{t\}=\emptyset$, then the estimate \eqref{eq:efficiency_estimate_3} follows directly from the previous three steps. If, instead, $T\cap \{t\} = \{t\}$, then the element indicator $\mathcal{E}_{ad,T}$ defined in \eqref{def:indicator_ad} contains the term $h_T^{\alpha+2-d}|\bar{\yy}_\T(t)-\yy_t|^2\chi(\{t\in T\})$. If this is the case, a simple application of the triangle inequality yields
\begin{equation}\label{eq:adjoint_eq_estimate_10}
h_T^{\alpha+2-d}|\bar{\yy}_\T(t)-\yy_t|^2
\lesssim
h_T^{\alpha+2-d}|\mathbf{e}_\yy(t)|^2 + h_T^{\alpha+2-d} |\bar{\yy}(t)-\yy_t|^2.
\end{equation}
The term $h_T^{\alpha+2-d}|\mathbf{e}_\yy(t)|^2 $ is trivially bounded by $h_T^{\alpha+2-d}\|\mathbf{e}_\yy\|_{\LL^\infty(T)}^2$. 
To control the second term on the right--hand side of \eqref{eq:adjoint_eq_estimate_10}, we follow the ideas developed in the proof of \cite[Theorem 5.3]{MR3264365} that yield the existence of a smooth function $\eta$ such that
\begin{equation}\label{eq:eta_function_weight}
\eta(t)=1,\quad \|\eta\|_{L^\infty(\Omega)}=1,\quad \|\nabla\eta\|_{L^\infty(\Omega)}=h_T^{-1},\quad R_\eta:=  \text{supp}(\eta)\subset\mathcal{N}_T^*.
\end{equation} 
We now define, given $T'\in \mathcal{N}_T^*$ and $S'\in \Sides_{T'}$, $T'_\eta :=R_\eta\cap T'$ and $S'_\eta:= R_\eta \cap S'$; see Fig. \ref{fig:patch}. We also define $\ww_{\eta}:=(\bar{\yy}(t)-\yy_t)\eta\in \HH_0^1(\rho^{-1},\Omega)$. Since the pair $(\bar{\zz},\bar{r})$ solves \eqref{eq:adj_eq}, we thus have that
\begin{multline}\label{eq:adjoint_eq_estimate_11}
|\bar{\yy}(t)-\yy_t|^2=\langle(\bar{\yy}(t)-\yy_t)\delta_t,\ww_\eta\rangle=a(\bar{\zz},\ww_\eta)-b(\ww_\eta,\bar{r})\\
=
a(\mathbf{e}_\zz,\ww_\eta)-b(\ww_\eta,e_r)+a(\bar{\zz}_\T,\ww_\eta)-b(\ww_\eta,\bar{r}_\T)\\
\lesssim 
\left( \|\nabla \mathbf{e}_\zz\|_{\LL^2(\rho,R_\eta)}^2+\|e_r\|_{L^2(\rho,R_\eta)}^2\right)^{\frac{1}{2}}\|\nabla \ww_\eta\|_{\LL^2(\rho^{-1},R_\eta)}\\
+ \sum_{T'\in\mathcal{N}_T^*: T'_\eta\subset R_\eta}\|\Delta \bar{\zz}_\T+\nabla \bar{r}_\T\|_{\LL^2(T'_\eta)}\|\ww_\eta\|_{\LL^2(T'_\eta)}\\
+
\sum_{T'\in \mathcal{N}_T^*: T'_\eta\subset R_\eta}\sum_{S'_\eta\subset\partial T'_\eta:S'_\eta\not\subset \partial R_\eta}\|[\![\nabla \bar{\zz}_\T\cdot \boldsymbol\nu]\!]\|_{\LL^2(S'_\eta)}\|\ww_\eta\|_{L^2(S'_\eta)}.
\end{multline}
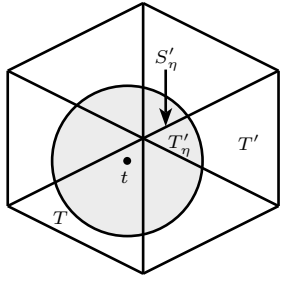
\begin{figure}[h]
\centering
\psset{xunit=0.3cm,yunit=0.3cm,algebraic=true,dimen=middle,dotstyle=o,dotsize=5pt 0,linewidth=1.6pt,arrowsize=3pt 2,arrowinset=0.25}
\begin{pspicture*}(-4.14,-3.1)(11.28,9.1)
\pscircle[linewidth=1.pt,fillcolor=gray,fillstyle=solid,opacity=0.15](2.3,2.){1}
\psline[linewidth=1.pt](-3.,0.)(3.,-3.)
\psline[linewidth=1.pt](3.,-3.)(9.,0.)
\psline[linewidth=1.pt](-3.,0.)(-3.,6.)
\psline[linewidth=1.pt](9.,0.)(9.,6.)
\psline[linewidth=1.pt](-3.,0.)(3.,3.)
\psline[linewidth=1.pt](3.,3.)(9.,0.)
\psline[linewidth=1.pt](3.,3.)(9.,6.)
\psline[linewidth=1.pt](-3.,6.)(3.,3.)
\psline[linewidth=1.pt](3.,3.)(3.,-3.)
\psline[linewidth=1.pt](3.,9.)(-3.,6.)
\psline[linewidth=1.pt](3.,9.)(9.,6.)
\psline[linewidth=1.pt](3.,9.)(3.,3.)
\psline[linewidth=1.pt]{->}(4,12ex)(4,7ex)
\begin{scriptsize}
\rput[bl](-1,-0.8){{${T}$}}
\psdots[dotsize=3pt 0,dotstyle=*](2.3,2.)
\rput[bl](2,1){$t$}
\rput[bl](4.1,2.2){$T'_\eta$}
\rput[bl](3.5,6){$S'_\eta$}
\rput[bl](7.2,2.5){$T'$}
\end{scriptsize}
\end{pspicture*}
\caption{Support $R_\eta$ of the function $\eta$ (shaded area) on the patch $\mathcal{N}_T^*$.}
\label{fig:patch}
\end{figure}

Finally, the regularity of the mesh, in conjunction with the fact that, since $t\in T$, $h_T\approx D_T$, and the estimates
\[
\|\nabla\eta\|_{L^2(\rho^{-1},R_\eta)}
\lesssim 
h_T^{\frac{d-2}{2}-\frac{\alpha}{2}},
\quad
\|\eta\|_{L^2(R_\eta)}
\lesssim
h_T^{\frac{d}{2}},
\quad
\|\eta\|_{L^2(S'_\eta)}
\lesssim
h_T^{\frac{d-1}{2}},
\]
allow us to conclude that
\begin{multline}\label{eq:adjoint_eq_estimate_12}
|\bar{\yy}(t)-\yy_t|^2
\lesssim
h_T^{\frac{d-2}{2}-\frac{\alpha}{2}}|\bar{\yy}(t)-\yy_t|\left( \|\nabla \mathbf{e}_\zz\|_{\LL^2(\rho,R_\eta)}^2+\|e_r\|_{L^2(\rho,R_\eta)}^2\right)^{\frac{1}{2}}\\
+ h_T^{\frac{d-2}{2}-\frac{\alpha}{2}}|\bar{\yy}(t)-\yy_t|\left(\sum_{T'\in\mathcal{N}_T^*: T'_\eta\subset R_\eta}h_{T}D_{T}^{\frac{\alpha}{2}}\|\Delta \bar{\zz}_\T+\nabla \bar{r}_\T\|_{\LL^2(T'_\eta)}\right.\\
\left.+
\sum_{T'\in\mathcal{N}_T^*: T'_\eta\subset R_\eta}\sum_{S'_\eta\subset\partial T'_\eta:S'_\eta\not\subset \partial R_\eta}h_{T}^{\frac{1}{2}}D_{T}^{\frac{\alpha}{2}}\|[\![\nabla \bar{\zz}_\T\cdot \boldsymbol\nu]\!]\|_{\LL^2(S'_\eta)}\right).
\end{multline}
Notice that 
 $\|\Delta \bar{\zz}_\T+\nabla \bar{r}_\T\|_{\LL^2(T'_\eta)}\lesssim \|\Delta \bar{\zz}_\T+\nabla \bar{r}_\T\|_{\LL^2({T'})}$ and $\|[\![\nabla \bar{\zz}_\T\cdot \boldsymbol\nu]\!]\|_{\LL^2(S'_\eta)}\lesssim \|[\![\nabla \bar{\zz}_\T\cdot \boldsymbol\nu]\!]\|_{\LL^2(S')}$. All these ingredients yield an estimate for $h_T^{\alpha+2-d}|\bar{\yy}_\T(t)-\yy_t|^2\chi(\{t\in T\})$.
 
A collection of the estimates \eqref{eq:adjoint_eq_estimate_4}, \eqref{eq:adjoint_eq_estimate_8}, \eqref{eq:adjoint_eq_estimate_9}, and \eqref{eq:adjoint_eq_estimate_12} yield the desired result.
\end{proof}


We conclude with the global efficiency of the error estimator $\mathcal{E}_{ocp}$.

\begin{theorem}[global efficiency property of $\mathcal{E}_{ocp}$]
\label{thm:global_efficiency}
Let $(\bar{\yy},\bar{p},\bar{\zz},\bar{r},\bar{\uu})\in \HH_0^1(\Omega)\times \\ L^2(\Omega)/\mathbb{R}\times \HH_0^1(\rho,\Omega)\times L^2(\rho,\Omega)/\mathbb{R}\times \mathbb{U}_{ad}$ be the solution to the optimality system \eqref{eq:weak_pde}, \eqref{eq:adj_eq} and \eqref{eq:variational_ineq} and $(\bar{\yy}_\T,\bar{p}_\T,\bar{\zz}_\T,\bar{r}_\T,\bar{\uu}_\T)\in \mathbf{V}(\T)\times Q(\T)\times\mathbf{V}(\T)\times Q(\T)\times \mathbb{U}_{ad}(\T)$ its numerical approximation given by \eqref{eq:discrete_state_eq}--\eqref{eq:discrete_adj_eq}. If $\Omega\subset\mathbb{R}^2$ and $\alpha\in (d-2,2)$, then 
\begin{equation}
\mathcal{E}_{ocp}^2(\bar{\zz}_\T,\bar{r}_\T,\bar{\yy}_\T,\bar{p}_\T,\bar{\uu}_\T)
\lesssim
\| \mathbf{e}  \|^2_{\Omega},
\label{global_efficiency_2}
\end{equation}
where the hidden constant is independent of the size of the elements in the mesh $\T$ and $\#\T$ but depends linearly on $\#\mathcal{D}$ and \normalfont{diam}$(\Omega)^{\alpha+2-d}$.
\end{theorem}
\begin{proof}
We invoke the 
local efficiency estimates \eqref{eq:efficiency_estimate_pressure} and \eqref{eq:efficiency_estimate_2} and, to arrive at
\begin{equation*}
{E}_{st}^2(\bar{\yy}_\T,\bar{p}_\T,\bar{\uu}_\T)
\lesssim
 \text{diam}(\Omega)^{d-2}\|\mathbf{e}_{\yy}\|_{\LL^\infty(\Omega)}^2  +  \|e_{p}\|_{L^2(\Omega)}^2 + \text{diam}(\Omega)^{2}\|\mathbf{e}_{\uu}\|_{\LL^{2}(\Omega)}^2,
 \end{equation*}
and
\begin{equation*}
\mathcal{E}_{st}^2(\bar{\yy}_\T,\bar{p}_\T,\bar{\uu}_\T)
\lesssim
\|\mathbf{e}_{\yy}\|^2_{\LL^\infty(\Omega)}  +  \text{diam}(\Omega)^{2-d}\|e_{p}\|^2_{L^2(\Omega)} + \text{diam}(\Omega)^{4-d}\|\mathbf{e}_{\uu}\|^2_{\LL^{2}(\Omega)},
\end{equation*}
respectively.

On the other hand, in view of \eqref{def:estimator_ad_velocity}, the local efficiency estimate \eqref{eq:efficiency_estimate_3} implies
\begin{equation*}
\mathcal{E}_{ad}^2(\bar{\zz}_\T,\bar{r}_\T,\bar{\yy}_\T)
\lesssim
\|\nabla \mathbf{e}_{\zz}\|_{\LL^2(\rho,\Omega)}^2  +  \|e_{r}\|_{L^2(\rho,\Omega)}^2 + \left(\sum_{T\in\T:T\cap\mathcal{D}\neq\emptyset}h_T^{\alpha+2-d}\right)\|\mathbf{e}_{\yy}\|_{\LL^{\infty}(\Omega)}^2.
\end{equation*}
Now, since $\alpha\in(d-2,2)$ and $\#\mathcal{D}<\infty$, we can conclude that
\[
\sum_{T\in\T:T\cap\mathcal{D}\neq\emptyset}h_T^{\alpha+2-d}
\leq
\#\mathcal{D}\text{ diam}(\Omega)^{\alpha+2-d}.
\]
We notice that this estimate, that is where the linear dependence on $\#\mathcal{D}$ and \normalfont{diam}$(\Omega)^{\alpha+2-d}$ comes from, is independent of $\#\T$ 

Finally, an application of the triangle inequality yields
\[
\mathcal{E}_{ct}(\bar{\zz}_\T,\bar{\uu}_\T)
\leq
\|\bar{\uu}_\T-\Pi(-\lambda^{-1}\bar{\zz})\|_{\LL^2(\Omega)} + \|\Pi(-\lambda^{-1}\bar{\zz})-\Pi(-\lambda^{-1}\bar{\zz}_\T)\|_{\LL^2(\Omega)},
\] 
where $\Pi$ is defined in \eqref{def:proj_operator}. This, in conjunction with the Lipschitz continuity of the projection operator $\Pi$ and Theorem \ref{thm:weighted_poincare}, implies that
\[
\mathcal{E}_{ct}(\bar{\zz}_\T,\bar{\uu}_\T)
\lesssim
\|\bar{\uu}_\T-\bar{\uu}\|_{\LL^2(\Omega)} + \lambda^{-1}\|\nabla (\bar{\zz}-\bar{\zz}_\T)\|_{\LL^2(\rho,\Omega)}.
\]
The proof concludes by gathering all the obtained estimates.
\end{proof}


\section{Numerical examples.}\label{sec:numerical_ex}
         
In this section we conduct a series of numerical examples that illustrate the performance of the 	devised a posteriori error estimator. These experiments have been carried out with the help of a code that we implemented using \texttt{C++}. All matrices have been assembled exactly. The right hand sides as well as the approximation errors are computed with the help of a quadrature formula that is exact for polynomials of degree $19$ for two dimensional domains and degree $14$ for three dimensional domains.  
The global linear systems were solved using the multifrontal massively parallel sparse direct solver (MUMPS) \cite{MUMPS1,MUMPS2}.
         
For a given partition $\mathscr{T}$, we seek $(\bar{\yy}^{}_\mathscr{T},\bar{p}^{}_\mathscr{T},\bar{\zz}_\T,\bar{r}^{}_\mathscr{T},\bar{\uu}^{}_\mathscr{T})\,\in \mathbf{V}(\mathscr{T})\times Q(\T)\times\mathbf{V}(\mathscr{T})\times Q(\T)\times\mathbb{U}_{ad}(\T)$ that solves the discrete optimality system \eqref{eq:discrete_state_eq}--\eqref{eq:discrete_adj_eq}. The underlying nonlinear system is solved by using the Newton--type primal--dual active set strategy of \citep[Section 2.12.4]{Troltzsch}. Once the discrete solution is obtained, we use the local error indicator $\mathcal{E}_{ocp,T}$, defined as, 
\begin{multline}\label{def:indicator_ocp}
\mathcal{E}_{ocp,T}^2(\bar{\zz}_\T,\bar{r}_\T,\bar{\yy}_\T,\bar{p}_\T,\bar{\uu}_\T) := E_{st}^{2}(\bar{\yy}_{\T},\bar{p}_{\T},\bar{\uu}_{\T})
\\
+ \mathcal{E}_{st,T}^{2}(\bar{\yy}_{\T},\bar{p}_{\T},\bar{\uu}_{\T}) + \mathcal{E}_{ad,T}^{2}(\bar{\zz}_{\T},\bar{r}_{\T},\bar{\yy}_{\T}) +\mathcal{E}_{ct,T}^{2}(\bar{\zz}_{\T},\bar{\uu}_{\T}),
\end{multline}
to drive the adaptive procedure described in \textbf{Algorithm} \ref{Algorithm} and compute the global error estimator $\mathcal{E}_{ocp}$, in order to assess the accuracy of the approximation. A sequence of adaptively refined meshes is thus generated from the initial meshes shown in Figure \ref{fig:initial_meshes}. The total number of degrees of freedom reads
\[ \mathsf{Ndof}=2\dim(\mathbf{V}(\T))+2\dim(Q(\T))+\dim(\mathbf{U}(\T)). \]
The error is measured in the norm $\|\mathbf{e}\|_\Omega$, which is defined in \eqref{def:error_norm}.
\footnotesize{
\begin{algorithm}[ht]
\caption{\textbf{Adaptive primal--dual active set algorithm.}}
\label{Algorithm}
\SetKwInput{set}{Set}
\SetKwInput{ase}{Active set strategy}
\SetKwInput{al}{Adaptive loop}
\SetKwInput{Input}{Input}
\Input{Initial mesh $\mathscr{T}_{0}$, set of observable points $\mathcal{D}$, set of desired states $\{\yy_t\}_{t\in\mathcal{D}}$, Muckenhoupt parameter $\alpha$, vector constraints $\mathbf{a}$ and $\mathbf{b}$, and regularization parameter $\lambda$.}
\set{$i=0$.}
\ase{}
Choose initial discrete guesses $\uu_{\T}^{0},\boldsymbol{\mu}_\T^0\in\mathbf{U}(\T)$ ($\uu_{\T}^{0}$ is not necessarily admissible).
\\
Compute $[\bar{\yy}^{}_\mathscr{T},\bar{p}^{}_\mathscr{T},\bar{\zz}_\T,\bar{r}_\T,\bar{\uu}^{}_\mathscr{T}]=\textbf{Active-Set}[\mathscr{T}_i,\uu_{\T}^{0},\boldsymbol\mu_\T^0, \lambda,\alpha,\mathbf{a},\mathbf{b},\mathcal{D}, \{\yy_t\}^{}_{t\in \mathcal{D}}]$. $\textbf{Active-Set}$  implements the active set strategy of \citep[Section 2.12.4]{Troltzsch}.
\\
\al
\\
For each $T\in\mathscr{T}$ compute the local error indicator $\mathcal{E}_{ocp,T}$, which is defined in \eqref{def:indicator_ocp}.
\\
Mark an element $T$ for refinement if $\mathcal{E}_{ocp,T}^{2}> \displaystyle\frac{1}{2}\max_{T'\in \mathscr{T}}\mathcal{E}_{ocp,T'}^{2}$.
\\
From step $\boldsymbol{4}$, construct a new mesh, using a longest edge bisection algorithm. Set $i \leftarrow i + 1$, and go to step $\boldsymbol{1}$.
\\
\end{algorithm}}
\normalsize
\begin{figure}[h]
\centering
\begin{minipage}{0.315\textwidth}\centering
\includegraphics[trim={0 0 0 0},clip,width=4.9cm,height=3.6cm,scale=0.4]{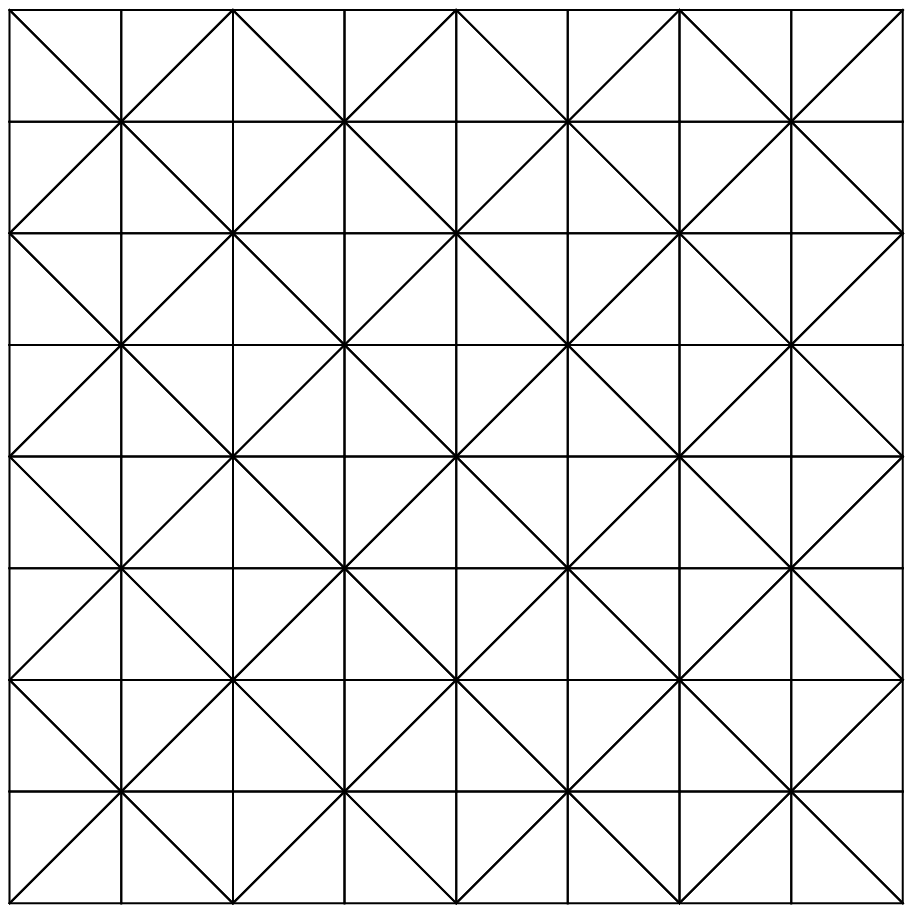}\\
\end{minipage}
\begin{minipage}{0.315\textwidth}\centering
\includegraphics[trim={0 0 0 0},clip,width=4.9cm,height=3.6cm,scale=0.4]{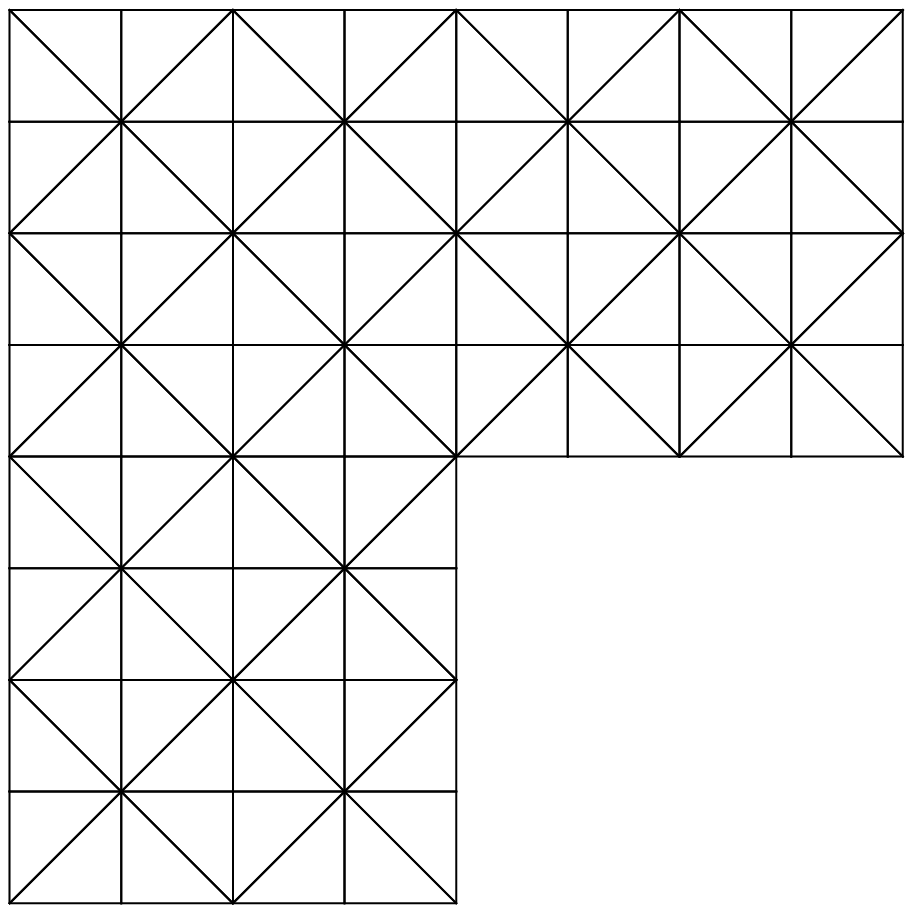}\\
\end{minipage}
\begin{minipage}{0.315\textwidth}\centering
\includegraphics[trim={0 0 0 0},clip,width=4.9cm,height=3.6cm,scale=0.4]{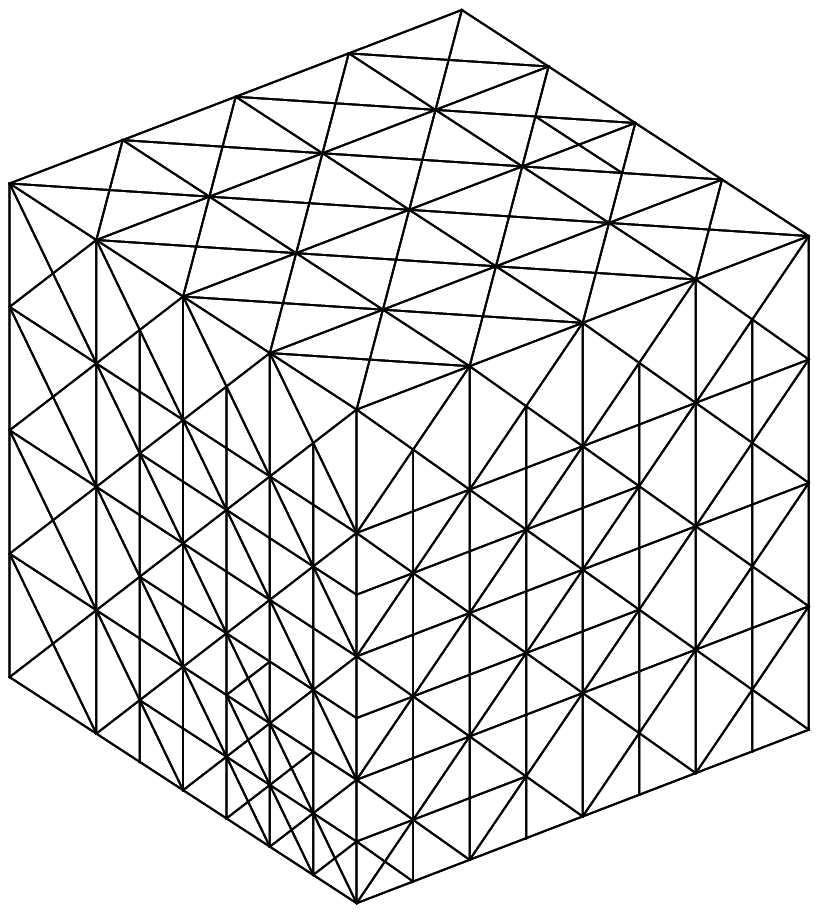}\\
\end{minipage}
\caption{The initial meshes used for Examples 1--2 (left), Example 3 (middle) and Examples 4--5 (right).}
\label{fig:initial_meshes}
\end{figure}
         
We consider problems with homogeneous boundary conditions whose exact solutions are not known. We also consider problems with inhomogeneous Dirichlet boundary conditions whose exact solutions are know. Notice that this violate the assumption of homogeneous Dirichlet boundary conditions which is needed for the analysis. In this case, we fix the optimal adjoint pair $(\bar{\zz},\bar{r})$ to be a linear combination of fundamental solutions of the Stokes equations \citep[Section IV.2]{Gal11}:
\begin{equation}\label{def:adjoint_deltas}
\bar{\zz}(\mathbf{x}):= 
\sum_{t\in\mathcal{D}}
\vartheta_{t}
\sum_{i=1}^{d}
\widetilde{\mathbf{T}}_{t}(x)\cdot \mathbf{e}_{i},\qquad 
\bar{r}(\mathbf{x}):= 
\sum_{t\in\mathcal{D}}
\vartheta_{t}
\sum_{i=1}^{d}
\mathbf{T}_{t}(x)\cdot \mathbf{e}_{i},
\end{equation}
where, if ${t} = \mathbf{x}_{t}$, $\mathbf{r}_{t} = \mathbf{x} - \mathbf{x}_{t}$, $\mathbb{I}_d$ is the identity matrix in $\mathbb{R}^{d\times d}$, then
\begin{align*}
\begin{array}{c}\displaystyle
\widetilde{\mathbf{T}}_{t}(\mathbf{x})
=
\left\{\begin{array}{ll}
-\dfrac{1}{4\pi}\bigg(\log|\mathbf{r}_{t}|\mathbb{I}^{}_2
-\dfrac{\mathbf{r}_{t}\mathbf{r}_{t}^{T}}{|\mathbf{r}_{t}|^2}
\bigg), & \text{if } \Omega \subset \mathbb{R}^{2}, \\
\dfrac{1}{8\pi}\bigg(\dfrac{1}{|\mathbf{r}_{t}|}\mathbb{I}^{}_3
+\dfrac{\mathbf{r}_{t}\mathbf{r}_{t}^{T}}{|\mathbf{r}_{t}|^3}
\bigg), & \text{if } \Omega \subset \mathbb{R}^{3},
\end{array}
\right. 
\\~\\ 
\displaystyle
\mathbf{T}_{t}(\mathbf{x})=\left\{\begin{array}{ll}
-\dfrac{\mathbf{r}_{t}}{2\pi|\mathbf{r}_{t}|^{2}}, & \text{if } \Omega \subset \mathbb{R}^{2}, \\
-\dfrac{\mathbf{r}_{t}}{4\pi|\mathbf{r}_{t}|^{3}}, & \text{if } \Omega \subset \mathbb{R}^{3};
\end{array}
\right.
\end{array}
\end{align*}
$\{ \mathbf{e}_{i} \}_{i=1}^d$ denotes the canonical basis of $\mathbb{R}^{d}$ and $\vartheta_{t}\in\mathbb{R}$ for all $t\in\mathcal{D}$. The sequence of vectors $\{\yy_{t}\}_{t \in \mathcal{D}}$ is computed from  the constructed solutions in such a way that the adjoint equations \eqref{eq:adj_eq} holds.
We finally mention that in order to simplify the construction of exact solutions, we have incorporated, in the \emph{momentum equation} of \eqref{eq:weak_pde}, an extra forcing term $\ff\in\LL^{\infty}(\Omega)$. With such a modification, the right hand side of the \emph{momentum equation} reads as follows: $(\ff+\uu,\vv)_{\LL^2(\Omega)}$.
         
\subsection{Two-dimensional examples}\label{sec:ex_2d}

We perform two dimensional examples on convex and nonconvex domains, and with different number of source points. The first two examples involve homogeneous Dirichlet boundary conditions in the state equations, but inhomogeneous Dirichlet boundary conditions in the adjoint equations. In the third example we consider homogeneous Dirichlet boundary conditions in the state and adjoint equations.         
\\~\\
\textbf{Example 1.} We let $\Omega=(0,1)^2$, $\mathcal{D}=\{(0.5,0.5)\}$,
$\mathbf{a} = (-0.5,-0.5)^T$, $\mathbf{b}=(-0.1,-0.1)^T$, and $\lambda=1$. 
The exact optimal state is
\begin{equation*}
\bar{\yy}(x_1,x_2)=\mathbf{curl}\left((\sin(2\pi x_1))^2(\sin(2\pi x_2))^2/(2\pi)\right),
\, \,
\bar{p}(x_1,x_2) = \sin(2\pi x_1)\sin(2\pi x_2),
\end{equation*}
while the exact optimal adjoint state is taken to be as in \eqref{def:adjoint_deltas} with $\vartheta_{t}=1$ for all $t\in\mathcal{D}$. It can be proved that
\begin{equation*}
\yy_{(0.5,0.5)}=\bar{\yy}(0.5,0.5)-(1,1)^T.
\end{equation*}
\\
\textbf{Example 2.} We let $\Omega=(0,1)^2$. In addition, we set $\mathbf{a} = (-0.85,-0.85)^T$, $\mathbf{b}=(-0.2,-0.2)^T$, $\lambda=1$, and
\[
\mathcal{D}=\{(0.25,0.25),(0.25,0.75),(0.75,0.25),(0.75,0.75)\}.
\]
The exact optimal state is
\begin{align*}
\bar{\yy}(x_1,x_2) & =\frac{1}{2}\mathbf{curl}\left(x_1^{2}(1-x_1)^{2}x_2^{2}(1-x_2)^{2}\right),
\\
\bar{p}(x_1,x_2) &= 50\left(x_1 - 1 + \frac{(e^{-x_1} - 1)}{(e^{-1} - 1)}\right)\left(x_2-1 + \frac{(e^{-x_2} - 1)}{(e^{-1} - 1)}\right) - \frac{25}{2}\left(\frac{e-3}{e - 1}\right)^2.
\end{align*}
while the exact optimal adjoint state is given by the linear combination of \eqref{def:adjoint_deltas} with $\vartheta_{t}=1$ for all $t\in\mathcal{D}$. It can be inferred that"t $\yy_{t}=\bar{\yy}(t)-(1,1)^{T}$ for all $t\in \mathcal{D}$.
\\~\\
\textbf{Example 3.} We let $\Omega=(0,1)^2\setminus[0.5,1)\times(0,0.5]$, $\mathbf{a} = (-0.3,-0.3)^T$, $\mathbf{b} = (0.4,0.4)^T$, $\lambda = 1$ and
\begin{align*}
\begin{array}{crc}
\mathcal{D}=\{(0.25,0.25),(0.25,0.75),(0.75,0.75)\},  &  &\\
\yy_{(0.25,0.25)}=(3,3)^T, \quad \yy_{(0.25,0.75)}=(-1,-1)^T,\quad \yy_{(0.75,0.75)}=(3,3)^T. & &
\end{array}
\end{align*}

\begin{figure}[h]
\psfrag{Example 1 - alfa 15}{\hspace{-0.8cm}$\|\mathbf{e}\|_{\Omega}$ Adaptive vs Uniform for $\alpha = 1.5$}
\psfrag{error vel ad 1}{$\|\nabla \mathbf{e}_{\zz}\|_{\LL^{2}(\rho,\Omega)}$}
\psfrag{error pr ad 1}{$\|e_{r}\|_{L^{2}(\rho,\Omega)}$}
\psfrag{error vel 111}{$\|\mathbf{e}_{\yy}\|_{\LL^{\infty}(\Omega)}$}
\psfrag{error pr 1111}{$\|e_{p}\|_{L^{2}(\Omega)}$}
\psfrag{error ct 1111}{$\|\mathbf{e}_{\uu}\|_{\LL^{2}(\Omega)}$}
\psfrag{error total 111111}{$\|(\mathbf{e}_{\yy},e_{p},\mathbf{e}_{\zz},e_{r},\mathbf{e}_{\uu})\|_{\Omega}$}
\psfrag{estimador total 11}{$\mathcal{E}_{ocp}$}
\psfrag{error total 11}{$\|(\mathbf{e}_{\yy},e_{p},\mathbf{e}_{\zz},e_{r},\mathbf{e}_{\uu})\|_{\Omega}$}
\psfrag{error ocp}{$\|\mathbf{e}\|_{\Omega}$}
\psfrag{error}{$\|\mathfrak{e}\|_{\Omega}$}
\psfrag{estimador total}{$\mathcal{E}_{ocp}$}
\psfrag{eta ocp}{$\mathcal{E}_{ocp}$}
\psfrag{eta}{$\mathcal{E}_{ocp}$}
\psfrag{eta ad}{$\mathcal{E}_{ad}$}
\psfrag{eta st vel}{$\mathcal{E}_{st}$}
\psfrag{eta st pr}{$E_{st}$}
\psfrag{eta ct}{$E_{st}$}
\psfrag{error total - reg}{$\|\mathbf{e}\|_{\Omega}$-\text{Uniform}}
\psfrag{error total - adap}{$\|\mathbf{e}\|_{\Omega}$-\text{Adaptive}}
\psfrag{Ndf-10}{\small$\textrm{Ndof}^{-1}$}
\psfrag{Ndf-25}{\small$\textrm{Ndof}^{-0.37}$}
\psfrag{Ndf-32}{\small$\textrm{Ndof}^{-3/2}$}
\psfrag{Ndofs}{$\textrm{Ndof}$}
\begin{minipage}{0.32\textwidth}\centering
\includegraphics[trim={0 0 0 0},clip,width=4.5cm,height=4.5cm,scale=0.6]{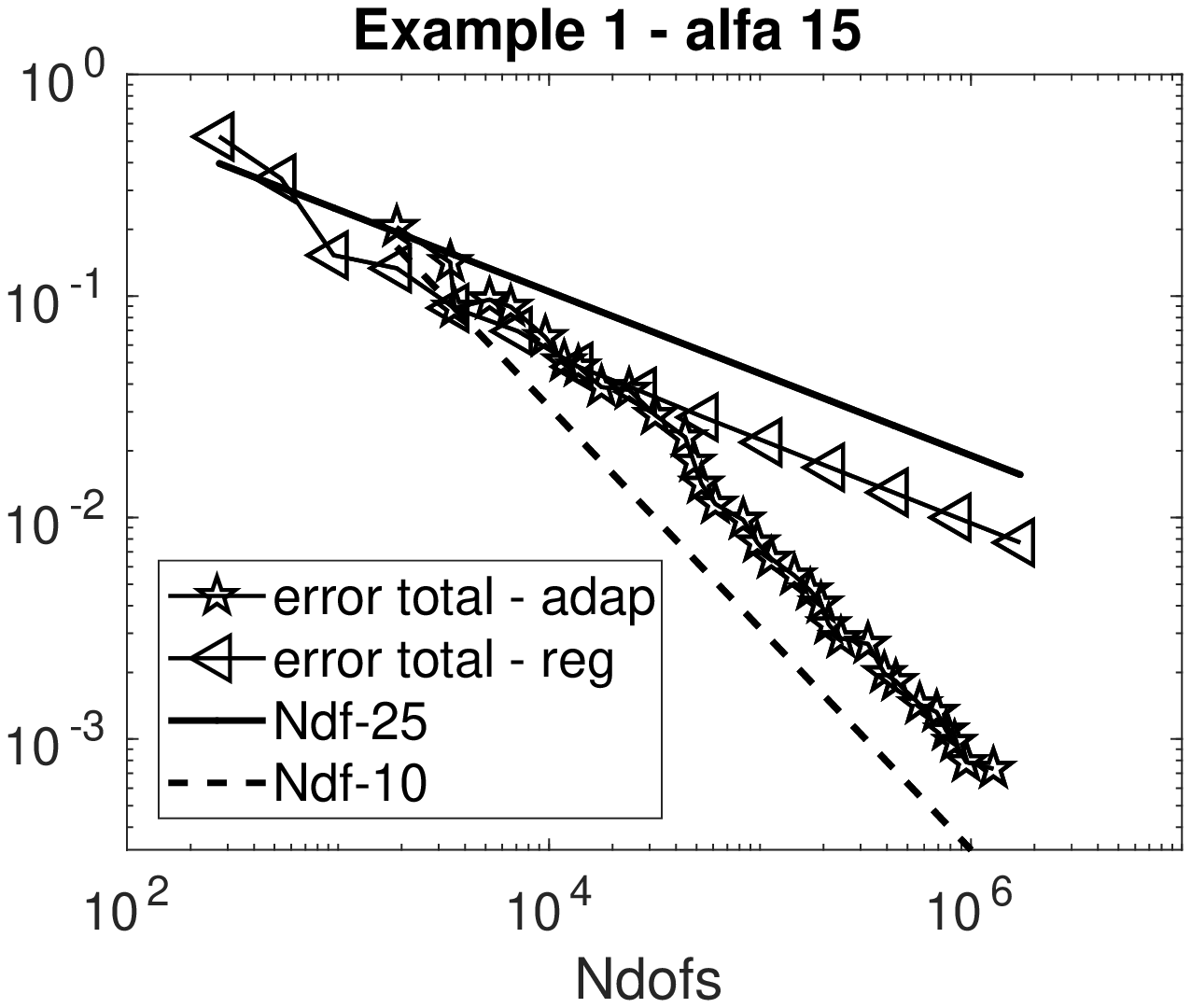}\\
\qquad \tiny{(A)}
\end{minipage}
\begin{minipage}{0.32\textwidth}\centering
\psfrag{Example 1 - alfa 15}{\hspace{-1.4cm} Error contributions for $\alpha = 1.5$  (Uniform)}
\includegraphics[trim={0 0 0 0},clip,width=4.5cm,height=4.5cm,scale=0.6]{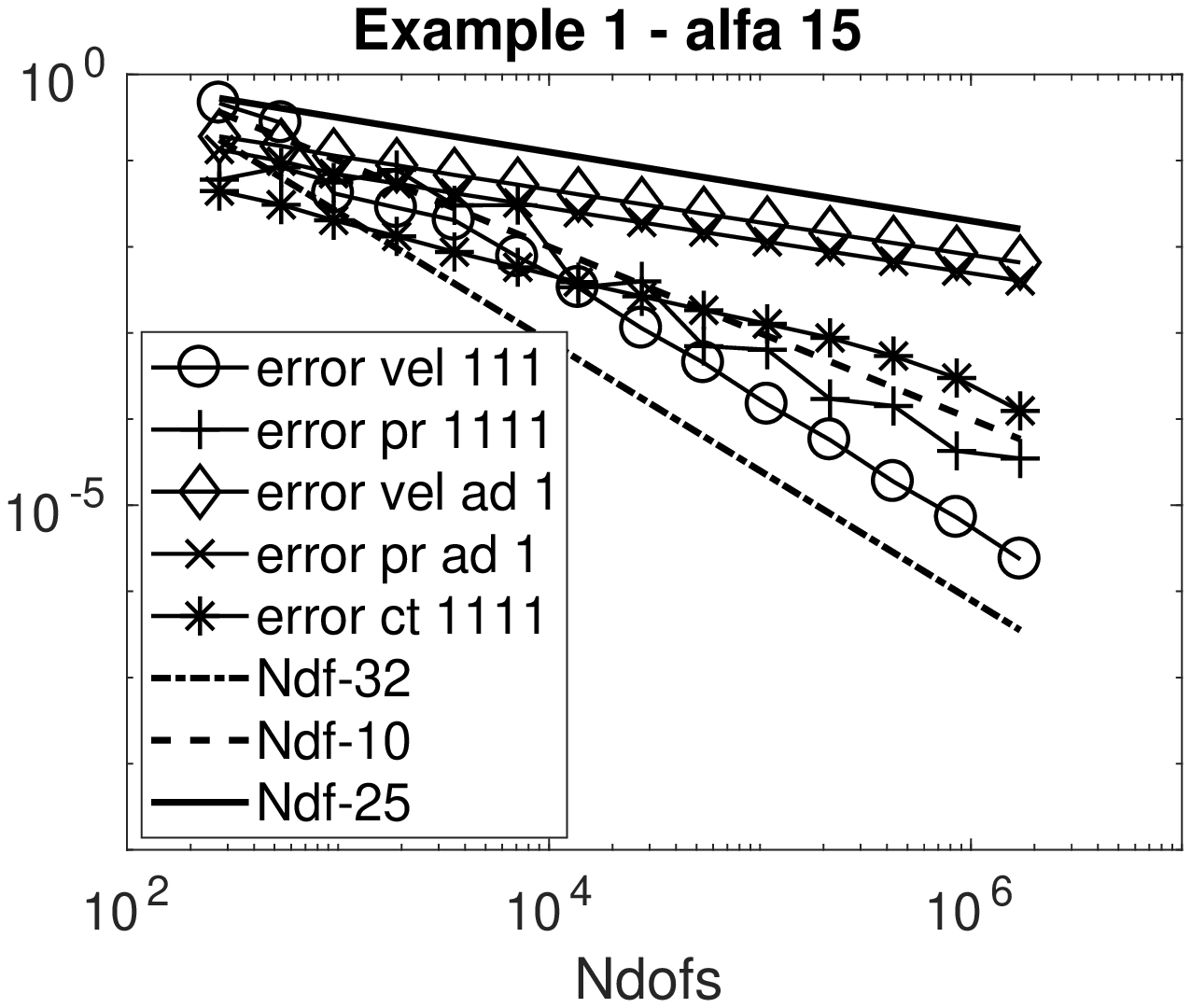}\\
\qquad \tiny{(B)}
\end{minipage}
\begin{minipage}{0.32\textwidth}\centering
\psfrag{Example 1 - errors}{\hspace{-1.6cm} Error contributions for $\alpha = 1.5$ (Adaptive)}
\includegraphics[trim={0 0 0 0},clip,width=4.5cm,height=4.5cm,scale=0.6]{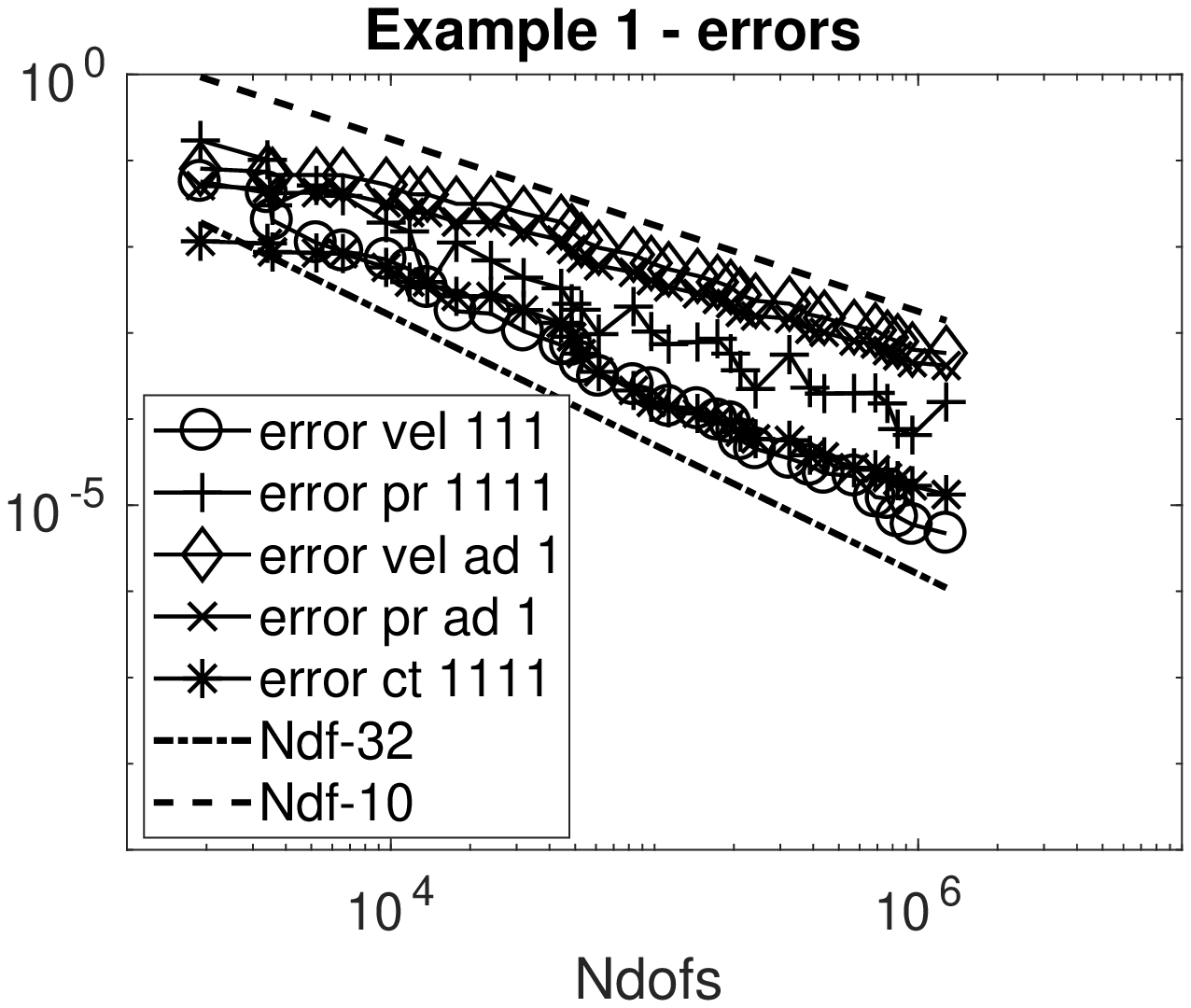}\\
\qquad \tiny{(C)}
\end{minipage}
\caption{Example 1: Experimental rates of convergence for
the total error with adaptive and uniform refinement (A), the contributions of the error with uniform refinement (B), and the contributions of the error with adaptive refinement (C), for $\alpha =1.5$.}
\label{fig:ex-1}
\end{figure}

\begin{figure}[h]
\psfrag{error vel ad 1}{$\|\nabla \mathbf{e}_{\zz}\|_{\LL^{2}(\rho,\Omega)}$}
\psfrag{error pr ad 1}{$\|e_{r}\|_{L^{2}(\rho,\Omega)}$}
\psfrag{error vel 111}{$\|\mathbf{e}_{\yy}\|_{\LL^{\infty}(\Omega)}$}
\psfrag{error pr 1111}{$\|e_{p}\|_{L^{2}(\Omega)}$}
\psfrag{error ct 1111}{$\|\mathbf{e}_{\uu}\|_{\LL^{2}(\Omega)}$}
\psfrag{estimador total 11}{$\mathcal{E}_{ocp}$}
\psfrag{error ocp}{$\|\mathbf{e}\|_{\Omega}$}
\psfrag{error}{$\|\mathbf{e}\|_{\Omega}$}
\psfrag{estimador total}{$\mathcal{E}_{ocp}$}
\psfrag{eta ocp}{$\mathcal{E}_{ocp}$}
\psfrag{eta}{$\mathcal{E}_{ocp}$}
\psfrag{eta ad}{$\mathcal{E}_{ad}$}
\psfrag{eta st vel}{$\mathcal{E}_{st}$}
\psfrag{eta st pr}{$E_{st}$}
\psfrag{eta ct}{$\mathcal{E}_{ct}$}
\psfrag{Ndf-10}{\small$\textrm{Ndof}^{-1}$}
\psfrag{Ndf-32}{\small$\textrm{Ndof}^{-{3}/{2}}$}
\psfrag{Ndofs}{$\textrm{Ndof}$}
\begin{minipage}{0.32\textwidth}\centering
\psfrag{Example 2 - alfa 199 - errors}{\hspace{-0.1cm} Error contributions for $\alpha = 1.99$}
\includegraphics[trim={0 0 0 0},clip,width=4.5cm,height=4.5cm,scale=0.6]{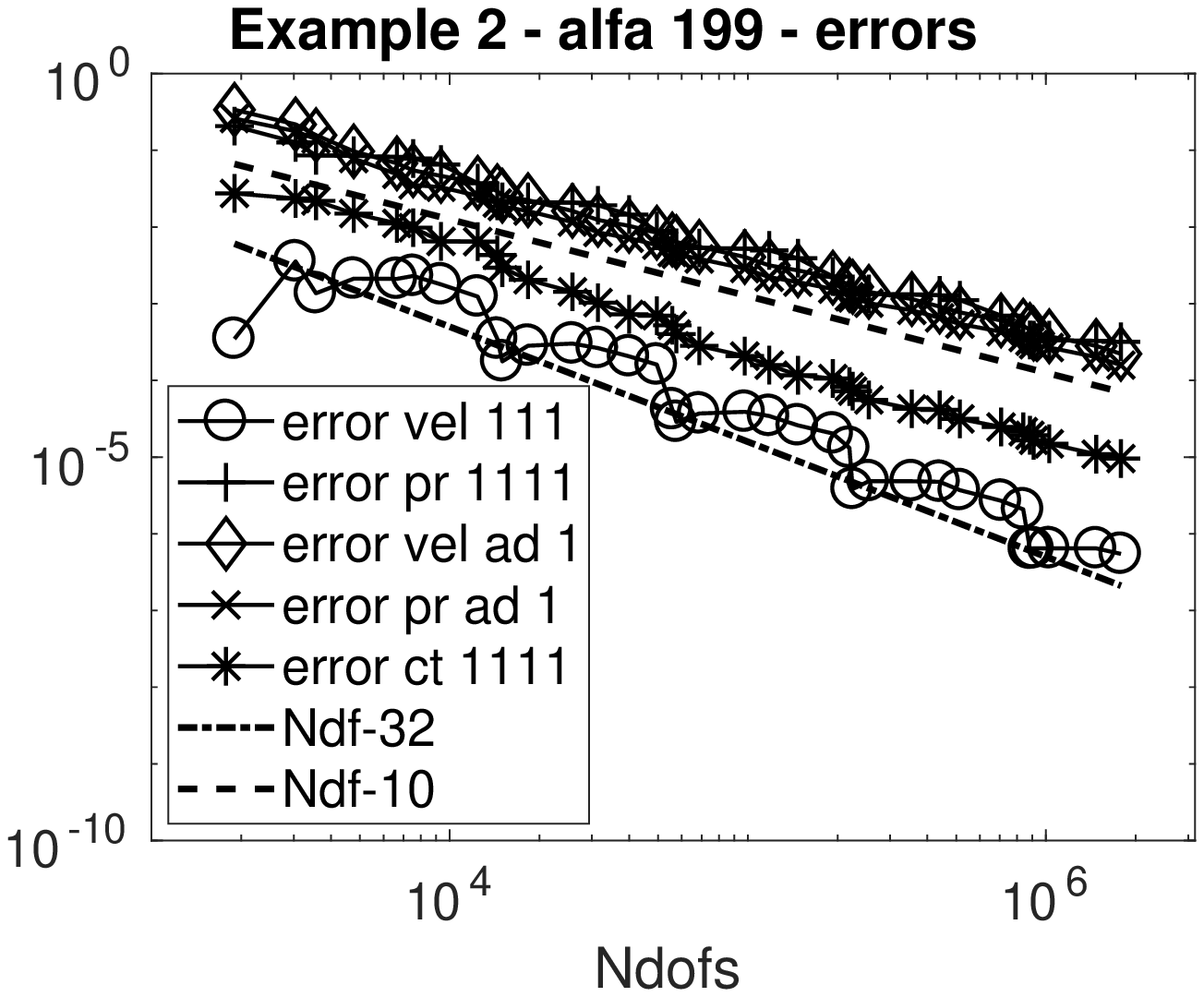}\\
\qquad \tiny{(A)}
\end{minipage}
\begin{minipage}{0.32\textwidth}\centering
\psfrag{Example 2 - alfa 199 - etas}{\hspace{-0.5cm}  Estimator contributions for $\alpha = 1.99$}
\includegraphics[trim={0 0 0 0},clip,width=4.5cm,height=4.5cm,scale=0.6]{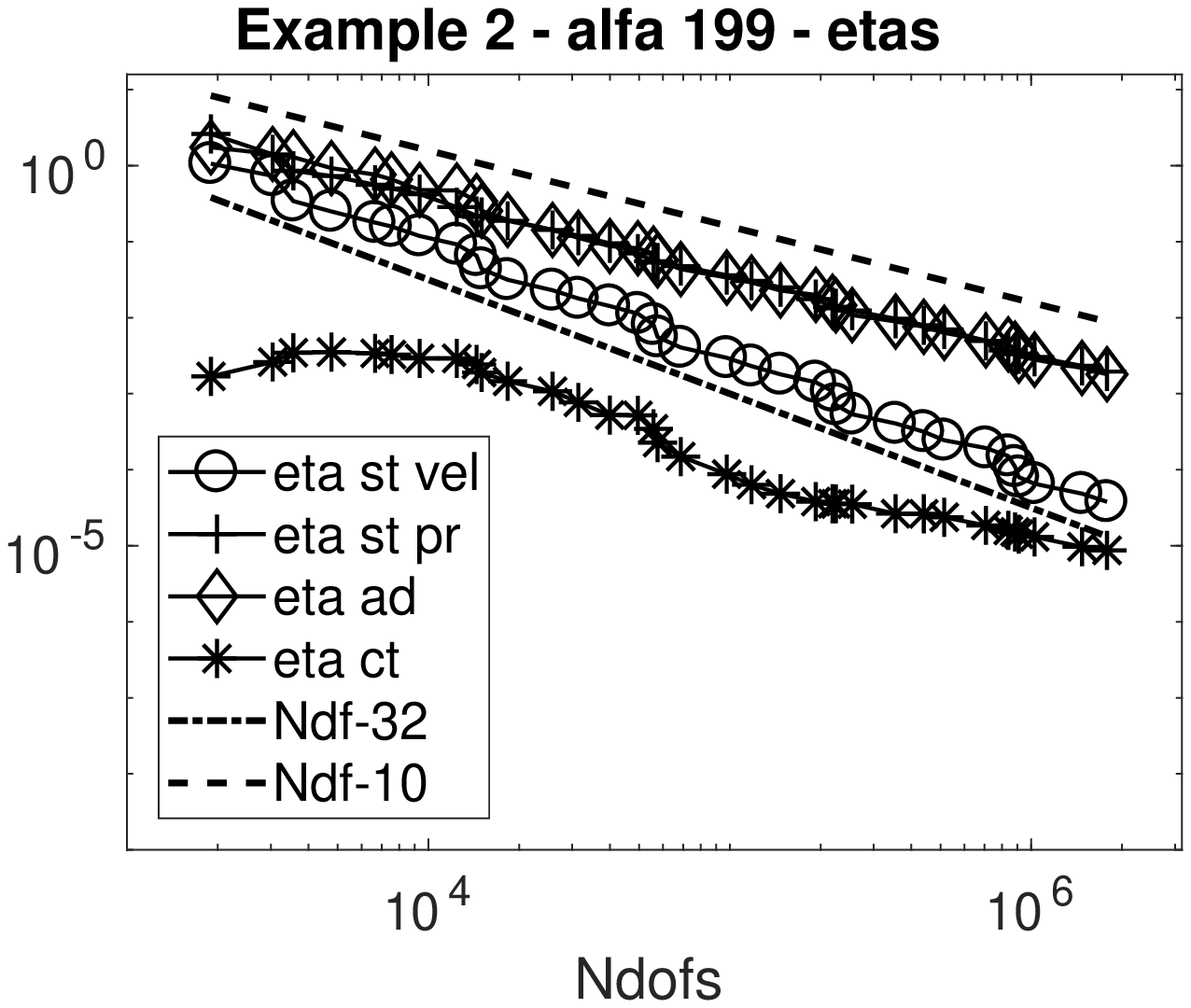}\\
\qquad \tiny{(B)}
\end{minipage}
\caption{Example 2: Experimental rates of convergence for the contributions of the total error (A) and the error estimator (B), for $\alpha = 1.99$.}
\label{fig:ex-2.2}
\end{figure}

In Figure \ref{fig:ex-1} we present, for the setting of Example 1 with $\alpha = 1.5$, the experimental rates of convergence for the total error and its individual contributions, with uniform and adaptive refinement. We observe that the designed adaptive procedure outperforms uniform refinement. In Figure \ref{fig:ex-2.2}, we present similar experimental rates of convergence for Example 2 with $\alpha = 1.99$.
From Figures \ref{fig:ex-1} and \ref{fig:ex-2.2}, we observe that our adaptive loop delivers optimal experimental rates of convergence for the individual contributions associated to the discretization of the state and adjoint equations. The individual contributions related to the control variable exhibit a suboptimal decayment. In order to improve such a suboptimal behavior
we propose a different marking strategy to be used in \textbf{Algorithm} \ref{Algorithm}. To present it, we first define the local indicator $\tilde{\mathcal{E}}_{T}$ as
\[ \tilde{\mathcal{E}}_{T}^{2} =E_{st}^{2}(\bar{\yy}_{\T},\bar{p}_{\T},\bar{\uu}_{\T}) + \mathcal{E}_{st,T}^{2}(\bar{\yy}_{\T},\bar{p}_{\T},\bar{\uu}_{\T}) + \mathcal{E}_{ad,T}^{2}(\bar{\zz}_{\T},\bar{r}_{\T},\bar{\yy}_{\T}).\] 
We thus replace the step 4 of \textbf{Algorithm} \ref{Algorithm} by: Mark an element $T\in\T$ for refinement if 
\begin{equation}\label{def:new_marking}
\tilde{\mathcal{E}}_{T}^{2}> \displaystyle\frac{1}{2}\max_{T'\in \mathscr{T}}\tilde{\mathcal{E}}_{T'}^{2} \quad \text{ or }\quad \mathcal{E}_{ct,T}^{2}> \displaystyle\frac{1}{2}\max_{T'\in \mathscr{T}}\mathcal{E}_{ct,T'}^{2}.
\end{equation}
This slight difference in \textbf{Algorithm} \ref{Algorithm} allows for an improvement in the experimental rates of convergence for the individual contributions $\mathcal{E}_{ct}$ and $\|\mathbf{e}_{\uu}\|_{\LL^2(\Omega)}$. The performance of the proposed adaptive strategy with the marking \eqref{def:new_marking}, for Example 2, is shown in Figure \ref{fig:ex-2.3}: optimal experimental rates of convergence for all the individual contributions are observed. Figure \ref{fig:ex-2.1}, presents, for Example 2, experimental rates of convergence for the total error and the global error estimator, considering $\alpha \in \{0.4, 0.6, 0.8, 1, 1.2, 1.4, 1.8, 1.99\}$. The adaptive loop delivers optimal results for all the values of the parameter $\alpha$ that we considered.

\begin{figure}[h]
\psfrag{error vel ad 1}{$\|\nabla \mathbf{e}_{\zz}\|_{\LL^{2}(\rho,\Omega)}$}
\psfrag{error pr ad 1}{$\|e_{r}\|_{L^{2}(\rho,\Omega)}$}
\psfrag{error vel 111}{$\|\mathbf{e}_{\yy}\|_{\LL^{\infty}(\Omega)}$}
\psfrag{error pr 1111}{$\|e_{p}\|_{L^{2}(\Omega)}$}
\psfrag{error ct 1111}{$\|\mathbf{e}_{\uu}\|_{\LL^{2}(\Omega)}$}
\psfrag{estimador total 11}{$\mathcal{E}_{ocp}$}
\psfrag{error ocp}{$\|\mathbf{e}\|_{\Omega}$}
\psfrag{error}{$\|\mathbf{e}\|_{\Omega}$}
\psfrag{estimador total}{$\mathcal{E}_{ocp}$}
\psfrag{eta ocp}{$\mathcal{E}_{ocp}$}
\psfrag{eta}{$\mathcal{E}_{ocp}$}
\psfrag{eta ad}{$\mathcal{E}_{ad}$}
\psfrag{eta st vel}{$\mathcal{E}_{st}$}
\psfrag{eta st pr}{$E_{st}$}
\psfrag{eta ct}{$\mathcal{E}_{ct}$}
\psfrag{Ndf-10}{\small$\textrm{Ndof}^{-1}$}
\psfrag{Ndf-32}{\small$\textrm{Ndof}^{-3/2}$}
\psfrag{Ndofs}{$\textrm{Ndof}$}
\begin{minipage}{0.32\textwidth}\centering
\psfrag{Example 2 - alfa 199 - errors}{\hspace{-0.1cm} Error contributions for $\alpha = 1.99$}
\includegraphics[trim={0 0 0 0},clip,width=4.5cm,height=4.5cm,scale=0.6]{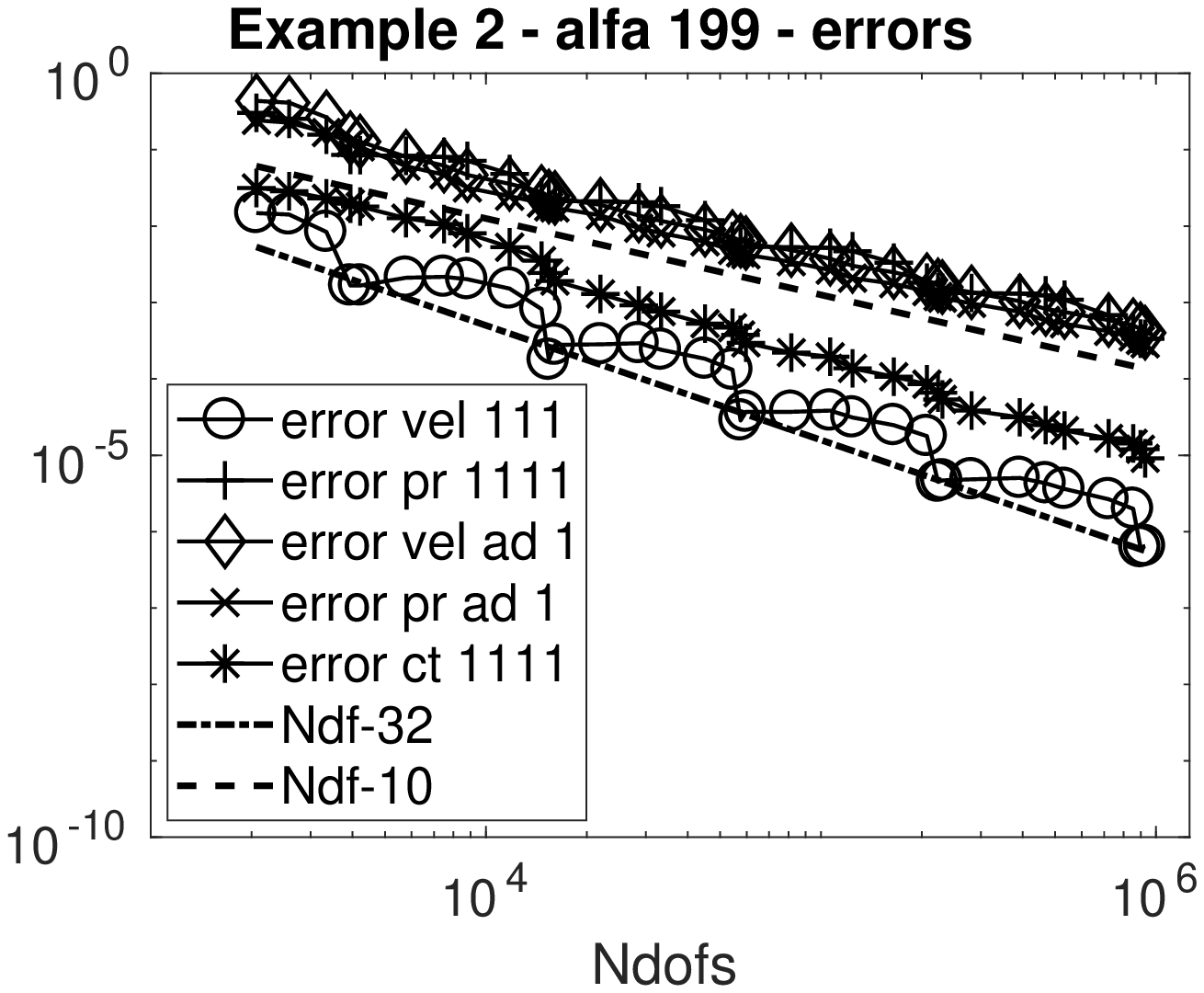}\\
\qquad \tiny{(A)}
\end{minipage}
\begin{minipage}{0.32\textwidth}\centering
\psfrag{Example 2 - alfa 199 - etas}{\hspace{-0.5cm}  Estimator contributions for $\alpha = 1.99$}
\includegraphics[trim={0 0 0 0},clip,width=4.5cm,height=4.5cm,scale=0.6]{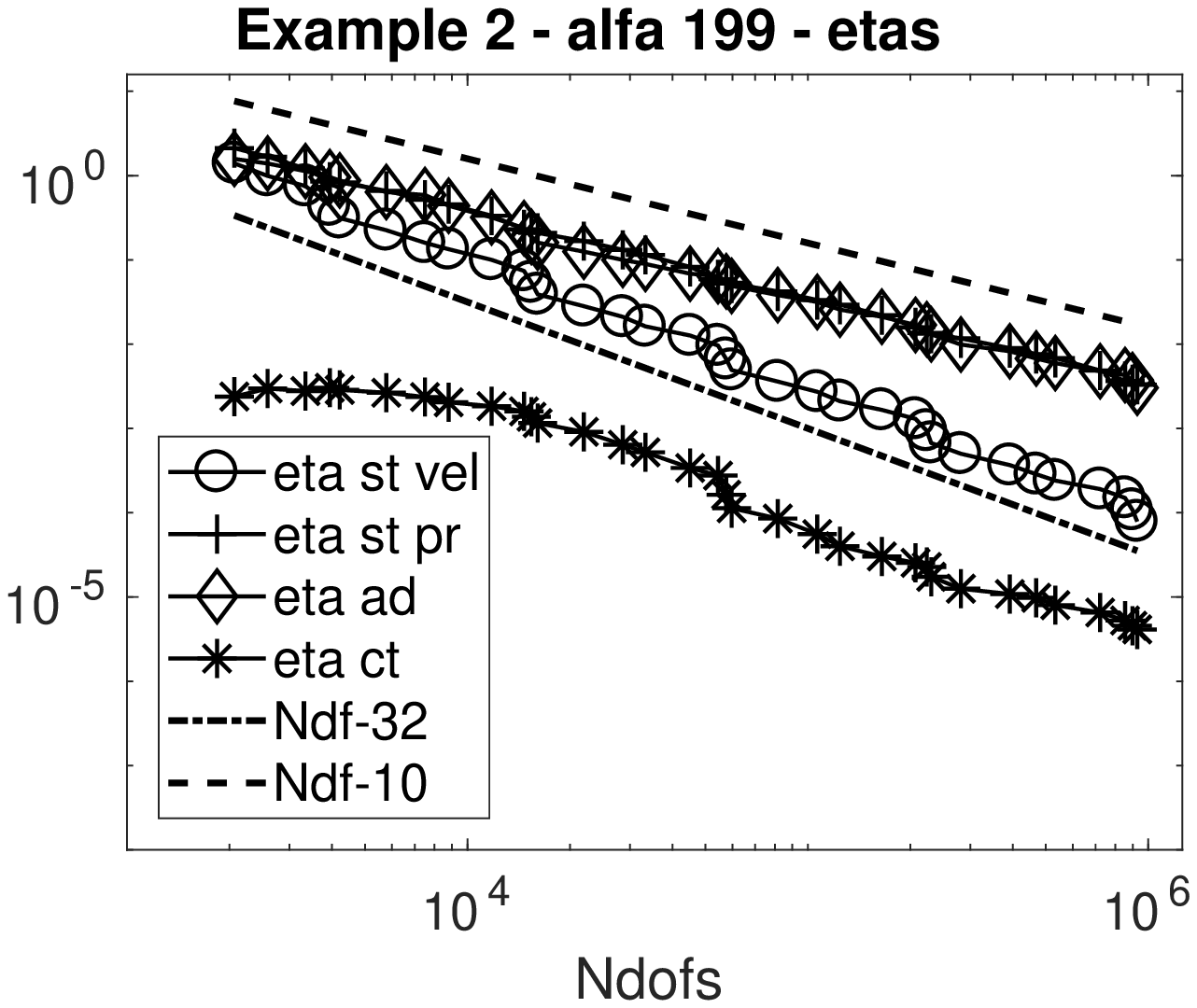}\\
\qquad \tiny{(B)}
\end{minipage}
\caption{Example 2: Experimental rates of convergence for the contributions of the total error (A) and the error estimator (B) considering $\alpha = 1.99$ and the alternative marking criterion \eqref{def:new_marking}.}
\label{fig:ex-2.3}
\end{figure}

\begin{figure}[h]
\psfrag{error vel ad 1}{$\|\nabla \mathbf{e}_{\zz}\|_{\LL^{2}(\rho,\Omega)}$}
\psfrag{error pr ad 1}{$\|e_{r}\|_{L^{2}(\rho,\Omega)}$}
\psfrag{error vel 111}{$\|\mathbf{e}_{\yy}\|_{\LL^{\infty}(\Omega)}$}
\psfrag{error pr 1111}{$\|e_{p}\|_{L^{2}(\Omega)}$}
\psfrag{error ct 1111}{$\|\mathbf{e}_{\uu}\|_{\LL^{2}(\Omega)}$}
\psfrag{estimador total 11}{$\mathcal{E}_{ocp}$}
\psfrag{error ocp}{$\|\mathbf{e}\|_{\Omega}$}
\psfrag{error}{$\|\mathbf{e}\|_{\Omega}$}
\psfrag{estimador total}{$\mathcal{E}_{ocp}$}
\psfrag{eta ocp}{$\mathcal{E}_{ocp}$}
\psfrag{eta}{$\mathcal{E}_{ocp}$}
\psfrag{eta ad}{$\mathcal{E}_{ad}$}
\psfrag{eta st vel}{$\mathcal{E}_{st}$}
\psfrag{eta st pr}{$E_{st}$}
\psfrag{eta ct}{$\mathcal{E}_{ct}$}
\psfrag{Ndf-12}{$\textrm{Ndof}^{-1/2}$}
\psfrag{Ndf-13}{$\textrm{Ndof}^{-1/3}$}
\psfrag{Ndf-32}{$\textrm{Ndof}^{-3/2}$}
\psfrag{Ndf-10}{\small{$\textrm{Ndof}^{-1}$}}
\psfrag{Ndofs}{$\textrm{Ndof}$}
\begin{minipage}{0.32\textwidth}\centering
\psfrag{Example 2 - alfa 04}{$\|\mathbf{e}\|_{\Omega}$ and $\mathcal{E}_{ocp}$ for $\alpha = 0.4$}
\includegraphics[trim={0 0 0 0},clip,width=4.4cm,height=3.9cm,scale=0.6]{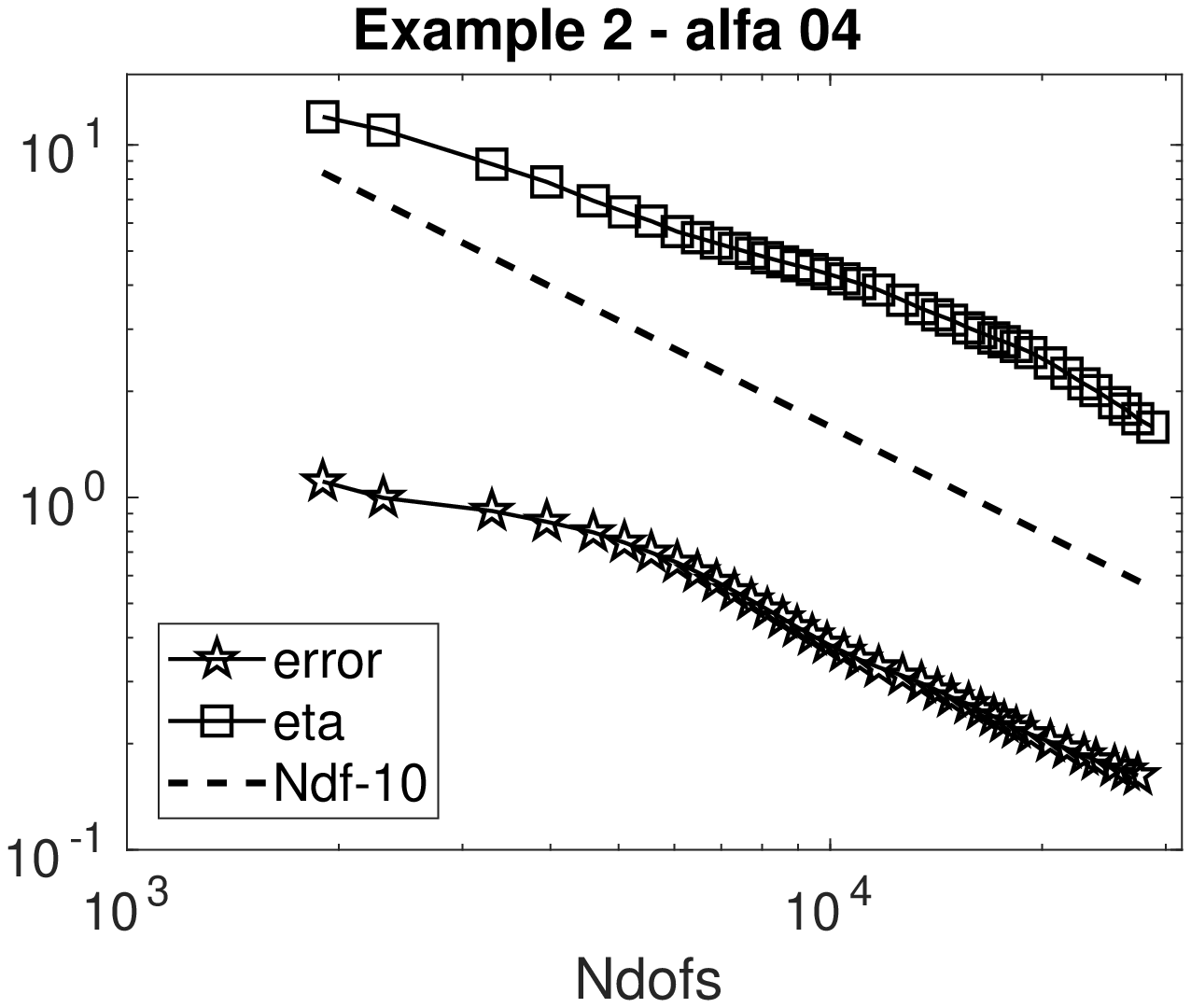}\\
\qquad \tiny{(A)}
\end{minipage}
\begin{minipage}{0.32\textwidth}\centering
\psfrag{Example 2 - alfa 06}{$\|\mathbf{e}\|_{\Omega}$ and $\mathcal{E}_{ocp}$ for $\alpha = 0.6$}
\includegraphics[trim={0 0 0 0},clip,width=4.4cm,height=3.9cm,scale=0.6]{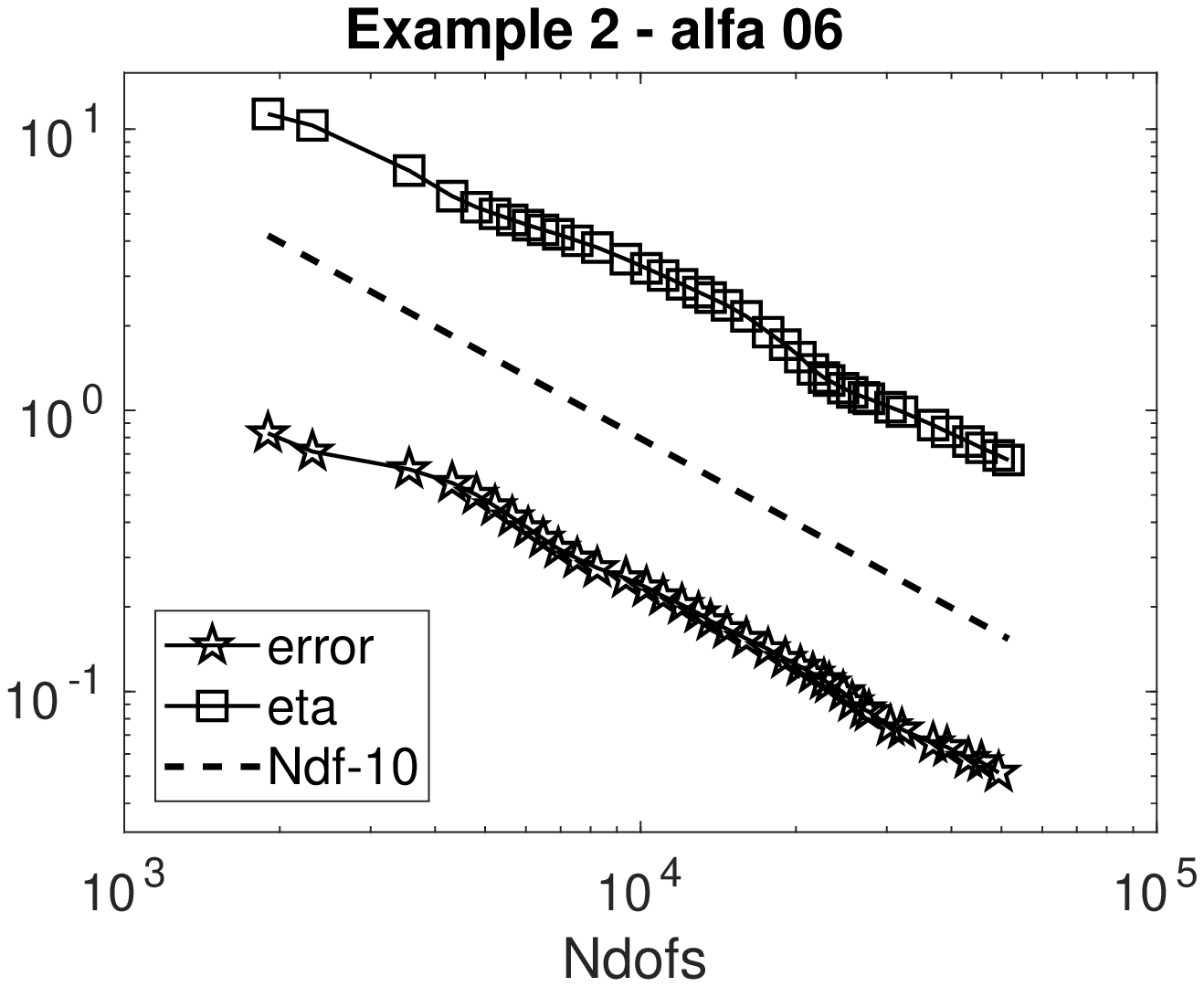}\\
\qquad \tiny{(B)}
\end{minipage}
\begin{minipage}{0.32\textwidth}\centering
\psfrag{Example 2 - alfa 08}{$\|\mathbf{e}\|_{\Omega}$ and $\mathcal{E}_{ocp}$ for $\alpha = 0.8$}
\includegraphics[trim={0 0 0 0},clip,width=4.4cm,height=3.9cm,scale=0.6]{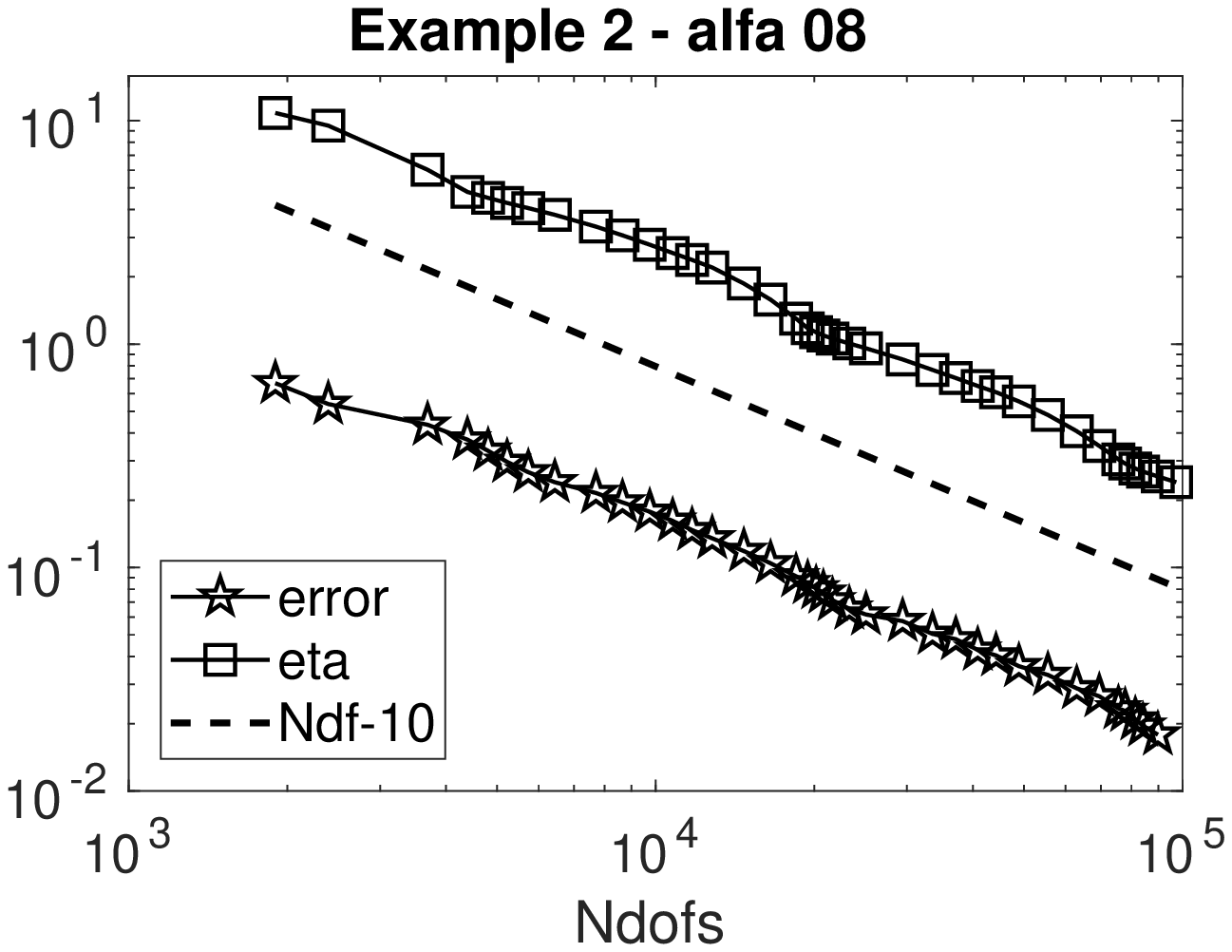}\\
\qquad \tiny{(C)}
\end{minipage}
\\
\begin{minipage}{0.32\textwidth}\centering
\psfrag{Example 2 - alfa 1}{$\|\mathbf{e}\|_{\Omega}$ and $\mathcal{E}_{ocp}$ for $\alpha = 1.0$}
\includegraphics[trim={0 0 0 0},clip,width=4.4cm,height=3.9cm,scale=0.6]{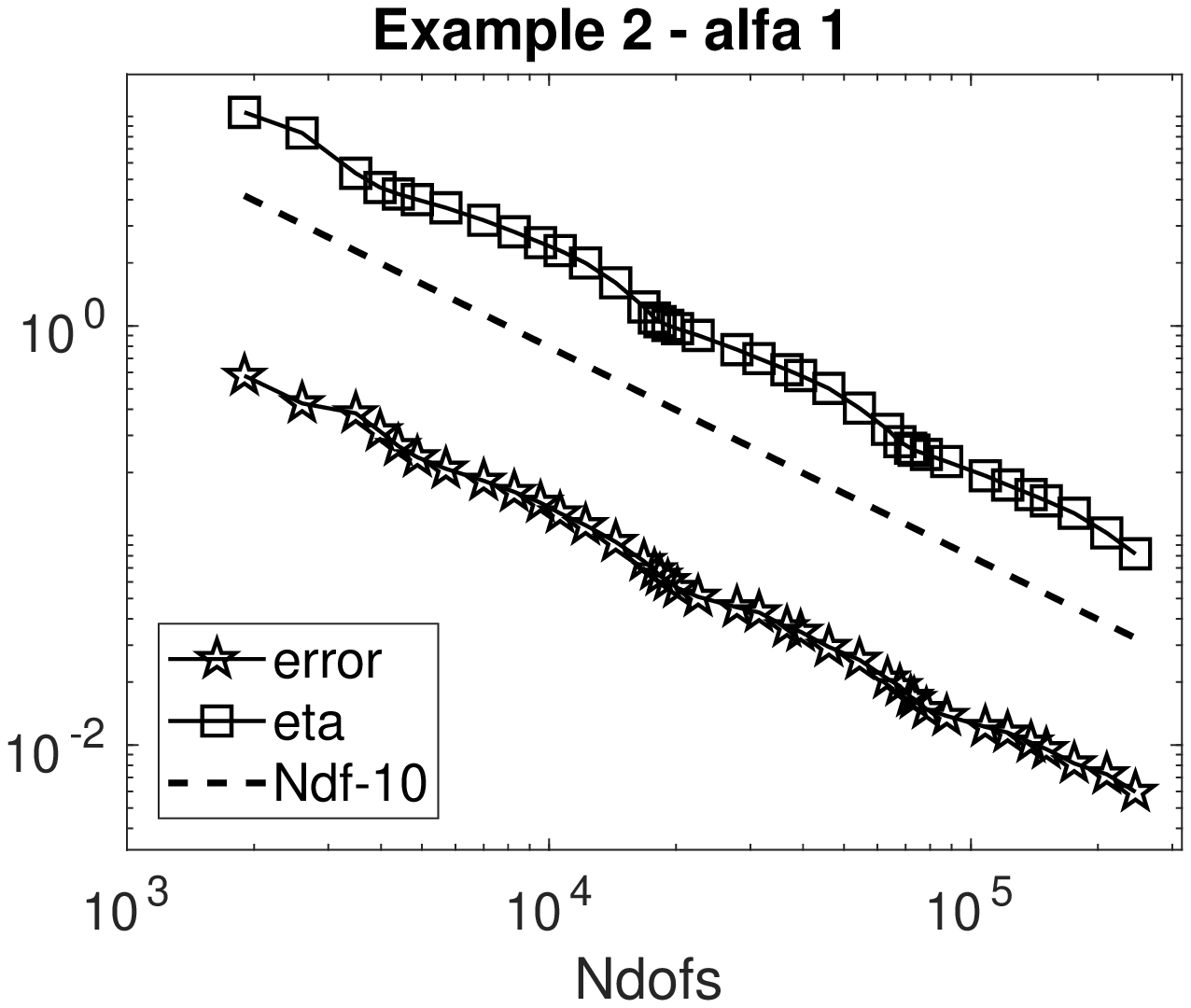}\\
\qquad \tiny{(D)}
\end{minipage}
\begin{minipage}{0.32\textwidth}\centering
\psfrag{Example 2 - alfa 12}{$\|\mathbf{e}\|_{\Omega}$ and $\mathcal{E}_{ocp}$ for $\alpha = 1.2$}
\includegraphics[trim={0 0 0 0},clip,width=4.4cm,height=3.9cm,scale=0.6]{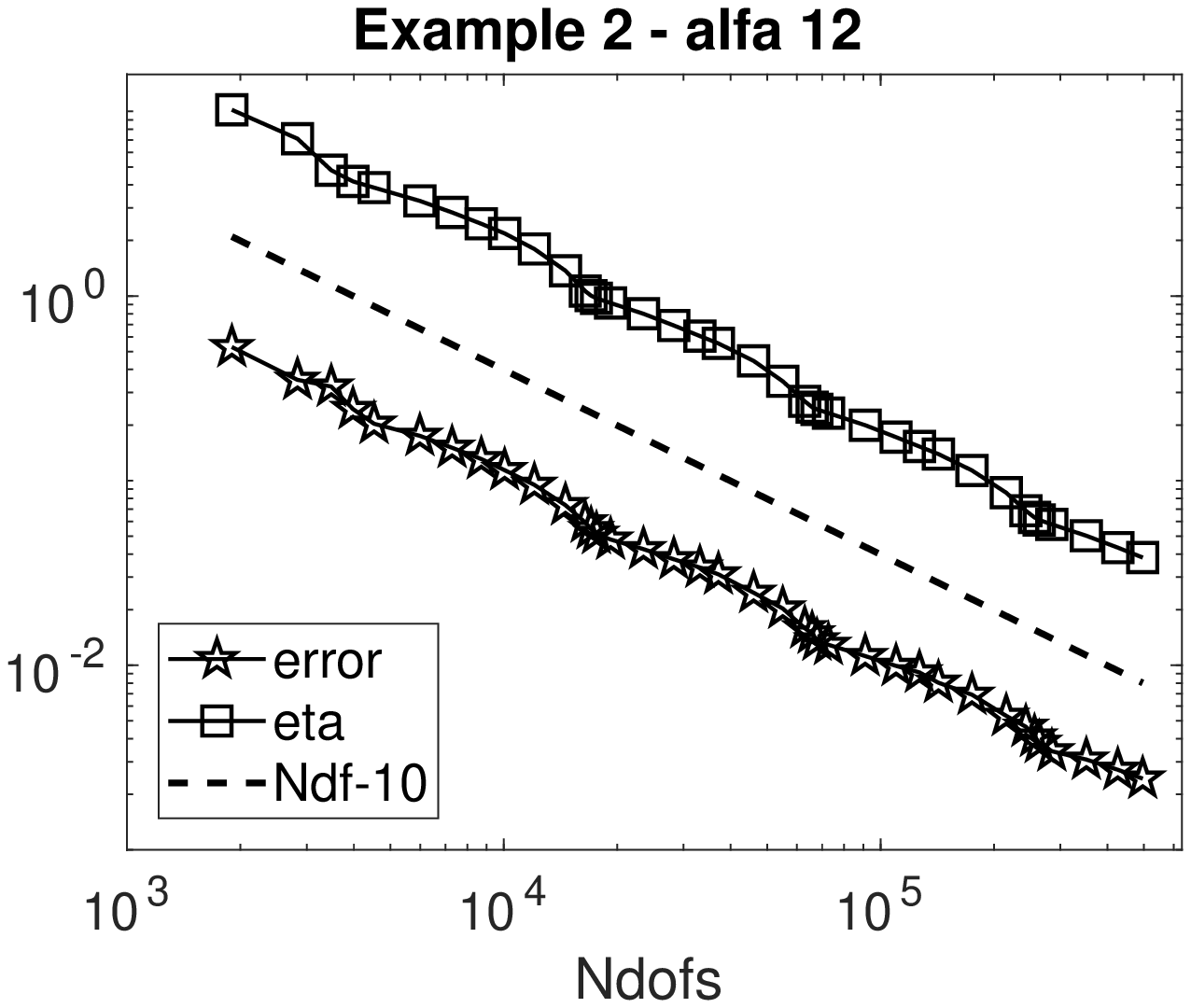}\\
\qquad \tiny{(E)}
\end{minipage}
\begin{minipage}{0.32\textwidth}\centering
\psfrag{Example 2 - alfa 14}{$\|\mathbf{e}\|_{\Omega}$ and $\mathcal{E}_{ocp}$ for $\alpha = 1.4$}
\includegraphics[trim={0 0 0 0},clip,width=4.4cm,height=3.9cm,scale=0.6]{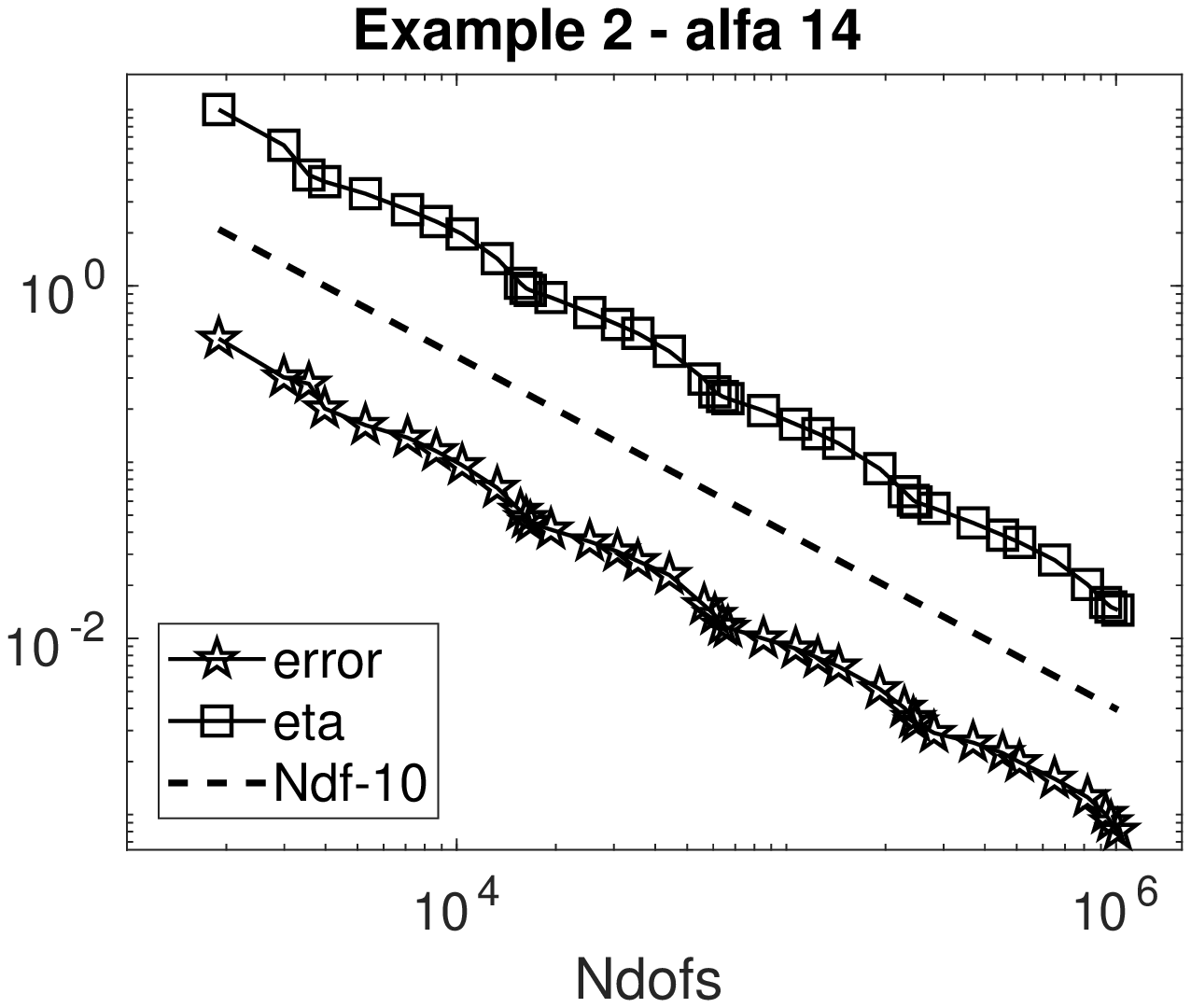}\\
\qquad \tiny{(F)}
\end{minipage}
\\
\begin{minipage}{0.32\textwidth}\centering
\psfrag{Example 2 - alfa 16}{$\|\mathbf{e}\|_{\Omega}$ and $\mathcal{E}_{ocp}$ for $\alpha = 1.6$}
\includegraphics[trim={0 0 0 0},clip,width=4.4cm,height=3.9cm,scale=0.6]{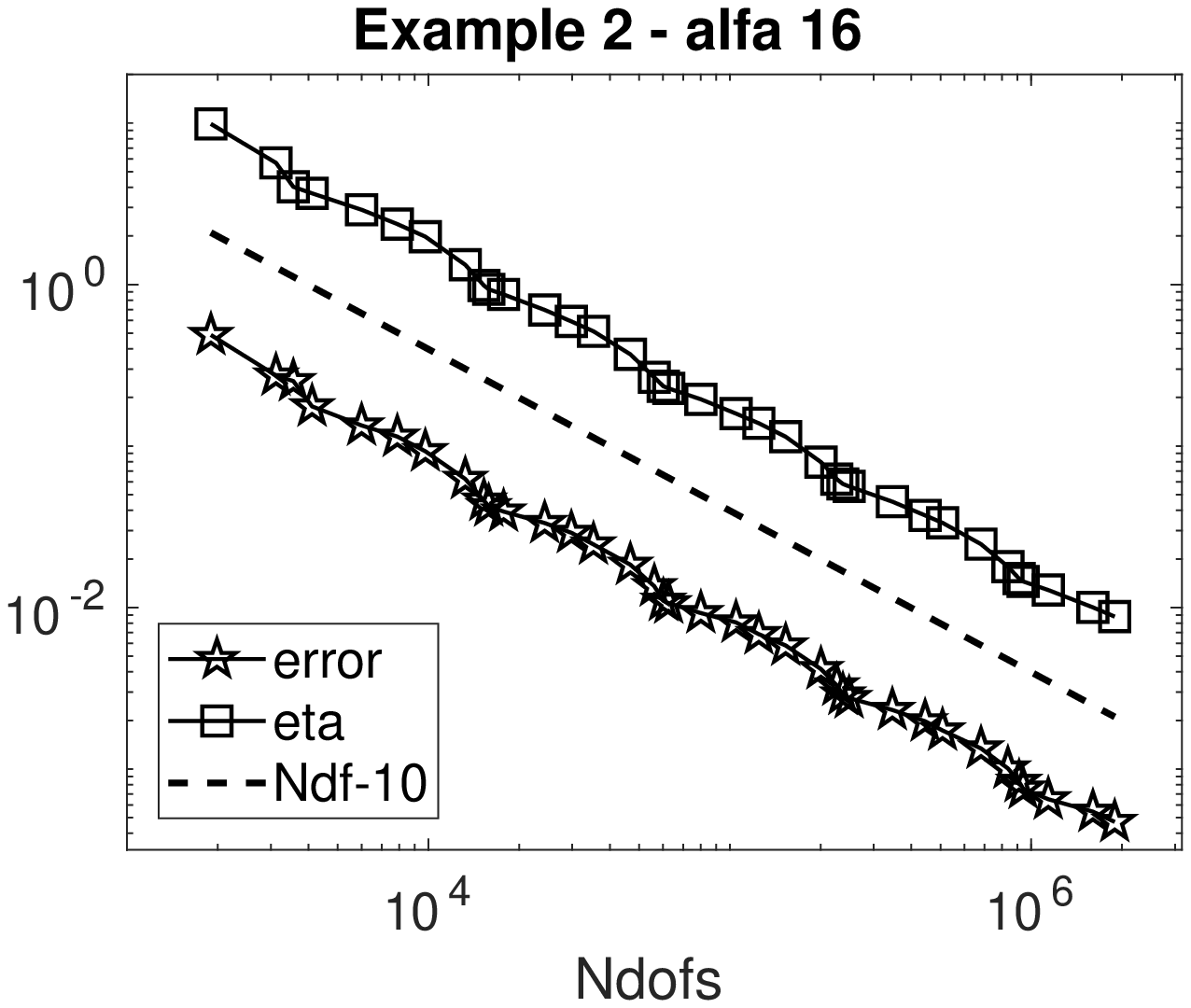}\\
\qquad \tiny{(G)}
\end{minipage}
\begin{minipage}{0.32\textwidth}\centering
\psfrag{Example 2 - alfa 18}{$\|\mathbf{e}\|_{\Omega}$ and $\mathcal{E}_{ocp}$ for $\alpha = 1.8$}
\includegraphics[trim={0 0 0 0},clip,width=4.4cm,height=3.9cm,scale=0.6]{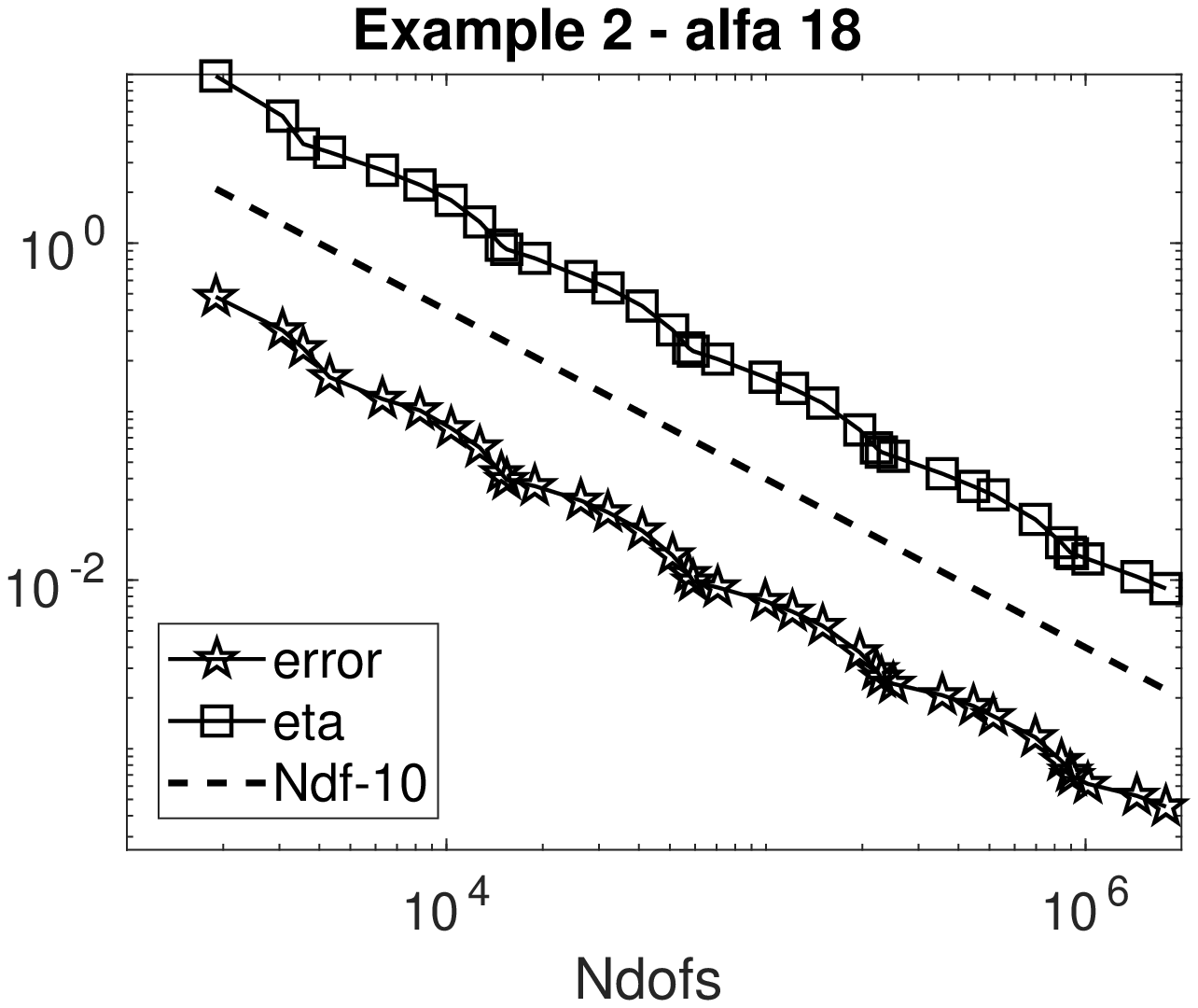}\\
\qquad \tiny{(H)}
\end{minipage}
\begin{minipage}{0.32\textwidth}\centering
\psfrag{Example 2 - alfa 199}{$\|\mathbf{e}\|_{\Omega}$ and $\mathcal{E}_{ocp}$ for $\alpha = 1.99$}
\includegraphics[trim={0 0 0 0},clip,width=4.4cm,height=3.9cm,scale=0.6]{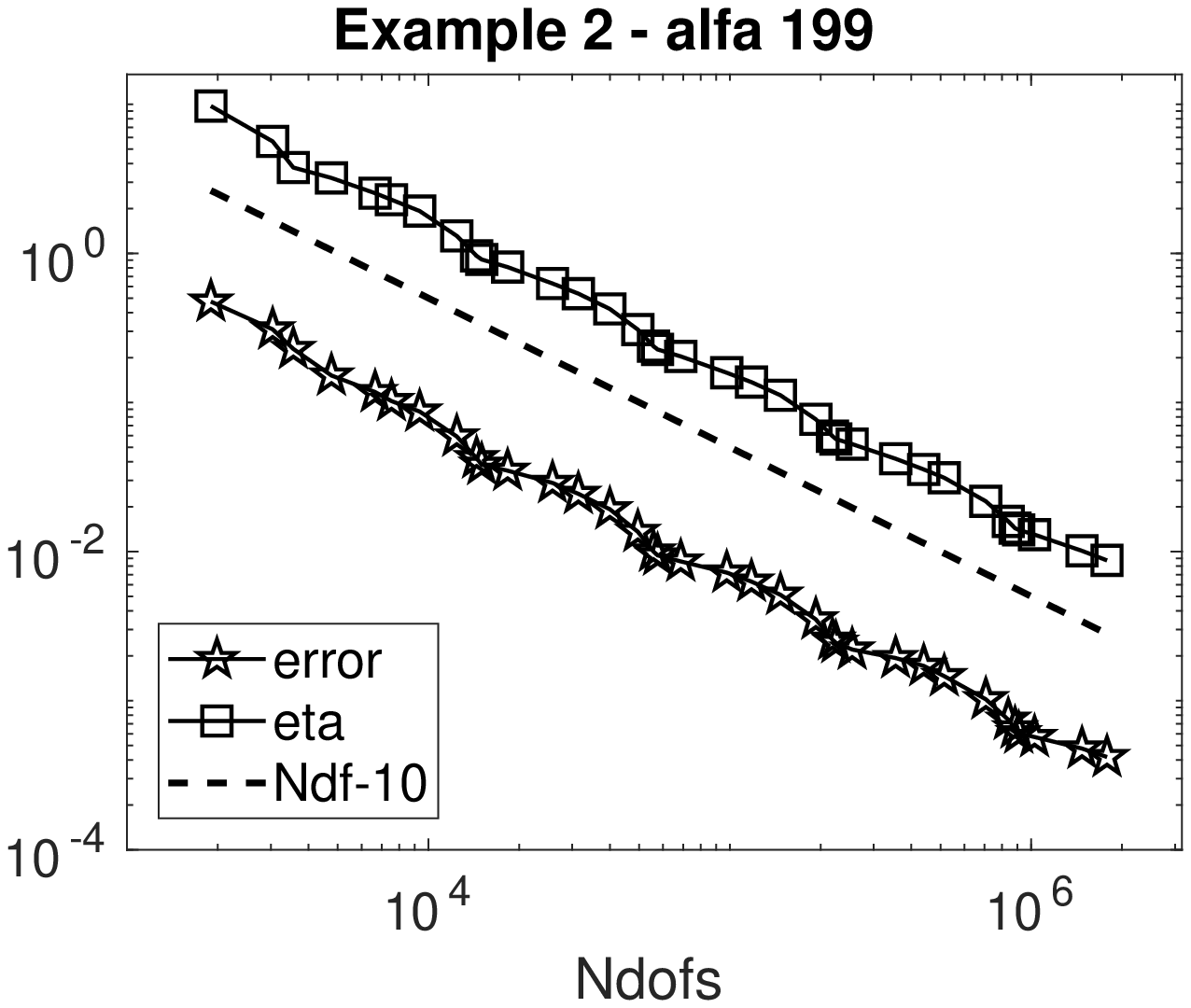}\\
\qquad \tiny{(I)}
\end{minipage}
\caption{Example 2: Experimental rates of convergence for the total error $\|\mathbf{e}\|_{\Omega}$ and the error estimator $\mathcal{E}_{ocp}$ for $\alpha \in \{0.4,0.6,0.8, 1, 1.2,1.4,1.6,1.8,1.99\}$ (A) - (I).}
\label{fig:ex-2.1}
\end{figure}

In Figure \ref{fig:ex_3} we present, for the setting of Example 3 with  $\alpha = 1$, the euclidean norm of the finite element approximation of the optimal adjoint velocity field $\bar{\zz}_\T$, the finite element approximation of the adjoint pressure $\bar{r}_\T$, and the euclidean norm of the finite element approximation of the optimal control $\bar{\uu}_\T$, on a suitable adaptively refined mesh. We also present experimental rates of convergence for the a posteriori error estimator $\mathcal{E}_{ocp}$ and its individual contributions.
         
\begin{figure}[h]
\psfrag{eta ad}{$\mathcal{E}_{ad}$}
\psfrag{eta vel}{$\mathcal{E}_{st}$}
\psfrag{eta pr}{$E_{st}$}
\psfrag{eta ct}{$\mathcal{E}_{ct}$}
\psfrag{yh}{$|\bar{\yy}_{\T}|$}
\psfrag{uh}{$|\bar{\uu}_{\T}|$}
\psfrag{zh}{$|\bar{\zz}_{\T}|$}
\psfrag{rh}{$\bar{r}_{\T}$}
\psfrag{Ndf-10}{$\textrm{Ndof}^{-1}$}
\psfrag{Ndf-32 11}{$\textrm{Ndof}^{-3/2}$}
\psfrag{Example 2 - alfa 1}{\hspace{-0.5cm}$\mathcal{E}_{st}, E_{st}, \mathcal{E}_{ad}$ and $\mathcal{E}_{ct}$ for $\alpha = 1$}
\psfrag{Ndofs}{$\textrm{Ndof}$}
\begin{minipage}{0.32\textwidth}\centering
\tiny{$|\bar{\zz}_{\T}|$}\\
\includegraphics[width=4.2cm,height=3.2cm,scale=0.7]{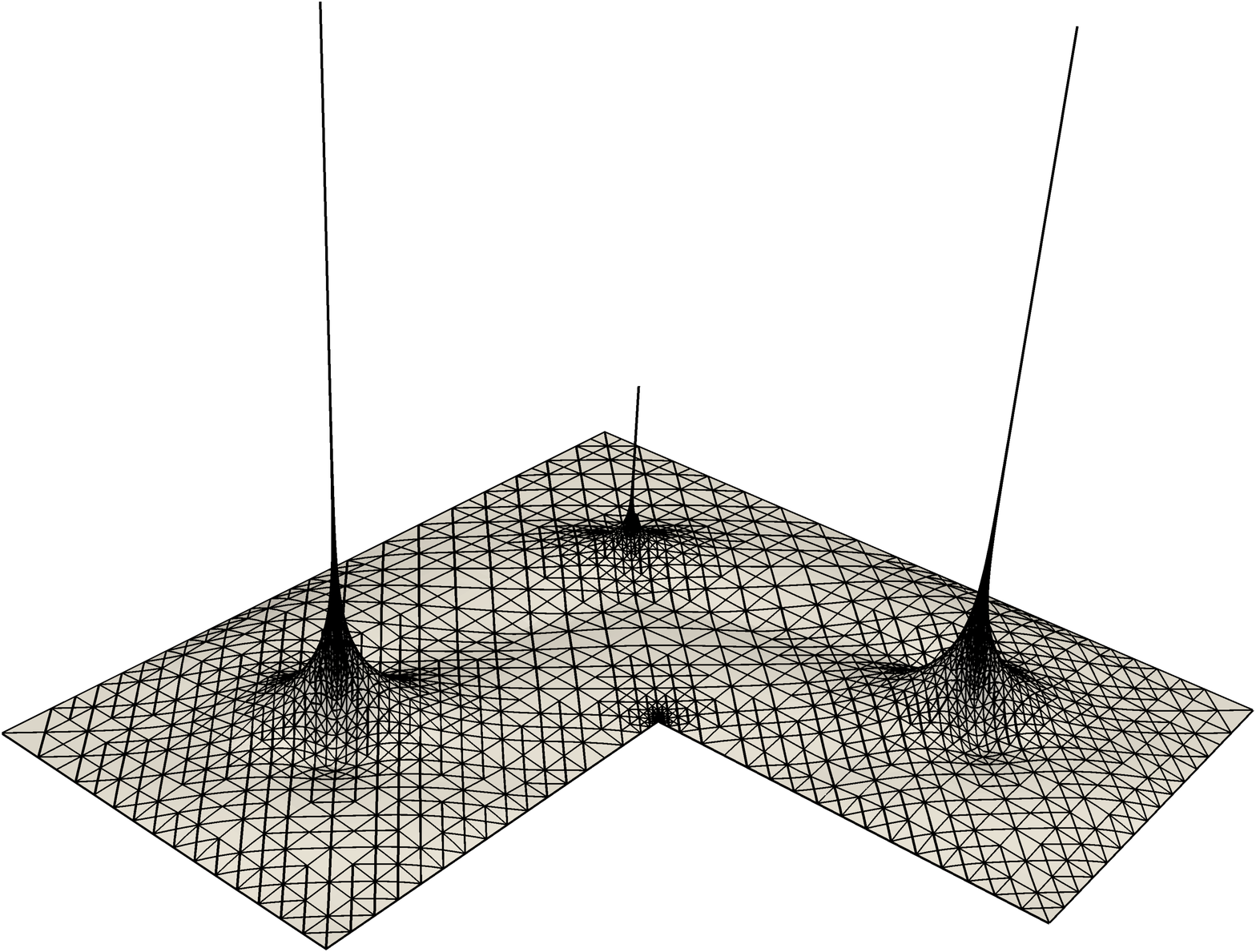}
\qquad \tiny{(A)}
\end{minipage}
\begin{minipage}{0.32\textwidth}\centering
\tiny{$\bar{r}_{\T}$}
\includegraphics[trim={0 0 0 0},clip,width=4cm,height=3.2cm,scale=0.7]{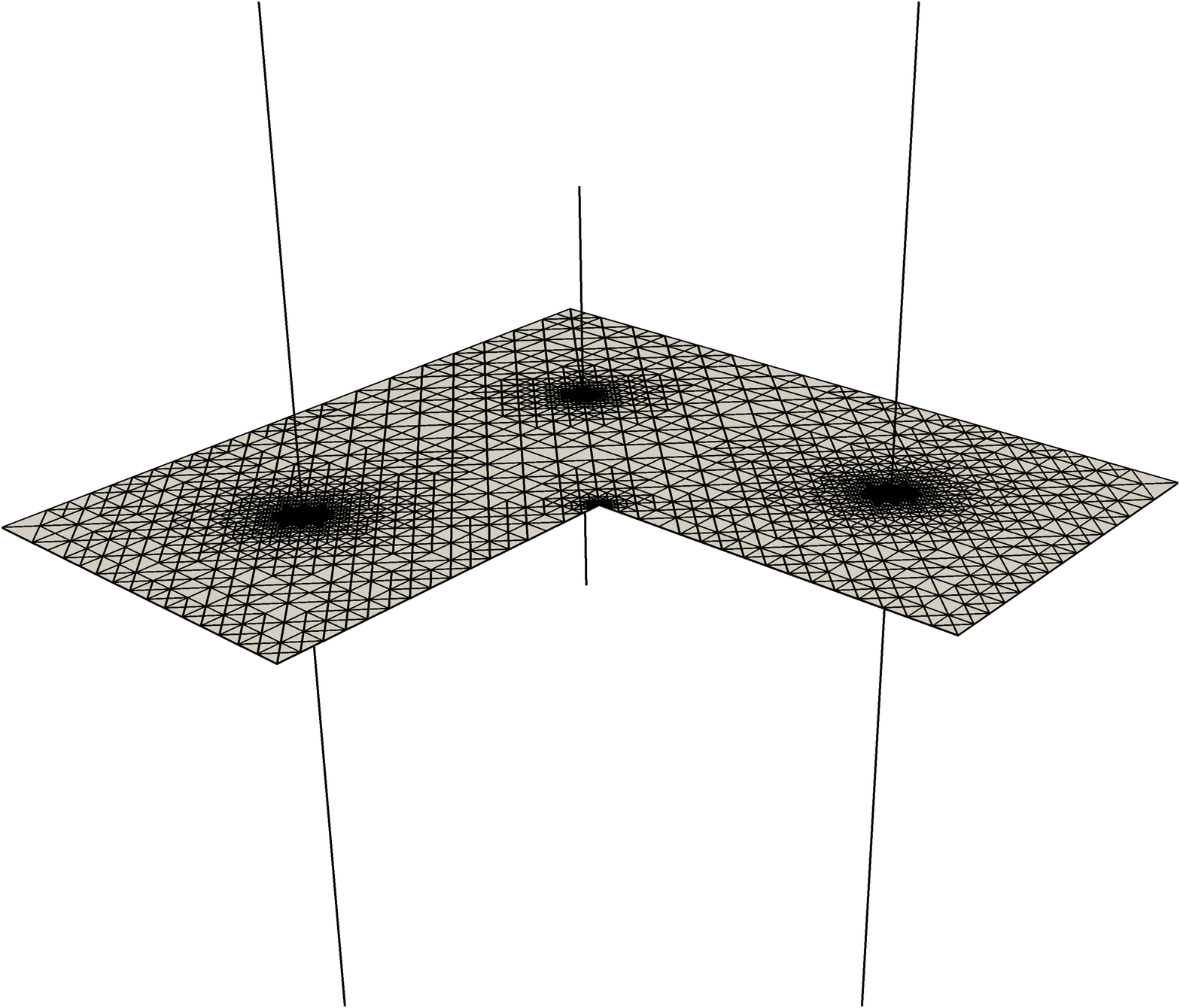}
\qquad \tiny{(B)}
\end{minipage}
\begin{minipage}{0.32\textwidth}\centering
\tiny{$|\bar{\uu}_{\T}|$}\\
\includegraphics[width=4.2cm,height=3.5cm,scale=0.7]{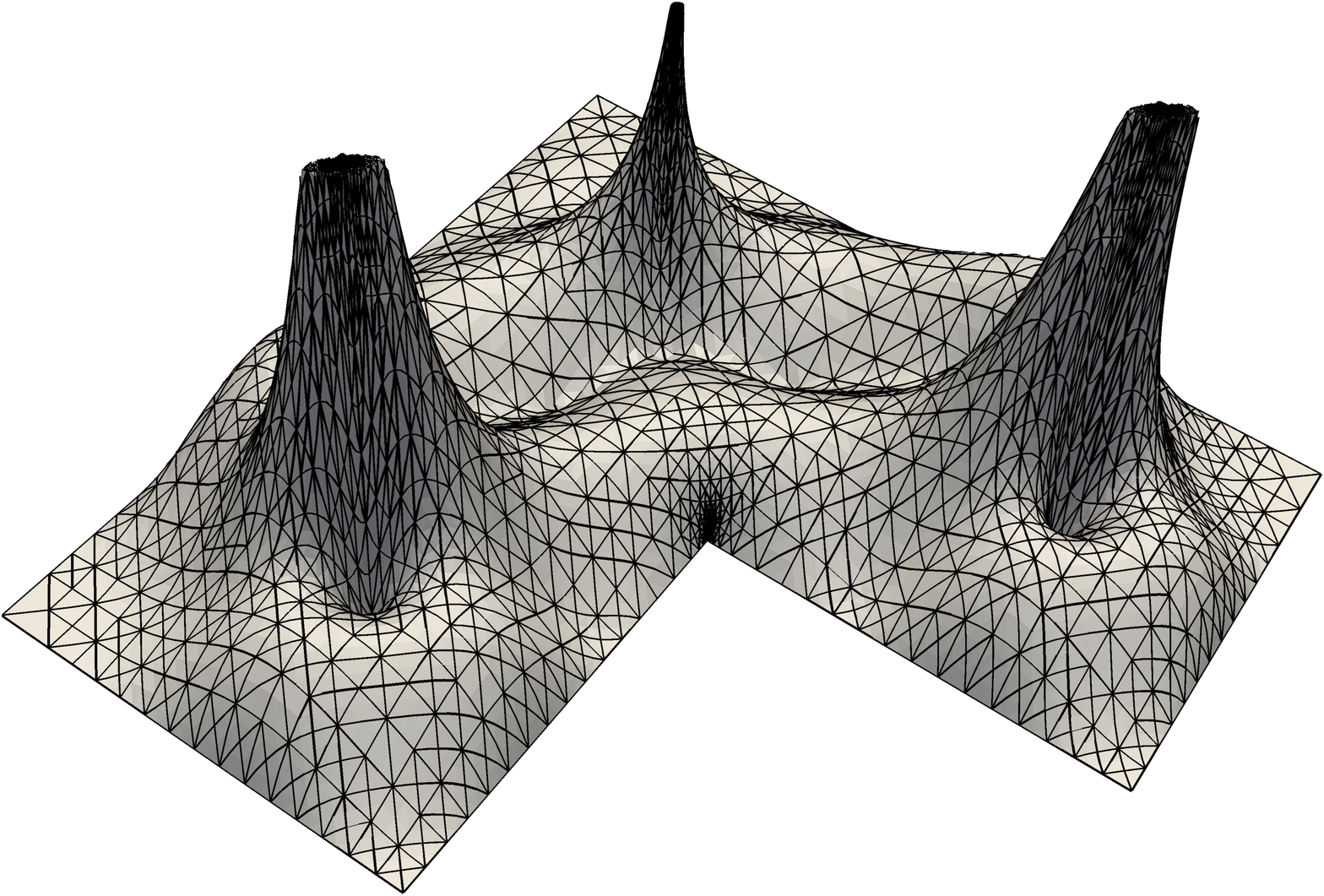}
\qquad \tiny{(C)}
\end{minipage}
\\~\\
\begin{minipage}{0.32\textwidth}\centering
\includegraphics[trim={0 0 0 0},clip,width=3.8cm,height=3.8cm,scale=0.7]{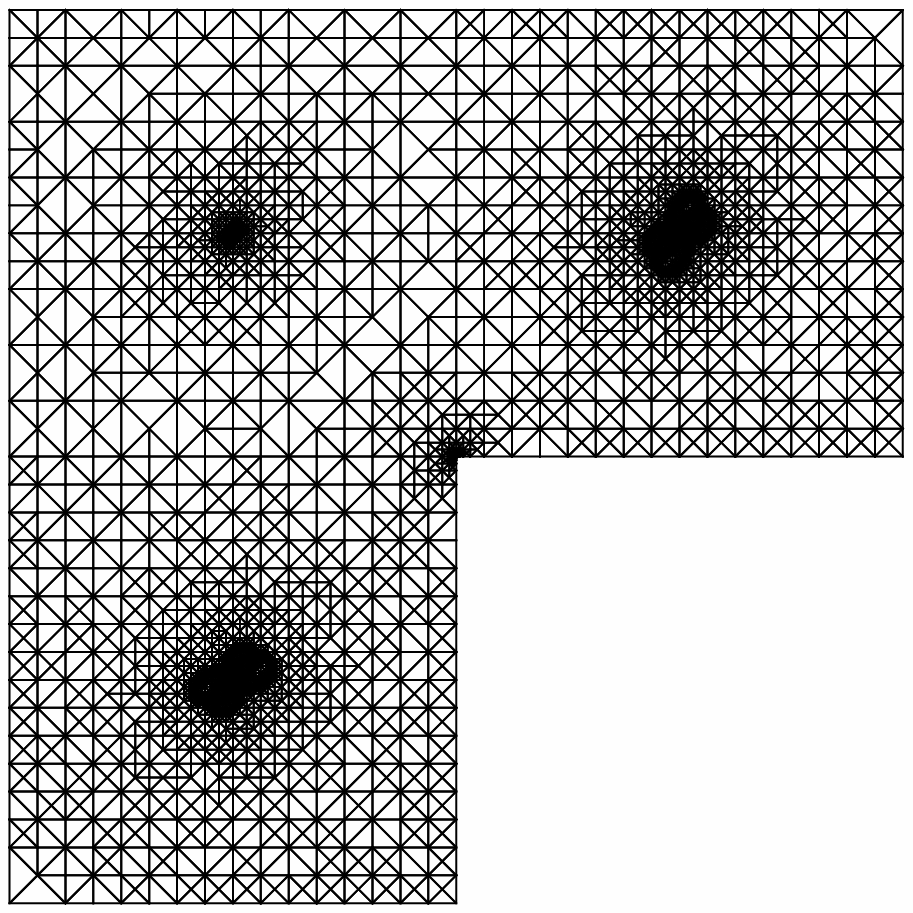}
\qquad \tiny{(D)}
\end{minipage}
\hspace{-0.35cm}
\begin{minipage}{0.32\textwidth}\centering
\includegraphics[trim={0 0 0 0},clip,width=4.5cm,height=4.0cm,scale=0.7]{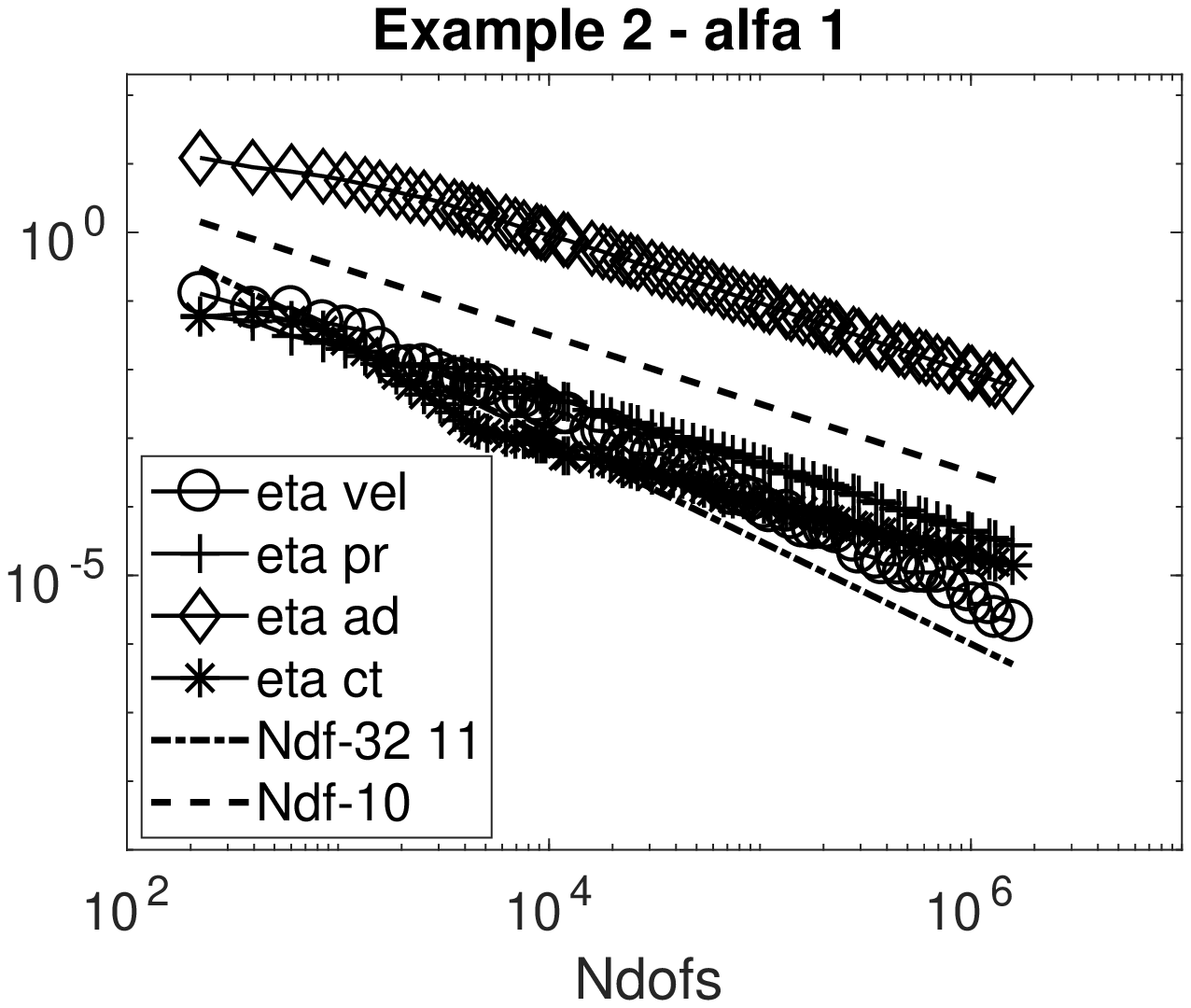}
\qquad \tiny{(E)}
\end{minipage}
\caption{Example 3: Finite element approximations of $|\bar{\zz}_{\T}|$, $\bar{r}_{\T}$ and $|\bar{\uu}_{\T}|$ (A)-(C) obtained on the 40th adaptively refined mesh (D) ($\alpha = 1$), and experimental rates of convergence for the individual contributions of the error estimator $\mathcal{E}_{ocp}$ (E). }
\label{fig:ex_3}
\end{figure}        
  
\subsection{Three-dimensional examples}
We now present three dimensional examples with homogeneous and inhomogeneous Dirichlet boundary conditions and different number of source points.
\\~\\
\textbf{Example 4.} We set $\Omega=(0,1)^3$, $\mathbf{a} = (-0.6,-0.6,-0.6)^T$, $\mathbf{b} = (-0.2,-0.2,-0.2)^T$, $\lambda = 1$, and 
\begin{align*}
\begin{array}{crc}
\mathcal{D}=\left\{(0.25,0.25,0.25),(0.25,0.25,0.75),(0.75,0.25,0.25),(0.75,0.25,0.75), \right. &  &\\
\left. (0.25,0.75,0.25),(0.25,0.75,0.75),(0.75,0.75,0.25),(0.75,0.75,0.75)\right\}.
\end{array}
\end{align*}
The exact optimal state is
\begin{align*}
\bar{\yy}(x_{1},x_{2},x_{3}) & = 2\mathbf{curl}((x_{1}x_{2}x_{3}(1-x_{1})(1-x_{2})(1-x_{3}))^{2}),
\\
\bar{p}(x_{1},x_{2},x_{3}) & = 2x_{1}x_{2}x_{3} - 0.25.
\end{align*}
The optimal adjoint state is as in \eqref{def:adjoint_deltas} with $\vartheta_{t}=4/5$ for all $t\in\mathcal{D}$. It can be inferred that $\yy_{t}=\bar{\yy}(t)-4/5(1,1,1)^{T}$ for all $t\in\mathcal{D}$.
\\~\\
\textbf{Example 5.} We set $\Omega=(0,1)^3$, $\mathbf{a} = (-2,-2,-2)^T$, $\mathbf{b} = (-1,-1,-1)^T$, $\lambda = 1$, and 
\begin{equation*}
\mathcal{D}=\left\{(0.25,0.25,0.25),(0.75,0.25,0.25),(0.25,0.75,0.75),(0.75,0.75,0.75) \right\}.
\end{equation*}
The set of observable points is
\begin{align*}
\yy_{(0.25,0.25,0.25)} = (-5,-5,-5)^T, \quad \yy_{(0.75,0.25,0.25)} = (1,1,1)^T, \\ \yy_{(0.25,0.75,0.75)} = (5,5,5)^T, \quad \yy_{(0.75,0.75,0.75)} = (-1,-1,-1)^T.
\end{align*}
         
In Figure \ref{fig:ex_4} we present, for Example 4, the experimental rates of convergence for the total error and the global error estimator, as well as their contributions, the effectivity index $\mathcal{E}_{ocp}/\|\mathbf{e}\|_{\Omega}$, and slices of the 40th adaptively refined mesh. We notice that the effectivity index is close to four. This shows the accuracy of the proposed a posteriori error estimator $\mathcal{E}_{ocp}$ when used in an adaptive loop solving a nonlinear optimal control problem in a three dimensional domain. Finally, in Figure \ref{fig:ex_5} we show, for Example 5, the experimental rates of convergence for the global estimator $\mathcal{E}_{ocp}$ and its individual contributions, together with slices of the 87th adaptively refined mesh.

\begin{figure}[ht]
\centering
\psfrag{error vel}{\LARGE{$\|\mathbf{e}_{\yy}\|_{\LL^{\infty}(\Omega)}$}}
\psfrag{error pre}{\LARGE{$\|e_{p}\|_{L^{2}(\Omega)}$}}
\psfrag{error vel ad}{\LARGE{$\|\nabla \mathbf{e}_{\zz}\|_{\LL^{2}(\rho,\Omega)}$}}
\psfrag{error pre ad}{\LARGE{$\|e_{r}\|_{L^{2}(\rho,\Omega)}$}}
\psfrag{error ct}{\LARGE{$\|\mathbf{e}_{\uu}\|_{\LL^{2}(\Omega)}$}}
\psfrag{e ocp}{\LARGE{$\|\mathbf{e}\|_{\Omega}$}}
\psfrag{eta ad}{\LARGE{$\mathcal{E}_{ad}$}}
\psfrag{eta vel}{\LARGE{$\mathcal{E}_{st}$}}
\psfrag{eta pre}{\LARGE{$E_{st}$}}
\psfrag{eta ct}{\LARGE{$\mathcal{E}_{ct}$}}
\psfrag{eta ocp}{\LARGE{$\mathcal{E}_{ocp}$}}
\psfrag{O(Ndofs-23)}{\LARGE{$\textrm{Ndof}^{-2/3}$}}
\psfrag{O(Ndofs-43)}{\LARGE{$\textrm{Ndof}^{-4/3}$}}
\psfrag{example A - e}{\hspace{-1.5cm}\LARGE{$\|\mathbf{e}\|_{\Omega}$ and $\mathcal{E}_{ocp}$ for $\alpha = 1.99$}}
\psfrag{example A - errors}{\hspace{-1.5cm}\LARGE{Error contributions for $\alpha = 1.99$}}
\psfrag{example A - etas}{\hspace{-2.0cm}\LARGE{Estimator contributions for $\alpha = 1.99$}}
\psfrag{Ndofs}{\LARGE{$\textrm{Ndof}$}}
\begin{minipage}{0.32\textwidth}\centering
\includegraphics[trim={0 0 0 0},clip,width=4.4cm,height=3.9cm,scale=0.8]{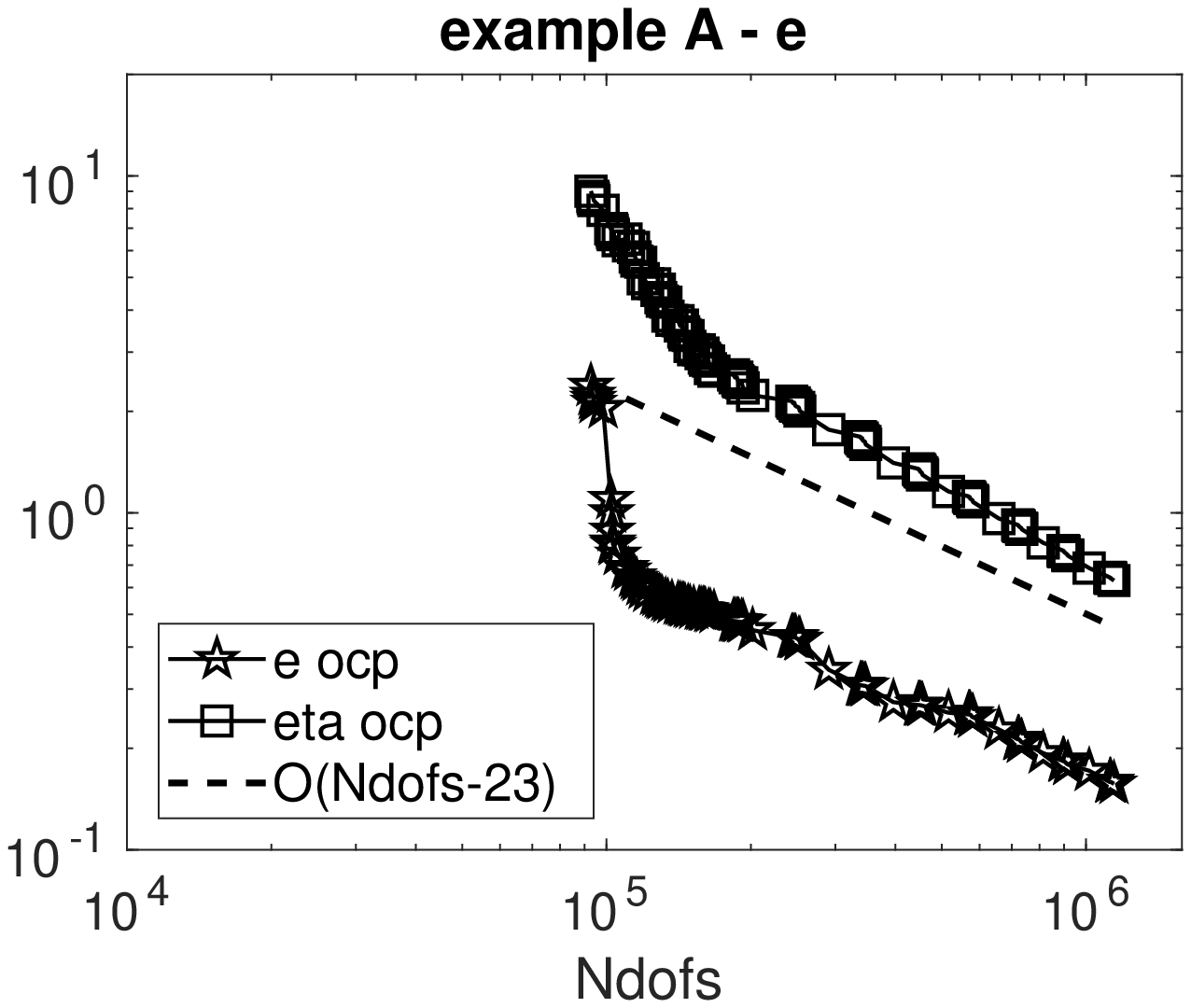}\\
\qquad \tiny{(A)}
\end{minipage}
\begin{minipage}{0.32\textwidth}\centering
\includegraphics[trim={0 0 0 0},clip,width=4.4cm,height=3.9cm,scale=0.8]{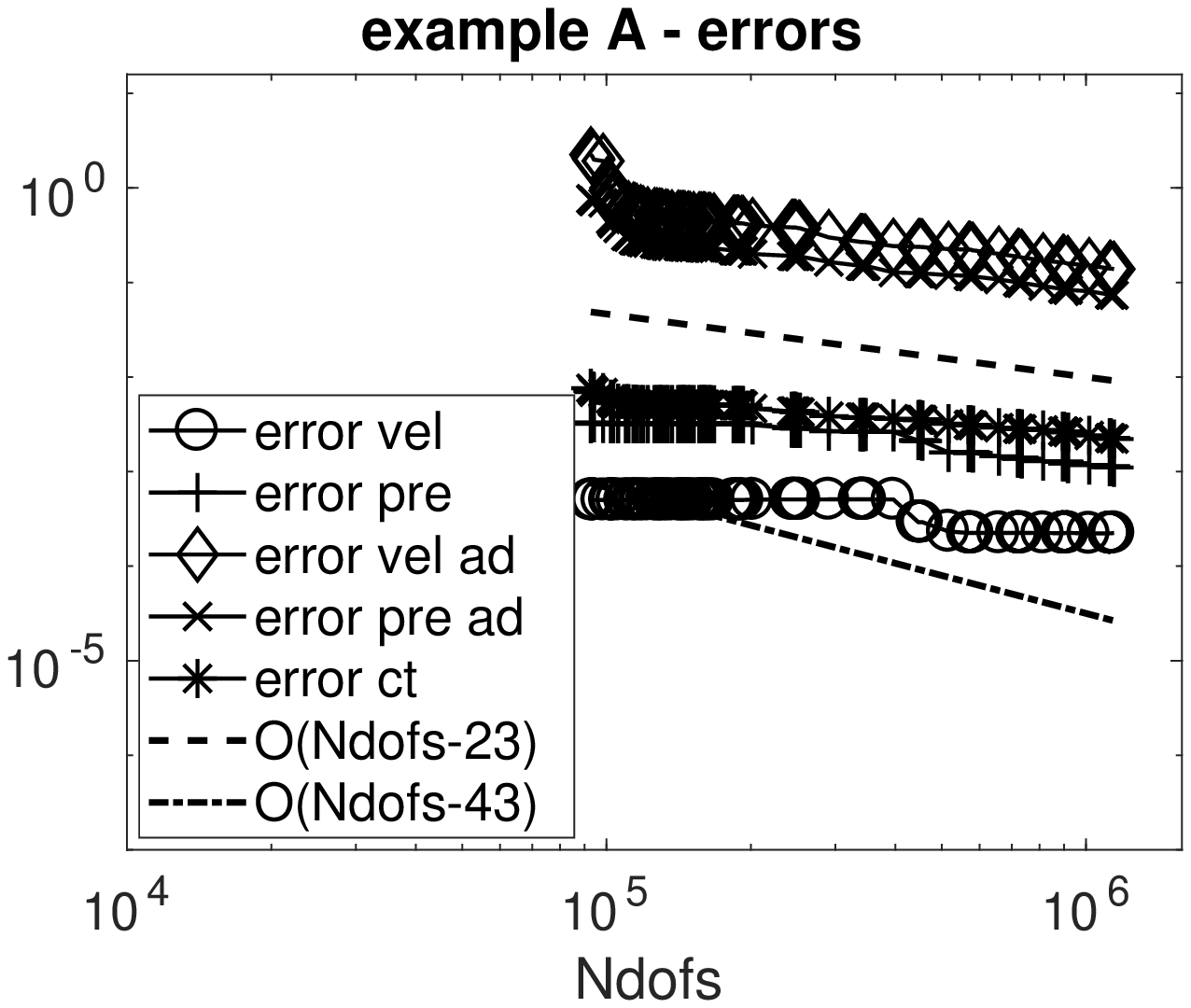}\\
\qquad \tiny{(B)}
\end{minipage}
\begin{minipage}{0.32\textwidth}\centering
\includegraphics[trim={0 0 0 0},clip,width=4.4cm,height=3.9cm,scale=0.8]{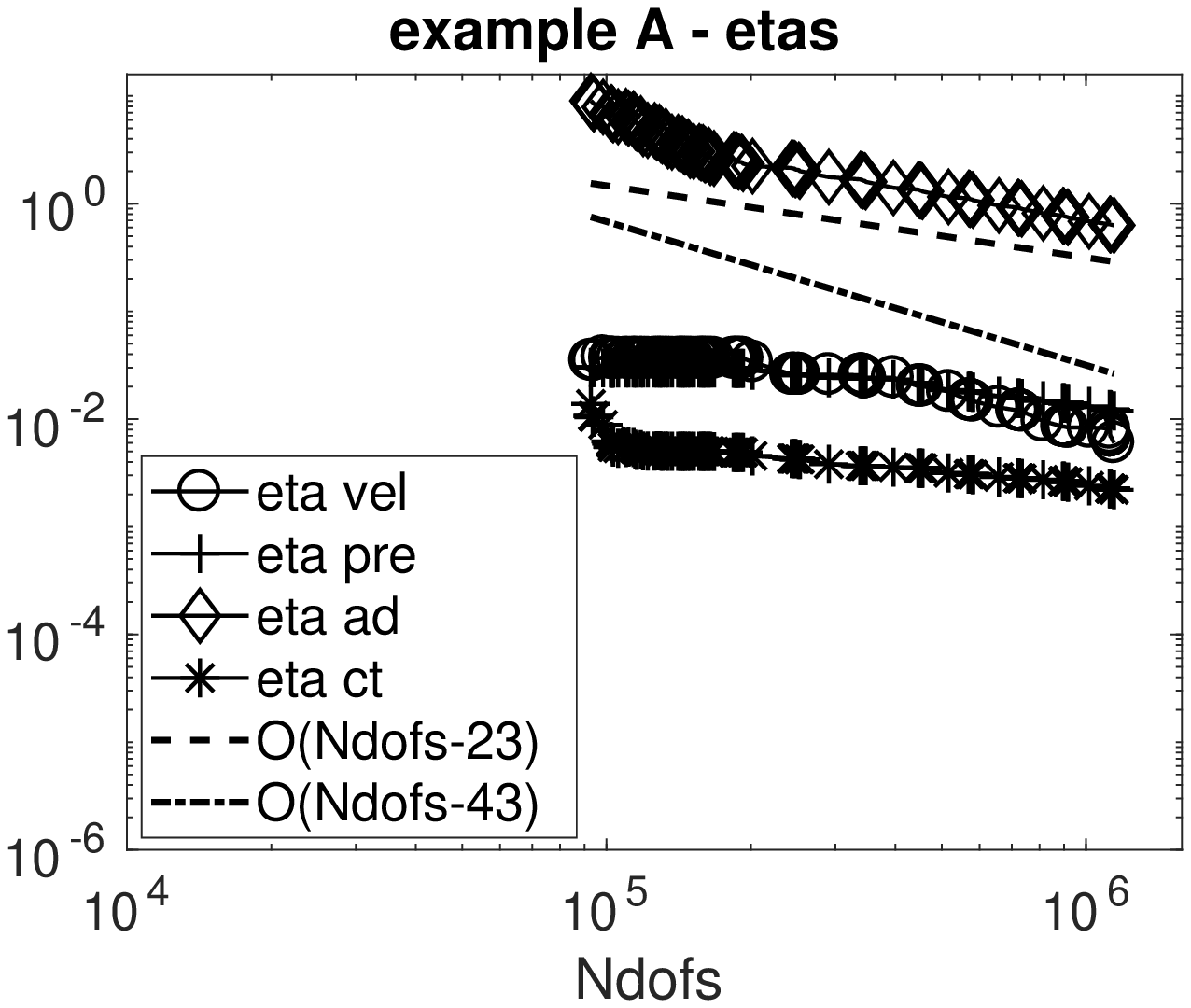}\\
\qquad \tiny{(C)}
\end{minipage}
\\
\begin{minipage}{0.32\textwidth}\centering
\psfrag{example A - ie}{\hspace{-1.5cm} Effectivity index for $\alpha = 1.99$}
\psfrag{estim d error}{\LARGE{$\mathcal{E}_{ocp} / \|\mathbf{e}\|_{\Omega}$}}
\psfrag{Ndofs}{\large{$\textrm{Ndof}$}}
\includegraphics[trim={0 0 0 0},clip,width=4.4cm,height=3.9cm,scale=0.6]{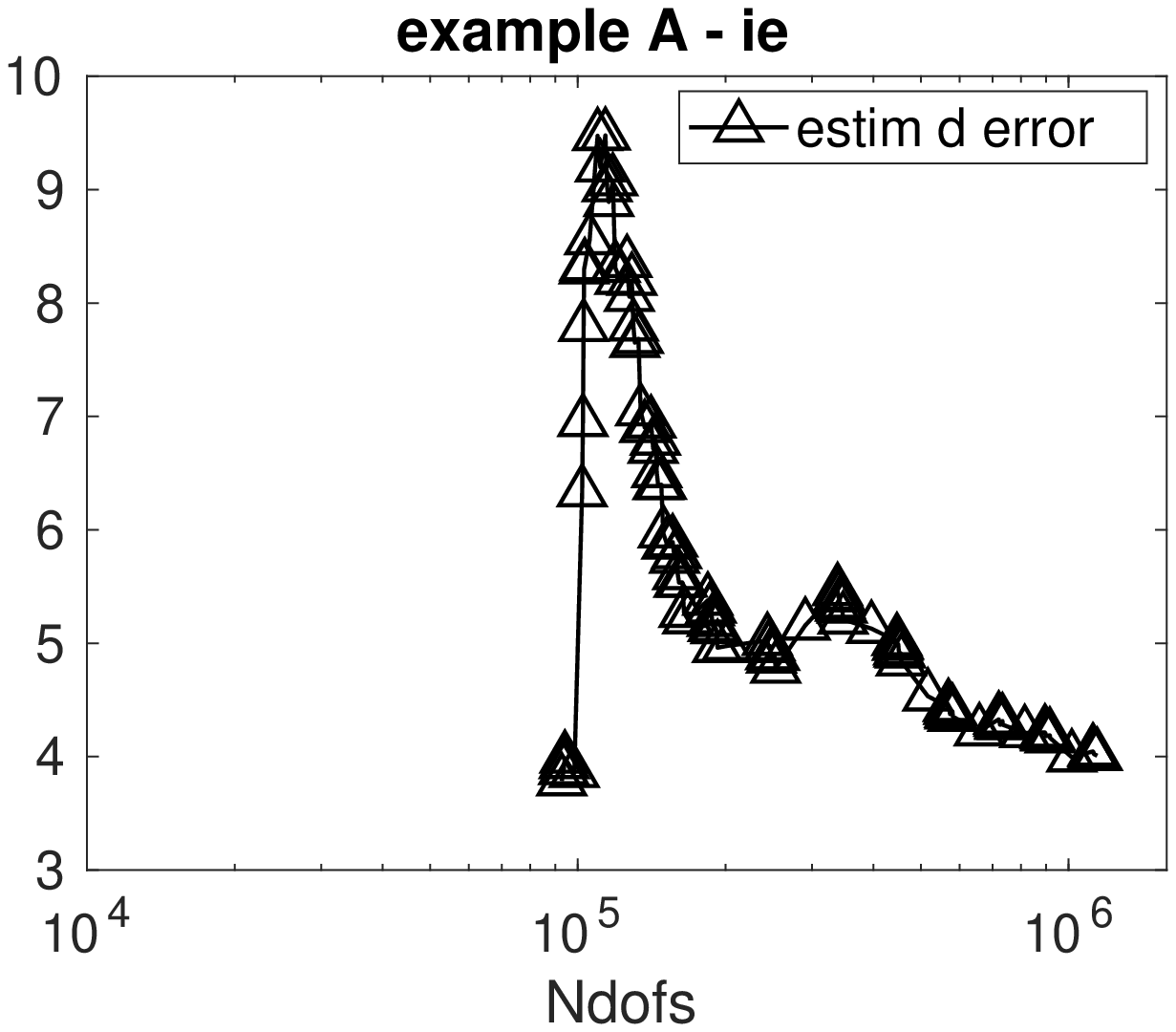}\\
\qquad \tiny{(D)}
\end{minipage}
\begin{minipage}{0.32\textwidth}\centering
\includegraphics[trim={0 0 0 0},clip,width=4.4cm,height=3.9cm,scale=0.6]{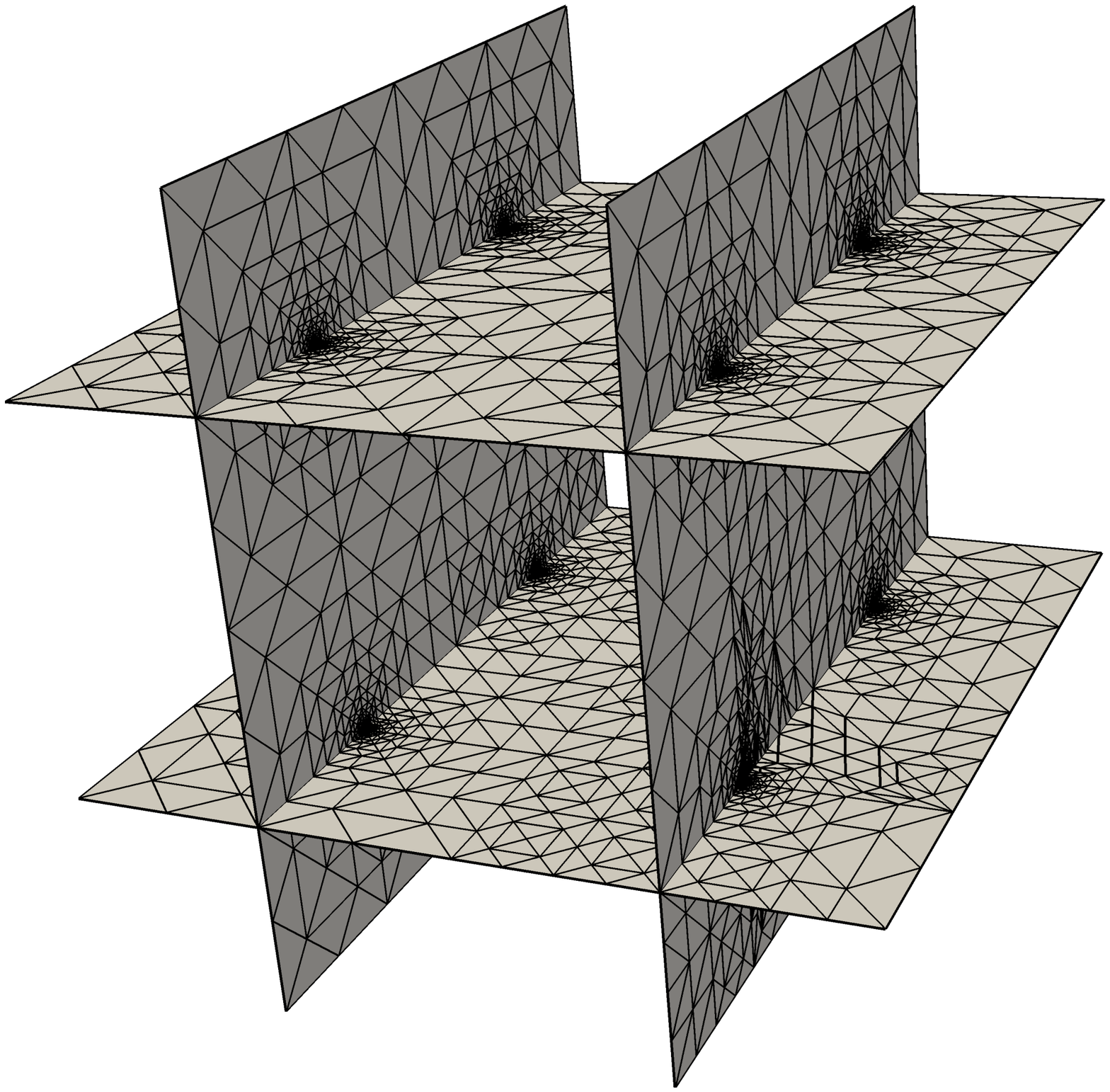}\\
\qquad \tiny{(E)}
\end{minipage}
\caption{Example 4: Experimental rates of convergence for the total error $\|\mathbf{e}\|_{\Omega}$ and the global estimator $\mathcal{E}_{ocp}$ (A), the individual contributions of $\|\mathbf{e}\|_{\Omega}$ (B) and $\mathcal{E}_{ocp}$ (C), the effectivity index (D), and slices of the 40th adaptively refined mesh (E). }
\label{fig:ex_4}
\end{figure}
         
\begin{figure}[ht]
\centering
\psfrag{error vel}{\LARGE{$\|\mathbf{e}_{\yy}\|_{\LL^{\infty}(\Omega)}$}}
\psfrag{error pre}{\LARGE{$\|e_{p}\|_{L^{2}(\Omega)}$}}
\psfrag{error vel ad}{\LARGE{$\|\nabla \mathbf{e}_{\zz}\|_{\LL^{2}(\rho,\Omega)}$}}
\psfrag{error pre ad}{\LARGE{$\|e_{r}\|_{L^{2}(\rho,\Omega)}$}}
\psfrag{error ct}{\LARGE{$\|\mathbf{e}_{\uu}\|_{\LL^{2}(\Omega)}$}}
\psfrag{e ocp}{\LARGE{$\|e\|_{\Omega}$}}
\psfrag{eta ad}{\LARGE{$\mathcal{E}_{ad}$}}
\psfrag{eta vel}{\LARGE{$\mathcal{E}_{st}$}}
\psfrag{eta pre}{\LARGE{$E_{st}$}}
\psfrag{eta ct}{\LARGE{$\mathcal{E}_{ct}$}}
\psfrag{eta ocp}{\LARGE{$\mathcal{E}_{ocp}$}}
\psfrag{O(Ndofs-23)}{\LARGE{$\textrm{Ndof}^{-2/3}$}}
\psfrag{O(Ndofs-43)}{\LARGE{$\textrm{Ndof}^{-4/3}$}}
\psfrag{example A - e}{\hspace{-0.5cm}\LARGE{$\mathcal{E}_{ocp}$ for $\alpha = 1.99$}}
\psfrag{example A - errors}{\hspace{-0.5cm}\LARGE{Error contributions for $\alpha = 1.99$}}
\psfrag{example A - etas}{\hspace{-2.0cm}\LARGE{Estimator contributions for $\alpha = 1.99$}}
\psfrag{Ndofs}{\LARGE{$\textrm{Ndof}$}}
\begin{minipage}{0.32\textwidth}\centering
\includegraphics[trim={0 0 0 0},clip,width=4.4cm,height=3.9cm,scale=0.8]{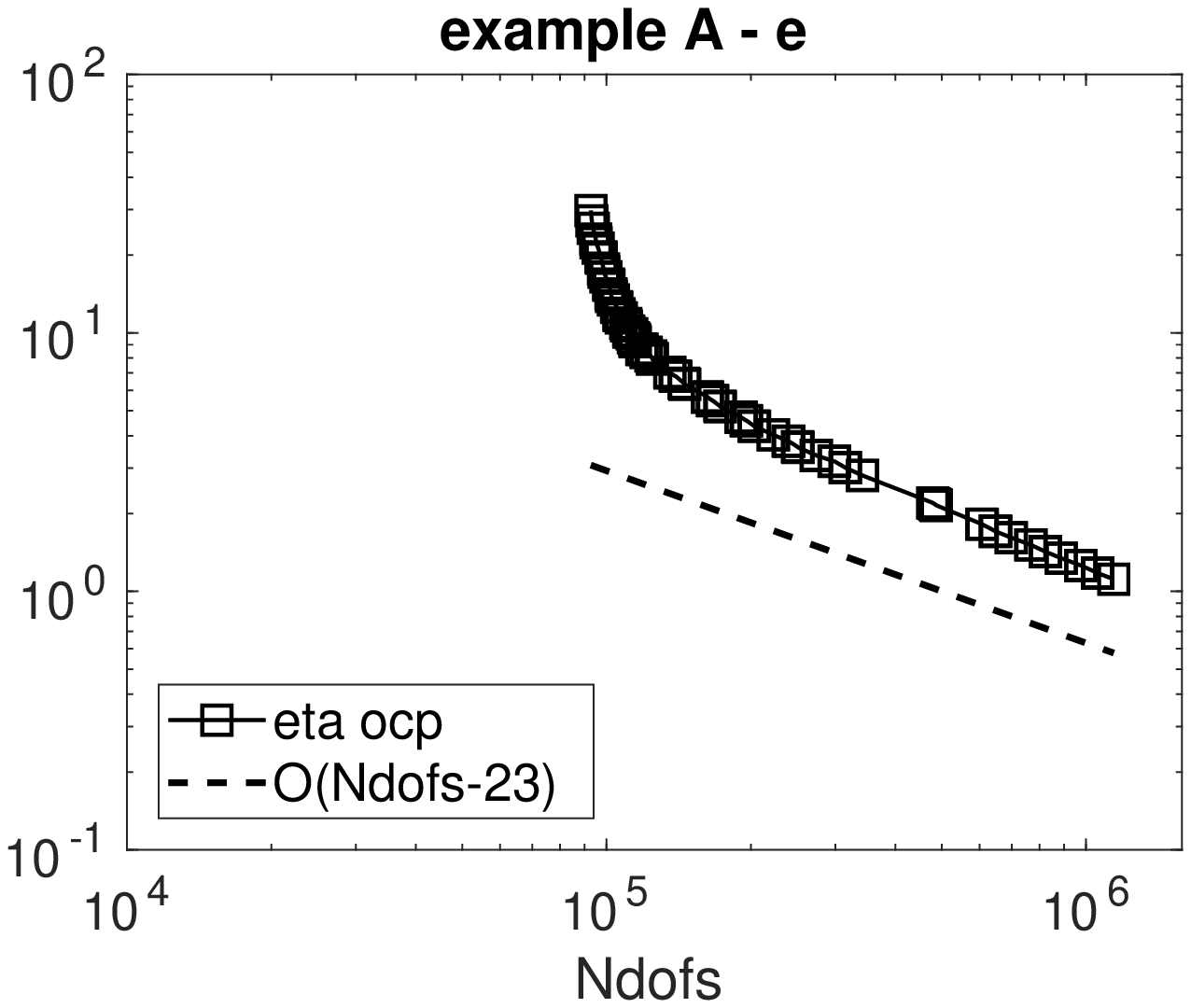}\\
\qquad \tiny{(A)}
\end{minipage}
\begin{minipage}{0.32\textwidth}\centering
\includegraphics[trim={0 0 0 0},clip,width=4.4cm,height=3.9cm,scale=0.8]{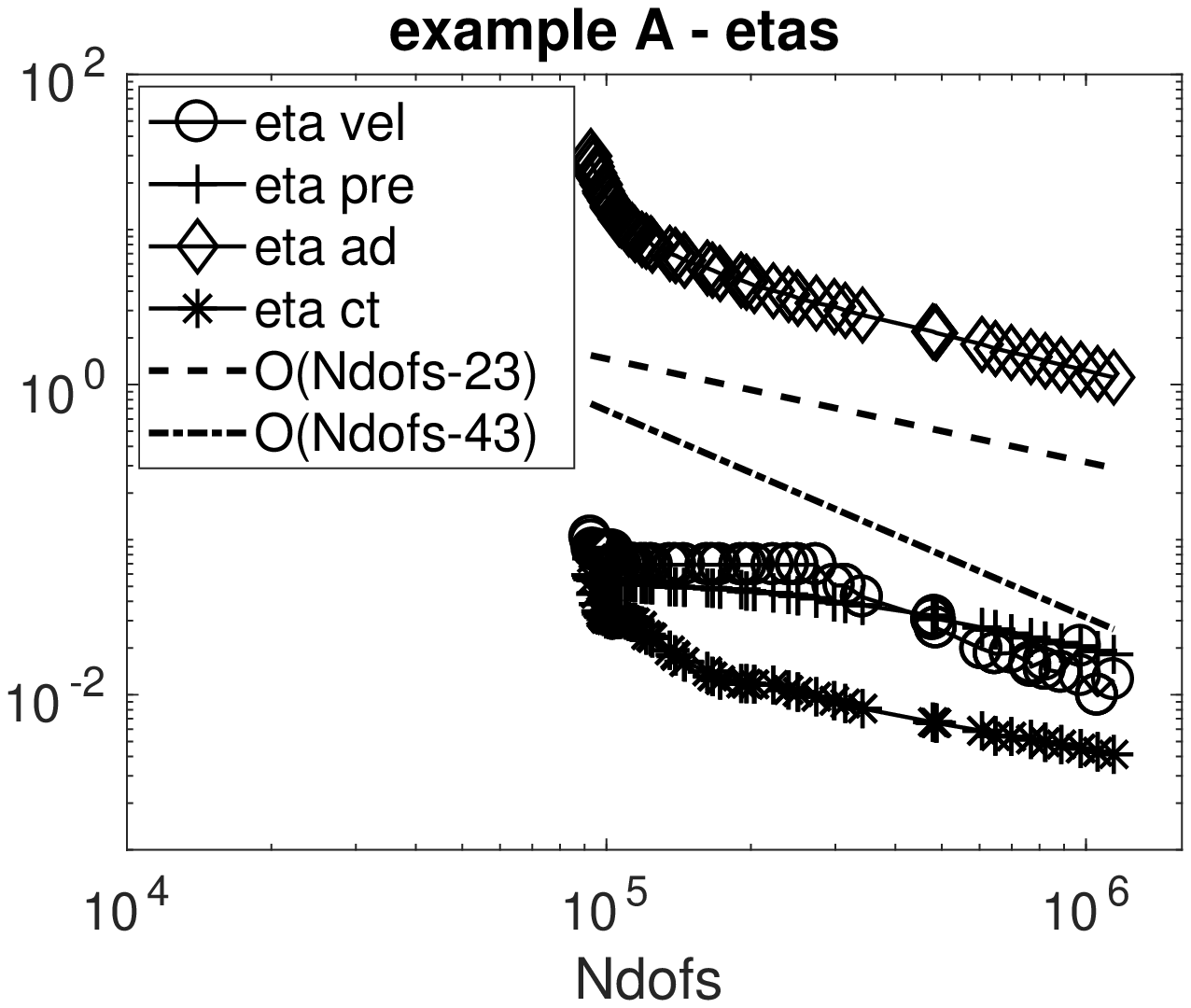}\\
\qquad \tiny{(B)}
\end{minipage}
\begin{minipage}{0.32\textwidth}\centering
\includegraphics[trim={0 0 0 0},clip,width=4.4cm,height=3.9cm,scale=0.8]{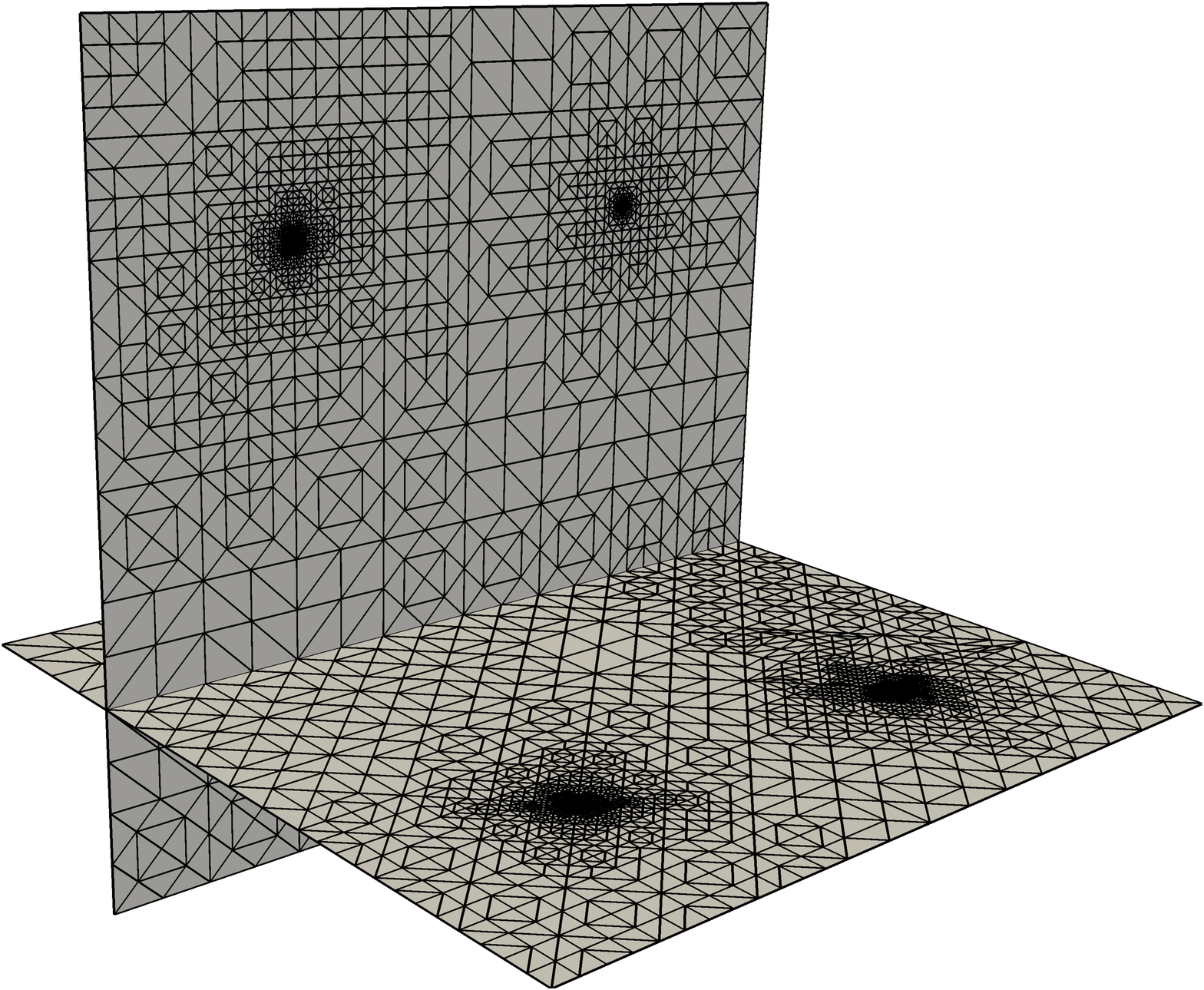}\\
\qquad \tiny{(C)}
\end{minipage}
\caption{Example 5: Experimental rates of convergence for the error estimator $\mathcal{E}_{ocp}$ (A), its individual contributions (B), and slices of the 87th adaptively refined mesh (C).}
\label{fig:ex_5}
\end{figure}
         
\subsection{Conclusions.}\label{sec:conclusions.}
         
In view of the presented numerical experiments we present the following conclusions.
       
\begin{itemize}
         
\item[$\bullet$] Most of the refinement occurs near the observation points. This attests to the efficiency of the devised estimators. When the domain involve geometric singularities, refinement is also being performed in regions that are close to them.

\item[$\bullet$] A larger value of $\alpha$ delivers the best results. Notice that, if $h_T < 1$, the larger the value of $\alpha$ then the smaller the value of $h_T^{\alpha + 2 -d}$.

\item[$\bullet$] We observe that, when the classical maximum strategy is used, the contributions $\|\mathbf{e}_{\uu}\|_{\LL^{2}(\Omega)}$ and $\mathcal{E}_{ct}$ do not exhibit an optimal decayment. This might be due to the fact $\mathcal{E}_{ct}$ is the smallest contribution of $\mathcal{E}_{ocp}$; a greater number of adaptive iterations is required for this contribution to be visible in $\mathcal{E}_{ocp}$. This deficiency can be improved, in two dimensions, by using the alternative marking criterion \eqref{def:new_marking}.

\item[$\bullet$] The contribution $\mathcal{E}_{ad}(\bar{\zz}_\T,\bar{r}_\T,\bar{\yy}_\T)$ of the global error estimator $\mathcal{E}_{ocp}$ is, most of the time, the dominating one.

\item[$\bullet$] In spite of the very singular nature of the problem that defines the adjoint variable, our proposed estimator is able to deliver optimal experimental rates of convergence, within an adaptive loop, for the contributions related to the discretization of the state and adjoint equations.

\end{itemize}

%


\bibliographystyle{siam}
\footnotesize
\bibliography{biblio}

\end{document}